\documentclass[11pt, a4paper,leqno]{amsart}
\usepackage{amsmath,amsthm,amscd,amssymb,amsfonts, amsbsy}
\usepackage{latexsym}
\usepackage{txfonts}
\usepackage{exscale}
\usepackage{bbm}
\usepackage{enumitem}
\usepackage{graphicx}
\usepackage[colorlinks,citecolor=red,linktoc=all,pagebackref,hypertexnames=false]{hyperref}
\usepackage{color}
\usepackage{xcolor}

\definecolor{DarkGreen}{rgb}{0.09, 0.45, 0.27}

\calclayout
\allowdisplaybreaks

\theoremstyle{plain}
\newtheorem{theorem}[equation]{Theorem}
\newtheorem{lemma}[equation]{Lemma}
\newtheorem{corollary}[equation]{Corollary}
\newtheorem{proposition}[equation]{Proposition}

\theoremstyle{definition}
\newtheorem{definition}[equation]{Definition}

\theoremstyle{remark}
\newtheorem{remark}[equation]{Remark}
\newtheorem{remarks}[equation]{Remarks}

\newtheorem{claim}[equation]{Claim}

\numberwithin{equation}{section}

\DeclareMathOperator{\curl}{curl}

\newcommand{\HH}{\mathcal{H}}

\newcommand{\RR}{{\mathbb{R}}}

\newcommand{\pO}{\partial\Omega}
\newcommand{\dist}{\operatorname{dist}}

\newcommand{\re}{\mathbb{R}}
\newcommand{\rn}{\mathbb{R}^n}
\newcommand{\co}{\mathbb{C}}

\newcommand{\ree}{\mathbb{R}^{n+1}}

\newcommand{\C}{\mathcal{C}}

\newcommand{\om}{\Omega}

\def\XXint#1#2#3{{\setbox0=\hbox{$#1{#2#3}{\int}$ }
\vcenter{\hbox{$#2#3$ }}\kern-.6\wd0}}

\newcommand{\A}{\mathcal{A}}

\newcommand{\E}{\mathcal{E}}

\newcommand{\s}{\mathcal{S}}

\newcommand{\oo}{\mathcal{O}}

\newcommand{\ntt}{\widetilde{N}_*}

\newcommand{\pom}{\partial\Omega}

\newcommand{\loc}{\operatorname{loc}}
\newcommand{\hm}{\omega}

\newcommand{\RNum}[1]{\uppercase\expandafter{\romannumeral #1\relax}}

\def\Yint#1{\mathchoice
    {\YYint\displaystyle\textstyle{#1}}%
    {\YYint\textstyle\scriptstyle{#1}}%
    {\YYint\scriptstyle\scriptscriptstyle{#1}}%
    {\YYint\scriptscriptstyle\scriptscriptstyle{#1}}%
      \!\iint}
\def\YYint#1#2#3{{\setbox0=\hbox{$#1{#2#3}{\iint}$}
    \vcenter{\hbox{$#2#3$}}\kern-.51\wd0}}
\def\longdash{{-}\mkern-3.5mu{--}} 
\def\tiltlongdash{\rotatebox[origin=c]{15}{$\longdash$}}
\def\fiint{\Yint\tiltlongdash}

\renewcommand{\emptyset}{\mbox{\textup{\O}}}

\DeclareMathOperator{\supp}{supp}
\DeclareMathOperator{\Gr}{Graph}
\DeclareMathOperator{\diam}{diam}
\DeclareMathOperator{\divg}{div}

\DeclareMathOperator{\Ir}{I}
\DeclareMathOperator{\IIr}{II}

\def\div{\mathop{\operatorname{div}}\nolimits}

\begin{document}
\allowdisplaybreaks
\author{S. Bortz}
\address{Department of Mathematics, University of Alabama, Tuscaloosa, AL, 35487, USA}
	\email{sbortz@ua.edu}

\author{T. Toro}
\address{Department of Mathematics, University of Washington, Seattle, WA
98195, USA}
\email{toro@uw.edu}

\author{Z. Zhao}
\address{Department of Mathematics, University of Chicago, Chicago, IL 60637, USA}
\email{zhaozh@uchicago.edu}

\title[H\"older Coefficients Implies $\log k \in VMO$]{Optimal Poisson Kernel Regularity for Elliptic Operators with H\"older Continuous Coefficients in Vanishing Chord-Arc Domains}

\begin{abstract}
We show that if $\Omega$ is a vanishing chord-arc domain and $L$ is a divergence-form elliptic operator with H\"older-continuous coefficient matrix, then $\log k_L \in VMO$, where $k_L$ is the elliptic Poisson kernel for $L$ in the domain $\Omega$. This extends the previous work of Kenig and Toro in the case of the Laplacian.

\end{abstract}
\thanks{T.T.  was partially supported by the Craig McKibben \& Sarah Merner Professor in Mathematics, by NSF grants DMS-1664867 and DMS-1954545, and by the Simons Foundation Fellowship 614610. Z.Z. was partially supported by the Institute for Advanced Study and by NSF grants DMS-1664867 and DMS-1902756.}


\maketitle
\tableofcontents

\section{Introduction}

In this article we study quantitative, asymptotic regularity of the (elliptic-) Poisson kernel for second order divergence form elliptic operators of the form $L = - \div A(X) \nabla$ in (bounded) rough domains, where the coefficient matrix is assumed to be H\"older continuous. This extends the work of Kenig and Toro \cite{KT-Duke} from the case of $L = -\Delta$, the Laplacian, to this natural class of variable coefficient operators. One may wish to interpret the results here as asymptotic optimality of the solution map for the linear operator $L$, with some extra consideration. Indeed, the solvability of the $L^p$-Dirichlet problem, with accompanying non-tangential estimates, is equivalent to the Poisson kernel satisfying a $L^{p'}$-reverse H\"older condition ($p' = p/(p-1)$). Our result here is equivalent to this condition being satisfied for all $p > 1$ and that for fixed $p$ the constant in the reverse H\"older inequality tends to the optimal value, $1$, when the balls shrink. In fact, the optimality (in the limit) of this constant for any fixed $p > 1$ implies the optimality for all $p > 1$ and $\log k \in VMO$ (see \cite{Korey} for a detailed discussion\footnote{This remarkable theory of `self-improvement' comes from the study of quasiconformal mappings \cite{Gehring, Iwan-Geh}.}). Here our geometric assumptions on the domain are optimal\cite{KT-Lens, AMT-onep}, that is, the domains are chord arc with vanishing constant (see Definitions \ref{deltaVCAD.def} and \ref{vcad-tt}). These domains can be described as having asymptotic flatness (in the sense of Reifenberg \cite{Rei}) coupled with surface measure which behaves asymptotically like Lebesgue measure.

Throughout, the ambient space is $\ree$, $n \ge 2$ and we often make the identification $\ree = \{(x,t) \in \rn \times \re\}$. We work with divergence form elliptic second order differential equations of the form $L = -\div A(X) \nabla$, where the real, $(n+1) \times (n+1)$ matrix-valued function $A$ satisfies the $\Lambda$-ellipticity condition for some $\Lambda \ge 1$, that is, $\|A\|_{L^\infty} \le \Lambda$ and for almost every $X \in \ree$
\[\Lambda^{-1} |\xi|^2 \le \langle A(X)\xi, \xi\rangle, \quad \forall \xi \in \mathbb{R}^{n +1}.\]

Our main result is the following.
\begin{theorem}\label{VCADcorr.thm}
Let $\om \subset \ree$ be a bounded vanishing chord arc domain (see Definitions \ref{deltaVCAD.def} and \ref{vcad-tt}) and $L = -\div A(X) \nabla$ be an elliptic operator, with (real, $\Lambda$-elliptic) coefficients satisfying the H\"older condition 
\begin{equation}\label{cond:Holder}
	|A(X) - A(Y)| \le C_A|X-Y|^\alpha, \quad \forall X, Y \in \ree
\end{equation}
for some $C_A > 0$ and $\alpha \in (0, 1]$. Then $\log k \in VMO$ (see Definition \ref{BVMO.def}), where $k$ is the (elliptic-)Poisson kernel for $L$ on the domain $\om$, that is, $k := \frac{d\hm^{X_0}}{d\sigma}$ for some $X_0 \in \om$. Here $\hm^{X_0}$ is the elliptic measure for $L$ with pole at $X_0$, and $\sigma= \mathcal{H}^n|_{\pom}$ is the surface measure for $\om$. 
\end{theorem}

Let us lay out the structure of the proof of Theorem \ref{VCADcorr.thm}. The restriction to real coefficients is necessary to define the elliptic measure (via the maximum principle and the Riesz representation theorem) as the solution map for the Dirichlet problem for $L$ on the domain $\om$. (Note that chord arc domains are Wiener regular and therefore the solution to the $L$-Dirichlet problem for some $f \in C_c(\pom)$ is 
 $u_f(X_0) = \int_{\pom} f(y) \, d\hm^{X_0}(y)$.) On the other hand, some of the solvability results which serve as the starting point for our analysis rely on (complex) analytic perturbation theory \cite{AAAHK, AAH, AAMc}. Indeed, our study begins with operators on the upper half space, where there is much known for both transversally independent complex $L^\infty$ perturbations of transversally independent operators with constant coefficients and operators which are a Dahlberg-Fefferman-Kenig-Pipher \cite{Dahlbergperturb, FKP} type\footnote{The perturbations are a quantitative, `averaged' refinement of those in \cite{FJK}.}  (`transversal') perturbations. For the former we prefer to cite the treatment in \cite{AAH} and for the latter \cite{AA}. These perturbations are maintained under pull-back on small Lipschitz graph domains and due to the H\"older condition we may view our operator (even after pull-back) as a two-fold perturbation of a constant coefficient matrix. Thus, to conclude Theorem \ref{VCADcorr.thm}, we employ `good' approximation schemes developed in \cite{DJ, Semmes-DJwNTA} and  \cite{KT-Duke, Semmes-SmallCAD1}, whereby one approximates a vanishing chord arc domain by domains with small Lipschitz constant and makes the delicate estimates required to show $\log k \in VMO$. 
Here the former approximation in \cite{DJ, Semmes-DJwNTA} allows us to establish rough `$A_\infty$ estimates' and the later approximation in \cite{KT-Duke, Semmes-SmallCAD1} allows us to establish the refined, asymptotic estimates of Theorem \ref{VCADcorr.thm} (the `rough' $A_\infty$ estimates are needed to control some errors).

This paper here brings together tools from partial differential equations, harmonic analysis and geometric measure theory developed over the last 40 years. We attempt (perhaps in vain) to give a reasonable account of the relevant results to the current work. For the harmonic measure, in 1976, Dahlberg \cite{DahlbergRH2} showed that in a Lipschitz domain the Poisson kernel satisfies an $L^2$-reverse H\"older condition, which, as we mentioned above, implies the $L^2$ solvability of the Dirichlet problem. This sparkle a deep interest in the study of elliptic operators in rough sets that has persisted for decades. In 1982, Jerison and Kenig \cite{JK-VMO} showed that in a bounded
$C^1$ domain $\log k \in VMO$. In 1997, Kenig and Toro \cite{KT-Duke} extended the work of Jerison and Kenig to vanishing chord arc domains, that is, Theorem \ref{VCADcorr.thm} with $L = -\Delta$, by using a version of the Semmes decomposition \cite{Semmes-SmallCAD1}.

In this context, for the variable coefficient case there is a good perturbative theory, see \cite{Esc, FKP, MPT}. One can extrapolate optimal Poisson kernel regularity ($\log k \in VMO$) from one operator to another, whenever the discrepancy between the operators is a vanishing perturbation of Dahlberg-Fefferman-Kenig-Pipher type. On the other hand, aside from constant coefficient operators and their (vanishing) perturbations, there appears to be a lack of understanding of the properties that characterize an operator for which
 $\log k \in VMO$. This paper provides an important example of a natural class of operators for which this is the case.
 It also addresses a gap in \cite{MT}. There the authors used the fact that for a uniformly elliptic operator with H\"older coefficients on Lipschitz domain with small Lipschitz constant 
 $\log k$ has small $BMO$ norm. This result had not been established. In fact attempts to fix this gap by standard methods were unsuccessful. A completely new idea is required to prove this fact, and its introduction is one of the major original contributions of the current paper. There were also small gaps and errors in some `localization estimates' in \cite{MT}, so we carefully reprove these results. For the most part the techniques in these `localization estimates' follow \cite{MT, KT-Duke}, but the proof of Theorem \ref{VCADcorr.thm} is completely different and require a number of new ingredients.

As mentioned above, we leverage powerful, refined theorems in the study of elliptic boundary value problems. We are particularly reliant on the theory built from layer potentials and the operational calculus of first order Dirac operators associated to divergence form elliptic operators with variable coefficients. The modern treatment of these objects was shaped by Auscher, Axelsson, Hofmann, McIntosh and many others \cite{AAAHK, AAH, AAMc, AHLMcT}. We refer the reader to \cite{AAAHK} for a relatively comprehensive history, but we remark that these works grew out of the testing conditions (`T1/Tb theory') for singular integrals and Littlewood-Paley type operators \cite{CoifmanMeyer, ChristJourne, DavidJourne, DavidJourneSemmes, McIntoshMeyer}, and the generalizations of this theory. Perhaps the most notable such generalization led to the resolution of the Kato conjecture \cite{AHLMcT} and served as the basis for the results in \cite{AAAHK, AAH, AAMc}. These works allow us to treat the $L^\infty$-perturbation, while we use the results in \cite{AA} to handle the Dahlberg-Fefferman-Kenig-Pipher type perturbation (the key is that we can locally write our operator as a two-fold perturbation). 
\smallskip

\textbf{Acknowledgement}: The first named author would like to thank Pascal Auscher, Moritz Egert and Steve Hofmann for helpful conversations concerning the first-order method and the perturbative theory for elliptic boundary value problems.

\section{Preliminaries}
First we introduce notation that will be standard throughout. For notation specific to chord arc domains and their variants with small constants, see the next subsections (Sections \ref{CADs.sect} and \ref{sec:sCAD}). 
Throughout the paper, by allowable constants we always mean the dimension $n\geq 2$, the ellipticity constant $\Lambda\geq 1$ and the H\"older constants $C_A>0$ and $\alpha \in (0,1]$.

\subsection{Notation}\label{sect:notation}
\begin{itemize}
\item Given a domain $\om \subset \ree$ with boundary $\pom$, for $x \in \pom$ and $r \in (0, \diam \pom)$ we let $\Delta(x,r) := B(x,r) \cap \pom$ denote the {\bf surface ball} of radius $r$ centered at $x$. We make clear which surface measure we are using any time there is possible ambiguity (e.g. when dealing with multiple domains simultaneously). 
\item Given $x = (x_1, \dots x_{n+1}) \in \ree$ (or $x \in \rn$ resp.) we define $|x|_\infty = \sup_{i}\{|x_i|\}$ to be the $\ell^\infty$ norm of $x$. Similarly, we let $|x| := |x|_2$ be the standard Euclidean ($\ell^2$) distance.
\item When working with the upper half space ($\ree_+ := \{(x,t) \in \rn \times \re: t > 0\}$) we use the following notation. Let $y \in \rn = \rn \times \{0\}$ and $r > 0$ then we define:
\begin{itemize}
\item The cube $Q(y,r) := \{x \in \rn \times \{0\}: |x-y|_\infty < r\}$, with side length $2r$ and the (surface) ball $\Delta(y,r) := \{x \in \rn \times \{0\}: |x-y|< r \}$.
\item Given an $n$ dimensional cube $Q = Q(y,r)$ we let $\ell(Q) := 2r$ be the side length of the cube.
\item Given an $n$-dimensional cube $Q$ we let $R_Q$ be the Carleson box relative to $Q$, that is, $R_Q = Q \times (0,\ell(Q))$.
\item The Whitney regions $W(y,r) = \Delta(y,r/2) \times (r/2, 3r/2)$ and 
\[D(y,r) : = B(y,10r) \cap \{(x,t): x \in \rn, t > r/2\}.\]
\end{itemize}
\item Given a (real) divergence form elliptic operator $L = -\div A \nabla$ we define its transpose (and, in this case, adjoint) $L^T :=  -\div A^T \nabla$, where $A^T$ is the transpose of $A$, that is $(A^T)_{i,j} = (A)_{j,i}$. 
\end{itemize}

\begin{definition}[Lipschitz domains]\label{lipdomains.dom}
We say a domain (connected open set) $\om \subset \ree$ is a {\bf $\gamma$-Lipschitz domain} if for every $x \in \pom$ there exists $r >0$ and an isometric coordinate system with origin $x = \oo$ such that 
\begin{equation}\label{lipdomainregion.eq}
\{Y \in \mathbb{R}^{n+1} : |Y - x|_\infty < r\} \cap \om = \{Y \in \mathbb{R}^{n+1}: |Y - x|_\infty < r\} \cap \{(y,t) : y \in \rn, t > \varphi(y)\}
\end{equation}
for some Lipschitz function $\varphi: \rn \to \re$ with $\varphi(\oo) = \oo$ and $\|\nabla \varphi\|_\infty \le \gamma$. We say a domain is a Lipschitz domain if it is a $\gamma$-Lipschitz domain for some $\gamma \ge 0$. 
We call a domain $\om \subset \ree$ of the form 
\[ \om := \{(y,t) : y \in \rn, t > \varphi(y)\}\]
for some Lipschitz function $\varphi: \rn \to \re$ with $\|\nabla \varphi\|_\infty < \infty$ a Lipschitz graph domain (and if $\|\nabla \varphi\|_\infty \le \gamma$, a $\gamma$-Lipschitz graph domain). 
\end{definition}

\begin{remark}\label{charts.rmk}The number of `charts' needed to cover the boundary in the definition of Lipschitz domain is often important, but if $\pom$ is bounded the compactness of the boundary ensures that only a finite number of charts is required. If $\om$ is a Lipschitz graph domain only one chart is needed. When we work with Lipschitz domains the number of charts will always be {\it uniformly} bounded.
\end{remark}

\begin{definition}[FKP-Carleson Norm]\label{Carnorm.def}
Given any matrix $B = B(x,t)$ defined on $\ree_+ := \{(x,t): x \in \rn, t > 0\}$ we define the Carleson norm of $B$ as
\[\|B\|_{\C}: = \sup_{Q}\left(\frac{1}{|Q|} \iint_{R_Q}\|B\|^2_{L^\infty(W(x,t))}(x,t)\,  \frac{dx\, dt}{t} \right)^{1/2},\]
where the supremum is taken over all cubes $Q \subset \rn$, $R_Q$ is the Carleson box $Q \times (0,\ell(Q))$  and we recall $W(x,t) = \Delta(x,t/2) \times (t/2, 3t/2)$.
\end{definition}

\begin{definition}[Nontangential Maximal function]\label{ntmax.def}
Given any locally $L^2$-integrable function $F:\ree_+ \to \re$ we define the ($L^2$-modified) non-tangential maximal function $\ntt F: \rn \to \re$ as
\begin{equation}\label{def:NF}
	\ntt F(x): = \sup_{t > 0} \left(\fint_{W(x,t)} |F(y,s)|^2 \, dy \, ds\right)^{1/2}.
\end{equation} 
For $p > 1$ we also define  the $L^p$-modified non-tangential maximal function $\ntt^p F: \rn \to \re$ as
\[\ntt^p (F(x)): = \sup_{t > 0} \left(\fint_{W(x,t)} |F(y,s)|^p \, dy \, ds\right)^{1/p}.\]
\end{definition}
\begin{remark}
	We use the notation $\widetilde{N}_*F$ to distinguish it with the standard non-tangential maximal function defined using $L^\infty$ norm. 	
\end{remark}

\subsection{PDE estimates in chord arc domains}\label{CADs.sect}
In this subsection we define chord arc domains and state without proof some well-known results about solutions to elliptic operators as well as elliptic measures on such domains. These two non-tangential maximal functions are equivalent if Moser's estimate holds.

Chord arc domains are domains with scale-invariant connectivity (Harnack chains), interior and exterior openness (corkscrews) and whose boundaries are quantitatively ``$n$-dimensional" (Ahlfors regular boundary). 

\begin{definition}[Two-sided Corkscrew condition {\cite{JK-NTA}}]\label{tscs.def}
We say a domain (an open and connected set) $\om \subset \ree$ satisfies the two-sided corkscrew condition if there exists a uniform constant $M > 2$ such that for all $x \in \pom$ and $r \in (0, \diam \pom)$ there exists $X_1,X_2 \in \ree$ such that
\[B(X_1, r/M) \subset B(x,r) \cap \om, \quad B(X_2, r/M) \subset B(x,r) \setminus \overline{\om}.\]
In the sequel, we write $\A(x,r) := X_1$, for the interior corkscrew point for $x$ at scale $r$. 
\end{definition}

\begin{definition}[Harnack chain condition {\cite{JK-NTA}} ]
We say a domain $\om \subset \ree$ satisfies the Harnack chain condition if there exists a uniform constant $M \ge 2$ such that if $X_1, X_2 \in \om$ with $\dist(X_i, \pom) > \epsilon > 0$ and $|X_1 - X_2|< 2^k \epsilon$ then there exists a `chain' of open balls $B_1, \dots, B_N$ with $N < M k$ such that $X_1 \in B_1$, $X_2 \in B_N$, $B_j \cap B_{j+1} \neq \emptyset$ for $j = 1, \dots, N-1$ and $M^{-1} \diam B_j \le \dist(B_j, \pom) \le M \diam B_j$ for $j = 1, \dots, N$.
\end{definition}

\begin{definition}[Ahlfors regular]
We say a set $E \subset \ree$ is Ahlfors regular if $E$ is closed and there exists a uniform constant $C$  such that
\[C^{-1}r^n \le H^n(B(x,r) \cap E) \le Cr^n, \quad \forall x \in E, \forall  r\in (0, \diam E).\]
\end{definition}

\begin{definition}[NTA and chord arc domains {\cite{JK-NTA}}]
We say a domain $\om \subset \ree$ is and NTA domain if it satisfies the two-sided corkscrew condition and the Harnack chain condition. We say a domain $\om \subset \ree$ is a {\bf chord arc domain} if it is an NTA domain and $\pom$ is Ahlfors regular. We refer to the constants $M$ and $C$ in the definitions of the two-sided corkscrew condition, Harnack chain condition and the Ahlfors regularity condition as the `chord arc constants'. 
\end{definition}

\begin{remark}Every bounded ($\gamma$-)Lipschitz domain is a chord arc domain as are Lipschitz graph domains. See Remark \ref{charts.rmk}.\end{remark}

Now we give several results on the behavior of solutions to real divergence form elliptic equations in chord arc domains. We note that while the original results are stated for harmonic functions their proofs carry over for real divergence form elliptic equations as the primary tools (e.g. the Harnack inequality and H\"older continuity at the boundary \cite{HKM-Book}) are still available and only introduce dependence on the ellipticity parameter.
We also remark that these estimates are suitably local. For instance, if $\om \subset \ree$ is a $\gamma$-Lipschitz domain and we work `well-inside'\footnote{Here this means, $\{Y : |Y - x|_\infty \ll r\} \cap \om\}$ in Definition \ref{lipdomains.dom}.} a region as in  \eqref{lipdomainregion.eq} then the estimates on the boundary behavior of solutions depend on dimension, ellipticity and the parameter $\gamma$.   We refer the reader to \cite{CFMS, KKoPT, JK-NTA}. In the remainder of this section, $L = -\div A \nabla$ is a second order divergence form elliptic operator with real, $\Lambda$-elliptic coefficients. 

\begin{lemma}[Carleson estimate \cite{JK-NTA}]\label{cscarlesonest.lem}
Let $\om \subset \ree$ be a chord arc domain, $x \in \pom$ and $10r \in (0,\diam \pom)$. If $Lu = 0$, $u \ge 0$ in $\om \cap B(x,2r)$ and $u$ vanishes continuously on $B(x, 2r) \cap \pom$ then
\[u(Y) \le C u(\A(x,r)), \quad \forall Y \in B(x,r) \cap \om.\]
Here $C > 0$ depends on $n$, $\Lambda$ and the chord arc constants for $\om$.
\end{lemma}

\begin{lemma}[H\"older continuity at the boundary \cite{JK-NTA,HKM-Book}]\label{HCatbdry.lem}
Let $\om \subset \ree$ be a chord arc domain, $x \in \pom$ and $10r \in (0,\diam \pom)$. If $Lu = 0$, $u \ge 0$ in $\om \cap B(x,4r)$ and $u$ vanishes continuously on $B(x, 4r) \cap \pom$ then
\[u(Y) \le C\left(\frac{|Y - x|}{r} \right)^\mu \sup\{u(Z): Z \in \partial B(x,2r)\cap \om \} \le C\left(\frac{|Y - x|}{r} \right)^\mu u(\A(x,r)), \]
for all $Y \in B(x,r) \cap \om$.
Here $C > 0$ and $\mu \in (0, 1)$ depend on $n$, $\Lambda$ and the chord arc constants for $\om$.
\end{lemma}

A simple consequence of this lemma is the following, which is sometimes referred to as Bourgain's estimate.

\begin{lemma}[Bourgain's estimate \cite{Bourgain, HKM-Book}]\label{Bourgain.lem}
Let $\om \subset \ree$ be a chord arc domain, $x \in \pom$ and $10r \in (0,\diam \pom)$. Then
\[\hm^{\A(x,r)}(B(x,r)) \gtrsim 1,\]
where the implicit constants depend on $n$, $\Lambda$ and the chord arc constants for $\om$. Here $\hm^X$ is the elliptic measure for the operator $L$ on $\om$ with pole at $X$. In particular, by the Harnack inequality, for $x \in \pom$ and $10r < 10R \in (0,\diam \pom)$
\[\hm^{\A(x,R)}(B(x,r)) \gtrsim 1, \]
where the implicit constants depend on $n$, $\Lambda$ and the chord arc constants for $\om$ and the ratio $R/r$.
\end{lemma}

\begin{lemma}[CFMS estimate \cite{CFMS}]\label{CFMS.lem}
Let $\om \subset \ree$ be a chord arc domain, $x \in \pom$ and $10r \in (0,\diam \pom)$. If $X_0 \in \om \setminus B(x,4r)$ then 
\[\frac{\hm^{X_0}(B(x,r))}{r^{n-1}G(X_0, \A(x,r))} \approx 1,\]
where the implicit constants depend on $n$, $\Lambda$ and the chord arc constants for $\om$. Here $\hm^X$ is the elliptic measure for the operator $L$ on $\om$ with pole at $X$ and $G(X,Y)$ is the $L$-Green function for $\om$ with pole at $Y$.
\end{lemma}

\begin{lemma}[Doubling for elliptic measure]\label{hmdoubling.lem}
Let $\om \subset \ree$ be a chord arc domain, $x \in \pom$ and $10r \in (0,\diam \pom)$. If $X_0 \in \om \setminus B(x,4r)$
\[\hm^{X_0}(B(x,2r)) \le C\hm^{X_0}(B(x,r)),\]
where $C$ depends on $n$, $\Lambda$ and the chord arc constants for $\om$.
\end{lemma}

\begin{lemma}[Comparison principle]\label{comparisonprinc.lem}
Let $\om \subset \ree$ be a chord arc domain, $x \in \pom$ and $10r \in (0,\diam \pom)$. If $Lu =Lv = 0$, $u, v \ge 0$ in $\om \cap B(x,2r)$, $u$ and $v$ are non-trivial functions which vanish continuously on $B(x, 2r) \cap \pom$ then
\[ 
\frac{u(X)}{v(X)} \approx \frac{u(\A(x,r))}{v(\A(x,r))}, \quad \forall X \in B(x,r) \cap \om,
\]
where the implicit constants depend on $n$, $\Lambda$ and the chord arc constants for $\om$. 
\end{lemma}

\begin{lemma}[Quotients of non-negative solutions]\label{QHC.lem}
Let $\om \subset \ree$ be a chord arc domain, $x \in \pom$ and $r \in (0,\diam \pom/10)$. If $Lu =Lv = 0$, $u, v \ge 0$ in $\om \cap B(x,2r)$, and $u$ and $v$ vanish continuously on $B(x, 2r) \cap \pom$, then $u/v$ is H\"older continuous of order $\mu = \mu(n, \Lambda, chord \ arc \ constants)$ in $\overline{B(x,r) \cap \om}$. In particular, $\lim_{Y \to y} (u/v)(Y)$ exists\footnote{Here the limit is taken within $\om$.} and, moreover,
\begin{equation}\label{QHC.eq}
\left| \frac{u(X)}{v(X)} -  \frac{u(\A(x,r))}{v(\A(x,r))} \right| \le C\left(\frac{|X - x|}{r}\right)^\mu\frac{u(\A(x,r))}{v(\A(x,r))}, \quad \forall X, Y \in B(x,r) \cap \om,
\end{equation}
where the constant $C >0$ and $\mu \in (0,1)$ depend on $n$, $\Lambda$ and the chord arc constants for $\om$.
\end{lemma}

Next we define the kernel function. It can be more generally defined (for any $X_0, X_1 \in \om$), but we use it only in this specific manner.

\begin{lemma}[Kernel function]\label{Kernelcompare.lem}
Let $\om \subset \ree$ be a chord arc domain, $x \in \pom$ and $10r \in (0,\diam \pom)$. If $X_0 \in \om \setminus B(x,4r)$ and $X_1 \in B(x,2r) \setminus B(x,r)$ we define for $z \in \pom$
\[H(z) := H(X_0,X_1, z) : = \frac{d\hm^{X_0}}{d\hm^{X_1}}(z).\]
The kernel function has the estimate
\[|H(z)/H(y) - 1| \le C\zeta^{\mu}, \quad \forall z,y: |z -x|, |y-x| < \zeta r,\]
for all $\zeta \in (0, 1/2)$, where $C > 0$ and $\mu \in (0,1)$ depend on $n$, $\Lambda$ and the chord arc constants for $\om$.
\end{lemma}

Later, we will `localize' the coefficients of our operator, the following lemma is useful in this regard.
\begin{lemma}\label{agreeingops.lem}
Let $\om_i \subset \ree$, $i = 1,2$ be chord arc domains such that
\[ \Omega_1 \cap B(x,10r) = \Omega_2 \cap B(x,10r) \]
where $x \in \pom_i$ and $r \in (0,\diam \pom_i/20)$. 
Suppose further that $L_i = -\div A_i \nabla$ are two divergence form $\Lambda$-elliptic operators with $A_1 = A_2$ on $B(x,10r)$. Let $\hm_i^{X_0}$ is the $L_i$-elliptic measure for $\om_i$ with pole at $X_0\in \Omega_i \cap B(x,8r) \setminus B(x,4r)$.\footnote{This means the statements on $x, r$ an $X_0$ are to hold simultaneously for $i = 1, 2$}
Then $\hm_1^{X_0}|_{B(x,r)}$ and $\hm_2^{X_0}|_{B(x,r)}$ are mutually absolutely continuous. In particular, if $\hm_1^{X_0}|_{B(x,r)}$ and $\HH^n|_{\partial\Omega_1 \cap B(x,r)} $ are mutually absolutely continuous then so are $\hm_2^{X_0}|_{B(x,r)}$ and $\HH^n|_{\partial\Omega_2 \cap B(x,r)} $.

Moreover,
\[\frac{d\hm_1^{X_0}}{d\hm_2^{X_0}}(y) \approx 1, \quad \hm_2^{X_0}-a.e.\ y \in B(x,r/2)\cap\pom,\]
where the implicit constants depend on $n$, $\Lambda$ and the chord arc constants for $\om_i$, $i = 1, 2$.

\end{lemma}
\begin{proof}[Sketch of the proof] Let $y \in B(x,r)$ and $s \in (0, r)$. Then the CFMS estimate (Lemma \ref{CFMS.lem}) implies that
\begin{equation}\label{tmp:hmGreen}
	\frac{\hm^{X_0}_1(B(y,s))}{\hm^{X_0}_2(B(y,s))} \approx \frac{G_1(X_0, \A(y,s))}{G_2(X_0, \A(y,s))},
\end{equation} 
where $G_i(X, Y)$ is the $L_i$-Green function for $\om_i$. 
On the other hand, Bourgain's estimate (Lemma \ref{Bourgain.lem}) and the CFMS estimate imply that 
\[G_i(X_0, \A(x,r/50))r^{n-1} \approx \hm^{X_0}_i(B(x,r/50) \approx 1, \quad i = 1, 2.\]
By \eqref{tmp:hmGreen} and the comparison principle\footnote{Here we need to view $G_i(X_0, Y)$ as a $L_i^T = -\div A_i^T \nabla$ null solution away from $X_0$. This follows by the fact that if $G(X,Y)$ is the $L$-Green function then $G_T(X,Y) = G(Y,X)$ is the $L^T$-Green function. We then apply the comparison principle with $L^T$ in place of $L$.}   (Lemma \ref{comparisonprinc.lem}) then shows 
\[\frac{\hm_1^{X_0}(B(y,s))}{\hm_2^{X_0}(B(y,s))} \approx \frac{G_1(X_0, \A(y,s))}{G_1(X_0, \A(y,s))} \approx \frac{G_1(X_0, \A(x,r/50))}{G_2(X_0, \A(x,r/50))} \approx 1,\]
where the constants depend on $n$, $\Lambda$ and the chord arc constants for $\om_i$, $i = 1, 2$.
From this estimate one can deduce all of the properties described in the lemma.
\end{proof}

\subsection{Chord arc domains with small constants}\label{sec:sCAD}
We study asymptotic flatness in the form of vanishing chord arc domains. To define these domains, we need to give a few more preliminary definitions.

\begin{definition}[Separation property]
Let $\om \subset \ree$ be a domain. We say that $\om$ has the {\bf separation property} if for each compact set $K \subset \ree$ there exists $R > 0$ such that for any $x_0 \in \pom \cap K$ and $r\in (0,R]$ there exists a vector $\vec{n} \in \mathbb{S}^n$ such that
\[\{X\in B(x_0, r): \langle X - x_0, \vec n \rangle  \ge r/4\}  \subset \om \]
and 
\[\{X\in B(x_0, r): \langle X - x_0, \vec n \rangle  \le  -r/4\} \subset \om^c. \]
Notice that $\vec n$ points `inward'. 
\end{definition}
\begin{remark}
	Provided that $\partial\Omega $	is $\delta$-Reifenberg flat for $\delta$ sufficiently small, there is an equivalent definition of the separation property, see \cite[Remark 1.1]{KT-Annals}.
\end{remark}

In the sequel, $D[\cdot; \cdot]$ will be used to denote the Hausdorff distance between two sets, that is, for $A, B \subset \ree$
\[D[A;B]: = \sup\{\dist(a, B): a \in A\} + \sup\{\dist(b,A): b \in B\}.\]

\begin{definition}[Reifenberg flatness] Given a closed set $\Sigma \in \ree$, $x_0 \in \Sigma$ and $r > 0$, we define
\[\Theta(x_0, r) := \inf_{P}\tfrac{1}{r}D[\Sigma \cap B(x_0, r);P \cap B(x_0, r)],\]
where the infimum is taken over all $n$-planes $P$ through $x_0$
and
\[\Theta(r) := \sup_{x_0 \in \Sigma} \Theta(x_0, r).\]
For $R > 0$ and $\delta \in (0, \delta_n]$, where $\delta_n$ is sufficiently small depending only on the dimension (see Remark \ref{flatvac.rmk}) we say $\Sigma$ is $(\delta, R)$-Reifenberg flat if 
\[\sup_{r \in (0, R)} \Theta(r) < \delta.\]
\end{definition}
\begin{remark}\label{flatvac.rmk}
The $\delta_n$ above is to ensure both that the definition is not vacuous and that if a domain has Reifenberg flat boundary and satisfies the separation property then it is, in fact, an NTA domain. See \cite[Lemma 3.1]{KT-Duke} and Appendix A in \cite{KT-Lens}.
\end{remark}

\begin{definition}[$(\delta, R)$-chord arc domains]\label{deltaVCAD.def}
Let $\om \subset \ree$ be a domain. Given $\delta \in (0, \delta_n]$ and $R > 0$, we say $\om$ is a {\bf $(\delta, R)$-chord arc domain} if 
\begin{enumerate}
\item $\om$ has the separation property (with parameter $R$),
\item $\pom$ is $(\delta, R)$-Reifenberg flat and
\item $\HH^n(B(x, r) \cap \pom) \le (1 + \delta)\omega_nr^n$, for all $x \in \pom$ and $r \in (0, R]$. 
\end{enumerate}
Here $\omega_n$ is the volume of the $n$-dimensional unit ball in $\rn$.

\end{definition}
\begin{remarks}\label{ADR.Deltarforallr.rmks} \ 
\begin{itemize}
	\item Since Hausdorff measure is non-increasing under projections, the $\delta$-Riefenberg flatness of the boundary in the above definition coupled with the upper bound on the surface measure gives the estimate
\[ (1 - \delta)\omega_nr^n \le \HH^n(B(x, r) \cap \pom)\le (1 + \delta)\omega_nr^n, \]
for any $x\in \pO$ and $r\in (0,R]$. See \cite[Remark 2.2]{KT-Duke}.
	\item There is an equivalent definition of $(\delta, R)$-chord arc domain, which is to replace the third assumption above by $\pO$ being Ahlfors regular (the Ahlfors regularity constant is not necessarily close to one) and the unit outer normal of $\Omega$ has BMO norm bounded above by $\delta$. See for example \cite[Definition 1.10]{KT-Lens}. These two definitions are equivalent (modulo constant) by \cite[Theorem 2.1]{KT-Duke} and \cite[Theorem 4.2]{KT-Annals} respectively. We choose the above definition in this paper for convenience.
\end{itemize}

\begin{definition}[Vanishing chord arc domains]\label{vcad-tt}
	We say $\om$ is a {\bf vanishing chord arc domain} 
\begin{itemize}
\item $\om$ is a $(\delta, R)$-chord arc domain for some $\delta \in (0, \delta_n]$ and $R > 0$, 
\item  $\limsup_{r \to 0^+} \Theta(r) = 0$ (where $\Sigma$ in the definition of $\Theta$ is $\pom$) and
\item $\lim_{r \to 0^+} \sup_{x \in \pom} \frac{\HH^n(B(x, r) \cap \pom)}{\hm_n r^n} = 1$.
\end{itemize}
\end{definition}
Note that if $\om$ is a vanishing chord arc domain then for every $\delta>0$ there exists $R_\delta > 0$ such that $\om$ is a $(\delta, R_\delta)$-chord arc domain.
\end{remarks}

To organize our arguments involving these types of domains, we introduce some notation. Given $x_0 \in \ree$ and $\vec{n} \in S^n$ we let $P(x_0, \vec n)$ be the plane through $x_0$ perpendicular to $\vec n$, that is,
\[P(x_0, \vec n) : = \{X \in \ree : \langle X - x_0, \vec n \rangle = 0\}; \]
and additionally, for $r > 0$ and $\xi \in \re$ we define the shifted cylinder
\begin{multline}\label{adcyl.eq}
\C(x_0, r, \vec n, \xi) := \bigg\{X \in \ree : |(X- x_0) - \langle X - x_0, \vec n \rangle \vec n| \le \frac{r}{\sqrt{n+1}}, \\
| \langle X - x_0, \vec n \rangle - \xi| \le  \frac{r}{\sqrt{n+1}}\bigg\},
\end{multline}
and if $\xi = 0$, we just write $\C(x_0, r, \vec n)$. 
Given $\delta>0$, we also define the following truncated cylinders
\[\C_\delta^+(x_0, r, \vec n) := \bigg\{X \in \C(x_0, r, \vec n):  \langle X - x_0, \vec n \rangle \ge 2\delta r\bigg\}\]
and
\[\C_\delta^-(x_0, r, \vec n) := \bigg\{X \in \C(x_0, r, \vec n):  \langle X - x_0, \vec n \rangle \ge - 2\delta r\bigg\}.\]
Notice that {\it both} $\C_\delta^+$ and $\C_\delta^-$ are `upper' portions of the cylinder $\C$. We also define the `strip'
\[\s_\delta(x_0, r, \vec n) := \bigg\{X \in \C(x_0, r, \vec n):  |\langle X - x_0, \vec n \rangle| \le 2\delta r\bigg\}. \]
We may omit the vector $\vec n$ in the above notation when there is no confusion.

Now we connect these objects to our definition of $(\delta, R)$-chord arc domain. Suppose $\om \subset \ree$ is a $(\delta, R)$-chord arc domain for some $\delta \in (0, \delta_n]$ and $R > 0$. Given $r \in (0,R]$ and $x_0 \in \pom$ we let $P_{x_0, r}$ be an $n$-plane such that 
\[D[\pom \cap B(x_0, r);P_{x_0, r} \cap B(x_0, r)] \le \delta r.\]
By the separation property and choice of $\delta_n$ sufficiently small we can ensure that for a choice of normal vector to the plane $P_{x_0, r}$, which we label $\vec n_{x_0,r}$, we have 
\[\C_\delta^+(x_0,r): = \C_\delta^+(x_0, r, \vec n_{x_0,r}) \subset \om\]
and
\[\C(x_0, r) \setminus C_\delta^-(x_0,r) \subset \om^c,\]
where $\C(x_0,r) : = \C(x_0, r, \vec n_{x_0,r})$ and $C_\delta^-(x_0,r) : = \C_\delta^-(x_0, r, \vec n_{x_0,r})$. (This means that $\vec n_{x_0,r}$ points {\it inwards}.)
We also define
\begin{equation}\label{def:Omtilde}
	\widetilde{\om}(x_0,r,\xi)= \om \cap \C(x_0, r, \vec n_{x_0,r},\xi)
\end{equation} 
for $\xi \in \re$ and drop $\xi$ from the notation when $\xi = 0$. 
We then note the following inclusions
\[\C_\delta^+(x_0,r) \subset \widetilde{\om}(x_0,r) \subset \C_\delta^-(x_0,r)\]
and
\[ \C(x_0, r) \cap \pom \subset \s_\delta(x_0, r),\]
where $\s_\delta(x_0, r): = \s_\delta(x_0, r, \vec n_{x_0,r})$. Note that, while $\widetilde{\om}(x_0,r)$ may not be an NTA or chord arc domain, the estimates established above hold in a smaller dilate of $\C(x_0,r)$ intersected with $\widetilde{\om}(x_0,r)$, when suitably interpreted (for instance, at points on the boundary of $\pom \cap \C(x_0,(1-c_n\delta)r)$). See \cite[Section 4]{KT-Duke}, in particular the discussion following \cite[Remark 4.2]{KT-Duke}.

Finally we recalls the definitions of $BMO$ and $VMO$ functions.
\begin{definition}[BMO and VMO]\label{BVMO.def}
Let $\mu$ be a Radon measure in $\mathbb{R}^n$. Then, for all $0 < r < \diam(\supp \mu)$ and all $f \in L^{2}(\mu)$ we define
\begin{equation} \label{e:tbmo}
\|f\|_{*}(x,r) = \sup_{0 < s < r} \left( \fint_{B(x,s)} \left| f_{\pom}(y) - \fint_{B(x,s)} f(z) \, d \mu(z) \right|^{2} \, d \mu(y)\right)^{\frac{1}{2}}.
\end{equation}
We say $f \in BMO(\mu)$ if
\begin{align} \label{e:bmonorm}
\nonumber \| f \|_{BMO(\mu)} := \sup_{0<r<\diam(\supp\mu)} \sup_{x \in \supp\mu} \| f ||_{*}(x,r) < +\infty.
\end{align}

We denote by VMO the closure of uniformly continuous functions on $\supp\mu$ in the BMO-norm. There is also a notion of VMO$_{\loc}$; $f \in$ VMO$_{\loc}$ if for every compact set $K \subset \mathbb{R}^{n}$,
\begin{equation} \label{e:vmoloc}
\lim_{r \to 0} \sup_{x \in \supp\mu \cap K} \left( \fint_{B(x,s)} \left| f(y) - \fint_{B(x,s)} f(z) \, d \mu(z) \right|^{2} \, d \mu(y) \right)^{\frac{1}{2}}=0.
\end{equation}
\end{definition}
\begin{remark}
	In this paper we will work with BMO and VMO functions with respect to the surface measure $\mathcal{H}^{n}|_{\partial\Omega}$, where $\Omega$ is a domain with Ahlfors regular boundary. 
\end{remark}

\section{Perturbations of constant coefficient operators and Poisson kernels}\label{sec:pert}
In this section, we restrict our attention to the upper half space $\ree_+$, and study the Poisson kernels for elliptic operators that are \textit{perturbations} of constant-coefficient operators. Let $A_0$ be a real, constant $(n+1)\times (n+1)$ matrix satisfying the $\Lambda$-ellipticity condition. We say a real $(n+1) \times (n+1)$ matrix satisfying the $\Lambda$-ellipticity condition is an \textbf{$\epsilon$-perturbation of $A_0$}, if $A(x,t)$ can be decomposed as
\begin{equation}\label{eq:decompA}
	A(x,t) =  A_1(x) + B(x,t)
\end{equation} 
with
\begin{equation}\label{def:pert}
	\|A_0 - A_1 \|_{L^\infty} + \|B(x,t)\|_{\C} \le \epsilon,
\end{equation}
see Definition \ref{Carnorm.def} for the definition of the Carleson norm $\| \cdot\|_{\C}$.

 The goal of this section is to prove the following result:

\begin{theorem}\label{UHSperturbfinal.thrm}
For any $\tilde \beta \in (0,1)$, there exist $\tilde\delta= \tilde\delta(\tilde\beta, n, \Lambda) > 0$ and $\epsilon = \epsilon(\tilde\beta, n, \Lambda)$ such that the following holds. 
Suppose $A=A(x,t)$ is a real matrix-valued function on $\ree_+$ satisfying the $\Lambda$-ellipticity condition and moreover, $A$ is an $\epsilon$-perturbation of a constant-coefficient matrix $A_0$ satisfying the $\Lambda$-ellipticity condition.
Then the Poisson kernel for $L= -\divg(A\nabla)$ in the upper half space, denoted by $k_A^X$, satisfies
\begin{equation}\label{eq:pertB2}
	1\leq \frac{\left(\fint_{\Delta(y, \delta'r)} \left[k_A^X(z)\right]^2 \, dz \right)^{1/2}}{\fint_{\Delta(y,\delta' r)} k_A^X(z) \, dz } < 1 + \tilde\beta
\end{equation} 
for all $y\in \rn$, $r > 0$, $\delta' < \tilde \delta $ and $X \in D(y,r)$. We recall that
\[ D(y,r) := B(y,10 r) \cap \left\{(x,t) \in \ree_+: t > \frac{r}{2} \right\} \]
is the upper cap of the ball $B(y,10 r)$.
\end{theorem}

We will refer to the quotient in \eqref{eq:pertB2} as the \textit{$B_2$ constant of $k_A^X$ on $\Delta(y,\delta' r)$}.
For a constant-coefficient operator $A_0$, the Poisson kernel $k_{A_0}^{(0,1)}$ is smooth, so roughly speaking
\[ \text{ the $ B_2 $ constant of $k_{A_0}^{(0,1)}$ on $\Delta(0,\delta)$ } \rightarrow 1 \text{ as } \delta \to 0. \]
We expect the same to hold for perturbations of constant-coefficient operators. 
The above theorem says that, if a matrix $A$ is a sufficiently small perturbation of a constant-coefficient matrix, then its $B_2$ constant is sufficiently close to $1$ on small enough scale (in proportion to the distance of the pole to the boundary).

We first prove a simple estimate on the Poisson kernel of constant-coefficient operators.

\begin{lemma}[Non-degeneracy for constant-coefficient operators]\label{constnondegen.lem}
There exists a constant $c_1 = c_1(n, \Lambda)  > 0$ such that the following holds. If $A_0$ is a real, constant matrix satisfying the $\Lambda$-ellipticity condition, then the Poisson kernel for $L_0 = -\div A_0 \nabla$ in the upper half space, denoted by $k_{A_0}$, satisfies
\begin{equation}\label{eq:ndeg1}
	 \|k^X_{A_0}\|_{L^1(\Delta(y, \delta r))} \geq c_1\delta^{n},
\end{equation}
\begin{equation}\label{eq:ndeg2}
	 \|k^X_{A_0}\|_{L^2(\Delta(y, \delta r))} \geq c_1 \left( \frac{\delta}{r} \right)^{\frac{n}{2}}.
\end{equation}
for all $y\in \rn$, $r>0$, $\delta \in (0,1]$ and $X\in D(y,r)$. 
\end{lemma}
\begin{proof}
The second inequality follows from combining \eqref{eq:ndeg1} and H\"older's inequality. 
Since the solutions to constant-coefficient operator are scale invariant, it suffices to show
\begin{equation}\label{eq:ndegsi}
	\|k_{A_0}^X \|_{L^1(\Delta(0,\delta))} \geq c_1 \delta^n, \quad \text{ for every } X\in D(0,1). 
\end{equation} 

We claim that it suffices to consider the case when $A_0$ is symmetric. If not, let
\[ A_{0,s} := \frac{A_0 + A_0^T}{2} \]
be the symmetrization of $A_0$. Since $A_0$ is a constant coefficient matrix, the solution $u\in C^2$, and we have
\begin{align*}
	\divg(A_0 \nabla u ) = \sum_{i,j} \partial_i(a_{ij} \partial_j u ) & = \sum_{i,j} a_{ij} \, \partial_i\partial_j u \\
	&= \frac{\sum_{i,j} a_{ij} \, \partial_i\partial_j u + \sum_{i,j} a_{ji} \,\partial_i\partial_j u}{2} = \divg(A_{0,s} \nabla u).
\end{align*} 
Assume the matrix $A_0$ is symmetric. By the decomposition of symmetric matrices and a change of variable formula, we get an explicit formula for $k_{A_0}^{(0,1)}$:
\begin{equation}\label{eq:formulakA0}
	k_{A_0}^{(0,1)}(z) =  c_n \det (A_0)^{-\frac12} \, \dfrac{ \left|\sqrt{A_0}^{-1} \begin{pmatrix}
	0 \\
	1
\end{pmatrix} \right|}{ \left|\sqrt{A_0}^{-1} \begin{pmatrix}
	-z \\
	1
\end{pmatrix} \right|^{n+1}  },  
\end{equation} 
where $c_n = \Gamma((n+1)/2)/\pi^{(n+1)/2}$ is the same constant in the Poisson kernel for the Laplacian, and $\sqrt{A_0}^{-1} $ is defined by the diagonalization decomposition of $A_0$ (the square root matrix and inverse matrix exist, since the eigenvalues of $A_0$ are all bounded from below by $1/\Lambda$).
Then the lower bound \eqref{eq:ndegsi} is obtained easily by \eqref{eq:formulakA0} and the Harnack inequality.

\end{proof}

Next we analyze how the Poisson kernels change under a perturbation of the coefficient matrices.

\begin{proposition}\label{perturbconstop.prop}
There exists $\epsilon_0 = \epsilon_0(n,\Lambda)$ such that the following holds for every $\epsilon < \epsilon_0$. If $A=A(x,t)$ is a real matrix valued function on $\ree_+$ satisfying the $\Lambda$-ellipticity condition and moreover, $A$ is an $\epsilon$-perturbation of a real, constant-coefficient matrix $A_0$ satisfying the $\Lambda$-ellipticity condition, then 
\[\|k_A^{(x,t)} - k_{A_0}^{(x,t)}\|_{L^2(\rn)} \le C\epsilon^{\eta'}t^{-n/2},\]
where $C = C(n,\Lambda)$ and $\eta' = \eta'(n, \Lambda) > 0$ and $k_A^{(x,t)}$ and  $k_{A_0}^{(x,t)}$ (resp.) are the Poisson kernels for the operators $L = -\div A \nabla$ and $L_0 = -\div A_0 \nabla$ (resp.) in the upper half space. 
\end{proposition}
\begin{remark}
	The above proposition could be amended, for instance, to include the case of perturbations of real, \textit{$t$-independent, symmetric} coefficients, but the case of a perturbation of constant coefficients is sufficient for our purposes.
\end{remark}

\begin{proof}
Without loss of generality assume $x = 0$. By the explicit formula in \eqref{eq:formulakA0}, it is easy to see that the kernel $k_{A_0}^{(0,1)}$ for the constant coefficient matrix $A_0$ has the form $ \frac{c}{\left(1+a_1^2 z_1^2 + \cdots + a_{n-1}^2 z_{n-1}^2 \right)^{\frac{n+1}{2}} }$, and thus $k_{A_0}^{(0,1)} \in L^2(\mathbb{R}^n)$. Fix $t > 0$. We can similarly get that $k_{A_0}^{(0,t)} \in L^2(\mathbb{R}^n)$. 

Our first goal is to show 
\begin{equation}\label{FG.eq}
\|k_{A_1}^{(0,t)} - k_{A_0}^{(0,t)}\|_{L^2(\rn)} \le C\epsilon^{\eta'}t^{-n/2},
\end{equation}
where $A_1(x)$ is the $t$-independent part in the decomposition \eqref{eq:decompA} of $A$. 

To that end we use the $L^2$-duality:
\begin{align}
\|k_{A_1}^{(0,t)} - k_{A_0}^{(0,t)}\|_{L^2(\rn)} &= \sup_{\|f\|_{L^2(\rn)} = 1} \left|\int_{\rn} \left(k_{A_1}^{(0,t)}(y) - k_{A_0}^{(0,t)}(y) \right) f(y) \, dy \right| \nonumber \\
& = \sup_{\|f\|_{L^2(\rn)} = 1} \left|u_f^1(0,t) - u_f^0(0,t) \right|, \label{eq:L2duality}
\end{align}
where for each $f\in L^2(\RR^n)$, we define 
\begin{equation}\label{eq:hmsol}
	u_f^i(x,t) = \int_{\RR^n} k_{A_i}^{(x,t)} (y) f(y) dy,\quad i=0,1. 
\end{equation} 
They are solutions to the Dirichlet problem for the operators $L_i = -\divg A_i \nabla $ with data $f$. Note that when we consider (not necessarily symmetric) elliptic operators in an unbounded domain $\mathbb{R}_+^{n+1}$, there may be several weak solutions to the Dirichlet problem, depending on which function spaces we consider (see \cite{Axel10} for example). Throughout the paper, we will refer to the solution obtained as in \eqref{eq:hmsol} (i.e. by integrating the boundary data against the elliptic measure) as the \textit{elliptic measure solution}. In particular when $0\leq f\in C_c^\infty(\mathbb{R}^n)$ the elliptic measure always satisfies $0\leq u_f(t,x) \leq \sup |f|$.

From the work of Auscher, Axelsson and Hofmann \cite[Theorem 1.1]{AAH} on $L^\infty$ perturbation result of the boundary value problem, the Dirichlet problem to $-\divg(A_1(x) \nabla u)=0$ with $L^2$ boundary data is well defined: for every $f\in L^2(\mathbb{R}^n)$ there exists solutions $\hat{u}_f^i$, $i=0,1$, such that
\begin{equation}\label{1star.eq}
\|\ntt \hat{u}_f^1\|_{L^2(\rn)}, \|\ntt \hat{u}_f^0\|_{L^2(\rn)} \le C \|f\|_{L^2(\rn)};
\end{equation}
moreover, we have
\begin{equation}\label{2star.eq}
\|\ntt (\hat{u}_f^1 - \hat{u}_f^0)\|_{L^2(\rn)} \le C\epsilon \|f\|_{L^2(\rn)},
\end{equation}
provided that $\epsilon > 0$ in the statement of the theorem is small enough depending on $n$ and $\Lambda$. We remark that this requires the use of (complex) analytic perturbations, but we restrict to a ``real neighborhood", where the result is (trivially) also true. When the boundary data $0\leq f \in C_c^\infty(\mathbb{R}^n)$, these two solutions agree, i.e. $u_f^i = \hat{u}_f^i$. We defer the proof to Subsection \ref{sec:uniquesol} of Appendix \ref{Firstorderapp.app}.

With $(x,t) = (0,t)$ fixed, we define $W := W(0,t) =  \Delta(0,t/2) \times (t/2, 3t/2)$ and for $\gamma \in (0,1]$ we define $\gamma W = \Delta(0, \gamma t/2) \times (t  -  \tfrac{\gamma t}{2} , t+ \tfrac{\gamma t}{2})$. We make the observation that for $\gamma \le 1/2$, $\gamma W \subset W(y,t)$\footnote{See Definition \ref{Carnorm.def}.} for all $y \in \Delta(x,t/4)$. 
It follows that for $\gamma \in (0, 1/2]$ and $y \in \Delta(x, t/4)$
\begin{align*}
\left(\fiint_{\gamma W} |u_f^i(Y)|^2 \, dY\right)^{1/2} &\le C\gamma^{-(n+1)/2} \left(\fiint_{W(y,t)} |u_f^i(Y)|^2 \, dY\right)^{1/2} 
\\& \le C\gamma^{-(n+1)/2}\ntt u^i_f(y)
\end{align*}
for $i = 0,1$. Averaging over $y \in \Delta(x,t/4)$ and using \eqref{1star.eq} we see
\begin{multline}\label{1startilde.eq}
\left(\fiint_{\gamma W} |u_f^i(Y)|^2 \, dY\right)^{1/2} \le C\gamma^{-(n+1)/2}t^{-n/2}\|\ntt u_f^i\|_{L^2(\Delta(x, t/4))} 
\\ \le C\gamma^{-(n+1)/2}t^{-n/2} \|f\|_{L^2(\rn)}
\end{multline}
for $i = 0,1$.
Similarly by \eqref{2star.eq}, we obtain that for $\gamma \in (0,1/2]$
\begin{equation}\label{2starprime.eq}
\left|\fiint_{\gamma W} \left( u_f^1(Y) - u_f^0(Y) \right) \, dY\right| \le C \epsilon \gamma^{-(n+1)/2}t^{-n/2} \|f\|_{L^2(\rn)}.
\end{equation}

Now we recall that solutions to divergence-form (real-valued) elliptic operators satisfy interior H\"older regularity, by DeGiorgi-Nash-Moser theory \cite{DeG,Nash, Moser}. Thus for $Y, Z \in (1/4)W$ and $i = 0,1$
\begin{multline}\label{1starminusprime.eq}
|u_f^i(Y) - u_f^i(Z)| \le C\left(\frac{|Y - Z|}{t} \right)^\eta \left(\fiint_{(1/2)W} |u^i_f|^2\right)^{1/2}
\\ \le  C\left(\frac{|Y - Z|}{t} \right)^\eta t^{-n/2} \|f\|_{L^2(\rn)},
\end{multline}
where we used \eqref{1startilde.eq} with $\gamma =1/2$ in the second inequality, and the constants $C$ and $\eta > 0$ depend only on $n$ and $\Lambda$. In particular, if we let $Y = (0,t)$ and average over $Z \in \gamma W \subset (1/4) W$ we obtain
\begin{equation}\label{1starprime.eq}
\left |u_f^i(0,t) - \fiint_{\gamma W} u_f^i(Z) dZ\right| \le \gamma^\eta t^{-n/2} \|f\|_{L^2(\rn)}
\end{equation}
for $i = 0,1$.

Setting $\gamma = \epsilon^{\frac{2}{2\eta + n + 1}}$ and then using the triangle inequality, \eqref{2starprime.eq} and \eqref{1starprime.eq} we obtain
\begin{equation}\label{eq:pwdiff}
|u_f^1(0,t) - u_f^0(0,t)| \le C(\gamma^\eta + \epsilon\gamma^{-(n+1)/2})t^{-n/2}\|f\|_{L^2(\rn)} \le C\epsilon^{\eta'}t^{-n/2}\|f\|_{L^2(\rn)},
\end{equation}
where $\eta' = \tfrac{2\eta}{2\eta + n + 1}$. Then the claim \eqref{FG.eq} follows from combining \eqref{eq:pwdiff} and the $L^2$ duality \eqref{eq:L2duality}.

Next, we show 
\begin{equation}\label{SG.eq}
\|k_{A_1}^{(0,t)} - k_{A}^{(0,t)}\|_{L^2(\rn)} \le C\epsilon^{\eta'}t^{-n/2}.
\end{equation}
This will follow exactly as the proof of \eqref{FG.eq}, provided we have the analogous estimates to \eqref{1star.eq} and \eqref{2star.eq}. These are afforded by the work of Auscher and Axelsson \cite{AA} on the boundary value problems for perturbations of $t$-independent operators\footnote{In particular, we must use the representation in \cite[Equation (42)]{AA} and the bounds established in \cite[Theorem 9.2, Lemma 10.2]{AA}. Here it should be noted that the functions $\tilde h^+$ used in the expression \cite[Equation (42)]{AA} are also in a perturbative regime with linear dependence on the FKP-Carleson norm, see the proof of \cite[Corollary 9.5]{AA}. See Appendix \ref{Firstorderapp.app}.}. Let $f \in L^2(\rn)$ and let $u_f$ be the (unique) solution to the $L^2$-Dirichlet problem for the operator $L = -\div A \nabla$. With $u_f^1$ as before, we may use \cite[Sections 9 and 10]{AA}
\begin{equation}\label{1starpprime.eq}
\|\ntt u_f \|_{L^2(\rn)} \le C \|f\|_{L^2(\rn)}
\end{equation}
and 
\begin{equation}\label{2starpprime.eq}
\|\ntt^{3/2} (u_f - u_f^1) \|_{L^2(\rn)} \le C\epsilon \|f\|_{L^2(\rn)},
\end{equation}
provided that $\epsilon$ is small depending on $n$ and $\Lambda$. We discuss the derivation of these two estimates in detail in Appendix \ref{Firstorderapp.app}.

Above we have $L^p$-averages with $p = 3/2$ in the definition of $\ntt^{3/2}$, but we obtain `comparability' of $L^2$ and $L^{3/2}$ averages using Moser's local boundedness estimate for solutions since we work with real-valued elliptic operators. To be more precise, by denoting $\widetilde{W}(x,t):= \Delta(x, t) \times (t/4,t) \supset W(x,t)$, Moser type estimate gives
\[ \sup_{Y\in W(x,t)} |u_f(Y)| \lesssim \left( \fint_{\widetilde{W}(x,2t)} |u_f|^2 ~dy ds \right)^{1/2} \leq \widehat{N}_*(u_f)(x),  \]
where $\widehat{N}_*(u_f)$ is the non-tangential maximal function defined as $\widetilde{N}_*(u_f)$ in \eqref{def:NF}, but with a fatter Whitney region $\widetilde{W}(x,t)$ in place of $W(x,t)$. On the other hand, it is a classical result 
that $\|\widetilde{N}_*(u_f)\|_{L^2(\rn)} \approx \|\widehat{N}_*(u_f)\|_{L^2(\rn)}$ (i.e. the $L^2$-norm of non-tangential maximal function is independent of the opening angle). The same holds for $u_f^1$. Hence we get
\begin{align}
	\|\ntt(u_f - u_f^1) \|_{L^2(\rn)} & \lesssim \left\| \left(\widehat{N}_*(u_f) \right)^{1/4} \cdot \left( \ntt^{3/2}(u_f-u_f^1) \right)^{3/4} \right\|_{L^2(\rn)} \nonumber  \\
	& \qquad \quad + \left\| \left(\widehat{N}_*(u_f^1) \right)^{1/4} \cdot \left( \ntt^{3/2}(u_f-u_f^1) \right)^{3/4} \right\|_{L^2(\rn)} \nonumber \\
	& \lesssim \left(  \|\ntt(u_f)\|_{L^2(\rn)}^{1/4} + \|\ntt(u_f^1)\|_{L^2(\rn)}^{1/4} \right) \|\ntt^{3/2}(u_f - u_f^1) \|_{L^2(\rn)}^{3/4} \nonumber \\
	& \leq C \epsilon^{3/4} \|f\|_{L^2(\rn)},\label{eq:pertAA1}
\end{align}
where the last estimate follows by combining \eqref{1starpprime.eq} and \eqref{2starpprime.eq}. 
The inequalities \eqref{1starpprime.eq} and \eqref{eq:pertAA1} are the direct analogue of \eqref{1star.eq} and \eqref{2star.eq}. 
 Proceeding exactly as before we obtain \eqref{SG.eq}, and combining this with \eqref{FG.eq} we have 
\[\|k_A^{(x,t)} - k_{A_0}^{(x,t)}\|_{L^2(\rn)} \le C\epsilon^{\eta'}t^{-n/2},\]
as desired. 
\end{proof}

Proposition \ref{perturbconstop.prop} has one unfortunate drawback the estimate depends on the placement of the pole. This stops us from working with the small scales without doing some extra work. In order to do work with the small scales, we take three steps:
\begin{enumerate}
\item Use Proposition \ref{perturbconstop.prop} to immediately give us a good `single-scale $B_2$ type estimate'  for perturbations of constant coefficient operators (Corollary \ref{cor:B2npc}). 
\item Observe how the change of pole argument interacts with single-scale $B_2$ estimates to allow us to move down to small scales (Lemma \ref{goodpolechange.lem}) by paying a (controllable) penalty.
\item Combining (1) and (2) we obtain a good estimate for perturbations of constant coefficient operators down to small scales (Corollary \ref{Messcor.cor}).
\end{enumerate}

\begin{corollary}\label{cor:B2npc}
	For every $\beta \in (0,1)$ and $\delta\in (0,1)$ there exists $\epsilon_1 = \epsilon_1(\delta, \beta, n, \Lambda) \in (0, \epsilon_0)$ such that the following holds.
	Suppose $A=A(x,t)$ is a real matrix-valued function on $\ree_+$ satisfying the $\Lambda$-ellipticity condition and $A$ is an $\epsilon_1$-perturbation of a real, constant-coefficient matrix $A_0$ satisfying the $\Lambda$-ellipticity condition.
Then
\begin{equation}\label{Messcor.eq1}
\frac{\left(\fint_{\Delta(y,\delta r)} \left[k_A^X(z)\right]^2 \, dz \right)^{1/2}}{\fint_{\Delta(y,\delta r)} k_A^X(z) \, dz } \le (1 + \beta) \frac{\left(\fint_{\Delta(y,\delta r)} \left[k_{A_0}^X(z)\right]^2 \, dz \right)^{1/2}}{\fint_{\Delta(y,\delta r)} k_{A_0}^X(z) \, dz }
\end{equation}
for all $y\in \rn$, $r > 0$, $y \in \rn$ and $X \in D(y,r)$.
\end{corollary}

\begin{proof}
	Let $\beta_1>0$ be fixed, whose value will be chosen later. By Proposition \ref{perturbconstop.prop} and Lemma \ref{constnondegen.lem}, if $\epsilon_1$ is chosen so that $ C \epsilon_1^{\eta'} \le \beta_1 c_1 \delta^{n/2}$ we have
 \[\|k_A^X - k_{A_0}^X\|_{L^2(\Delta(y,\delta r))}  \le \beta_1 \|k_{A_0}^X \|_{L^2(\Delta(y,\delta r))} \]
 and 
  \[\|k_A^X - k_{A_0}^X\|_{L^1(\Delta(y,\delta r))}  \le \beta_1 \|k_{A_0}^X \|_{L^1(\Delta(y,\delta r))}. \]
 Thus, for this choice of $\epsilon_1$ we have
 \[\frac{\|k_A^X\|_{L^2(\Delta(y, \delta r))}}{\|k_A^X\|_{L^1(\Delta(y, \delta r))}} \le \frac{1 + \beta_1}{1 - \beta_1} \frac{\|k_{A_0}^X\|_{L^2(\Delta(y, \delta r))}}{\|k_{A_0}^X\|_{L^1(\Delta(y, \delta r))}}, \]
 which yields \eqref{Messcor.eq1} upon normalization and appropriate choice of $\beta_1$ (so that $\tfrac{1 + \beta_1}{1 - \beta_1} \le 1+\beta$).
\end{proof}

The above corollary says that for any $\beta \in (0,1)$, if $A$ is a small perturbation of a constant-coefficient matrix $A_0$, then the $B_2$ constant of $k_A^X$ is bounded by $(1+\beta)$ times that of $k_{A_0}^X$. As mentioned above, the caveat is that the smallness also depends on $\delta$, which is the ratio between the radius of the surface ball in consideration and the distance from the pole to the boundary. This prevents us from getting $B_2$ type estimates for all sufficiently small scales (with fixed pole). We will overcome this by a `good change of pole'.

\begin{lemma}[Change of pole comparison]\label{goodpolechange.lem}
Let $L = -\div A \nabla$ be a real divergence form operator on $\ree_+ $, with $A$ satisfying the $\Lambda$-ellipticity condition. We also assume the corresponding Poisson kernel $k_A$ exists and is in $L_{loc}^2(\rn)$. 
There are constants $\delta_0 = \delta_0(n,\Lambda)\in (0,1)$ and $C = C(n,\Lambda)>1$ such that for any $\delta\in (0,\delta_0)$ fixed, 

\begin{multline}\label{polechangeavg.eq}(1 - C\delta^\mu)\frac{\left(\fint_{\Delta(y,\delta'r)}\left(k_A^{(y,\tfrac{\delta'}{\delta}r)}(z)\right)^2 \right)^{1/2}}{\fint_{\Delta(y,\delta'r)}k_A^{(y,\tfrac{\delta'}{\delta}r)}(z)} \le \frac{\left(\fint_{\Delta(y,\delta'r)}\left(k_A^{X}(z)\right)^2 \right)^{1/2}}{\fint_{\Delta(y,\delta'r)}k_A^{X}(z)}
\\ \le (1 + C\delta^\mu)\frac{\left(\fint_{\Delta(y,\delta'r)}\left(k_A^{(y,\tfrac{\delta'}{\delta}r)}(z)\right)^2 \right)^{1/2}}{\fint_{\Delta(y,\delta'r)}k_A^{(y,\tfrac{\delta'}{\delta}r)}(z)}
\end{multline}
for any $y\in \rn$, $r > 0$, $\delta'\in (0, \delta/10]$ and $X \in D(y,r)$.
\end{lemma}

\begin{proof}
	We apply Lemma \ref{Kernelcompare.lem}, which allows us to change poles of the elliptic measure, to the upper half space $\ree_+$.
	To be precise, we consider the radius $s: =\frac{2}{3} \frac{\delta'}{\delta} r$, $\zeta = \frac{3}{2} \delta \in (0, 1/2) $ and poles $X_0 = X \in D(y,r)$ and $X_1 = (y, \delta'r/\delta) $. The fact that $X_0 \notin B(y,4s)$ follows from the assumption $\delta' \leq \delta/10$, and clearly $X_1\in B(y,2s)\setminus B(y,s)$. Then there are constants $C, \mu>0$ such that the kernel function
	\[ H(z) = \frac{d\omega^X}{d\omega^{(y, \frac{\delta'}{\delta} r)}} (z) \]
	satisfies
	\[
		\left| \frac{H(z)}{H(y)} - 1 \right| \leq C \delta^{\mu}, \quad \text{ for every } z\in \Delta(y, \zeta s) = \Delta(y, \delta'r ).
	\]
	Hence when $\delta$ is sufficiently small, we have for every $z, z'\in \Delta(y, \delta' r)$,
	\[ \frac{H(z)}{H(z')} \leq \frac{H(y) (1+C \delta^{\mu})}{H(y) (1-C \delta^{\mu})} \leq 1+C' \delta^{\mu}. \]
	On the other hand,
	\[ k_A^X(z) = \frac{d\omega^X}{d\sigma}(z) = H(z) \, \frac{d\omega^{(y, \frac{\delta'}{\delta} r)}}{d\sigma}(z) = H(z) \, k_A^{(y, \frac{\delta'}{\delta} r)}(z). \]
	Therefore
	\begin{align*} \frac{\left(\fint_{\Delta(y,\delta'r)}\left(k_A^{X}(z)\right)^2 \right)^{1/2}}{\fint_{\Delta(y,\delta'r)}k_A^{X}(z)}
 &\le \sup_{z,z' \in \Delta(y, \delta'r)}\frac{H(z)}{H(z')} \cdot \frac{\left(\fint_{\Delta(y,\delta'r)}\left(k_A^{(y,\tfrac{\delta'}{\delta}r)}(z)\right)^2 \right)^{1/2}}{\fint_{\Delta(y,\delta'r)}k_A^{(y,\tfrac{\delta'}{\delta}r)}(z)}
\\ &\le (1 + C'\delta^\mu)\frac{\left(\fint_{\Delta(y,\delta'r)}\left(k_A^{(y,\tfrac{\delta'}{\delta}r)}(z)\right)^2 \right)^{1/2}}{\fint_{\Delta(y,\delta'r)}k_A^{(y,\tfrac{\delta'}{\delta}r)}(z)}.
\end{align*}
This is the second inequality of \eqref{polechangeavg.eq}.
The first inequality is obtained similarly.
\end{proof}

\begin{corollary}\label{Messcor.cor} 
For every $\beta \in (0,1)$ and $\delta\in (0, \delta_0)$, where $\delta_0$ is from Lemma \ref{goodpolechange.lem}, there exists $\epsilon_1 = \epsilon_1(\delta, \beta, n, \Lambda) \in (0, \epsilon_0)$ such that the following holds for every $\epsilon< \epsilon_1$. (Here $\epsilon_0$ is as in Proposition \ref{perturbconstop.prop}.)

If $A=A(x,t)$ is a real matrix valued function on $\ree_+$ satisfying the $\Lambda$-ellipticity condition and moreover, $A$ is an $\epsilon$-perturbation of a real, constant-coefficient matrix $A_0$ satisfying the $\Lambda$-ellipticity condition, then
\begin{equation}\label{Messcor.eq2}
\frac{\left(\fint_{\Delta(y, \delta'r)} \left[k_A^X(z)\right]^2 \, dz \right)^{1/2}}{\fint_{\Delta(y,\delta' r)} k_A^X(z) \, dz }\leq 
(1 + C \delta^\mu)(1 + \beta) \frac{\left(\fint_{\Delta(0,\delta)} \left[k_{A_0}^{(0,1)}(z)\right]^2 \, dz \right)^{1/2}}{\fint_{\Delta(0,\delta)} k_{A_0}^{(0,1)}(z) \, dz } 
\end{equation}
for every $y\in \rn$, $r > 0$, $ \delta' \in (0,\delta/10]$ and $X \in D(y,r)$.
\end{corollary}

\begin{proof}
	The proof is by combining Corollary \ref{cor:B2npc} and the change of pole comparison in Lemma \ref{goodpolechange.lem}. Set $s= \frac{\delta'}{\delta} r$, the estimates
	\eqref{polechangeavg.eq} and \eqref{Messcor.eq1} yield
\begin{align*} \frac{\left(\fint_{\Delta(y,\delta'r)}\left(k_A^{X}(z)\right)^2 \right)^{1/2}}{\fint_{\Delta(y,\delta'r)}k_A^{X}(z)}
&\le (1 + C \delta^\mu)\frac{\left(\fint_{\Delta(y,\delta s)}\left(k_A^{(y, s)}(z)\right)^2 \right)^{1/2}}{\fint_{\Delta(y,\delta s)}k_A^{(y, s)}(z)}
\\ &\le (1 + C\delta^\mu)(1 + \beta) \frac{\left(\fint_{\Delta(y,\delta s)}\left(k_{A_0}^{(y, s)}(z)\right)^2 \right)^{1/2}}{\fint_{\Delta(y,\delta s)}k_{A_0}^{(y,s)}(z)}
\\ & \le (1 + C\delta^\mu)(1 + \beta) \frac{\left(\fint_{\Delta(0,\delta)} \left[k_{A_0}^{(0,1)}(z)\right]^2 \, dz \right)^{1/2}}{\fint_{\Delta(0,\delta)} k_{A_0}^{(0,1)}(z) \, dz }, 
\end{align*}
where we use the translation and dilation invariance for $k_{A_0}$, i.e. Poisson kernel for constant-coefficient operator, in the last estimate.
\end{proof}

We now apply Corollary \ref{Messcor.cor} twice. The first application is to get uniform control on the scale at which the $B_2$ constant becomes `near optimal' for constant coefficient operators. Surely \eqref{eq:formulakA0} allows us to show the following Lemma, but we we prove this via a compactness argument here only using the smoothness of the Poisson kernel for constant coefficient operators and the previous corollary.

\begin{lemma}\label{constantcoeffoptimal.thrm}
For every $\beta' \in (0,1)$ there exists $\delta_1 = \delta_1(\beta', n, \Lambda) \in ( 0, \delta_0) $ such that for any real, constant-coefficient elliptic matrix $A_0$ satisfying the $\Lambda$ ellipticity condition, we have
\[\frac{\left(\fint_{\Delta(0,\delta)} \left[k_{A_0}^{(0,1)}(z)\right]^2 \, dz \right)^{1/2}}{\fint_{\Delta(0,\delta)} k_{A_0}^{(0,1)}(z) \, dz } \le 1 + \beta'\]
for all $\delta \in (0, \delta_1)$.
\end{lemma}
\begin{proof} Let $\beta' \in (0,1)$ and set $\Lambda' = 2\Lambda$. Let $\beta \in (0,1)$ and $\delta_2 \in (0,\delta_0)$ be such that
\[(1+ C\delta_2^{\mu})(1+ \beta)^2 \leq 1 + \beta',\]
where $\delta_0 = \delta_0(n,\Lambda')$ is as in Lemma \ref{goodpolechange.lem} for $\Lambda'$. By \eqref{eq:formulakA0}, for every real, constant $A_0$ satisfying the $\Lambda$-ellipticity condition we have that $k_{A_0}^{(0,1)}(z)$ is a smooth function in $z$. Smoothness, combined with the non-degeneracy \eqref{eq:ndeg1}, yields that there exists $\tilde\delta=\tilde\delta(A_0)$, which we may assume is less than $\delta_2$, such that
\[\frac{\left(\fint_{\Delta(0,\tilde\delta)} \left[k_{A_0}^{(0,1)}(z)\right]^2 \, dz \right)^{1/2}}{\fint_{\Delta(0,\tilde\delta)} k_{A_0}^{(0,1)}(z) \, dz } \le 1 + \beta.\]
Thus, the estimate \eqref{Messcor.eq2} in Corollary \ref{Messcor.cor} implies that there exists\footnote{Note that, in particular, $\tilde\epsilon_1 = \tilde\epsilon_1(A_0)$ since $\tilde\delta$ (momentarily) depends on $A_0$.} $\tilde \epsilon_1 = \tilde\epsilon_1(\tilde\delta, \beta, n, \Lambda')< (5\Lambda)^{-1}$ such that if $A_0'$ is a real, constant coefficient matrix the $\Lambda'$-ellipticity condition and $\|A_0 - A_0'\|_{L^\infty} < \tilde\epsilon_1$
\begin{align*}
\frac{\left(\fint_{\Delta(0,\delta)} \left[k_{A_0'}^{(0,1)}(z)\right]^2 \, dz \right)^{1/2}}{\fint_{\Delta(0,\delta)} k_{A_0'}^{(0,1)}(z) \, dz } &\le (1 + C\tilde{\delta}^\mu)(1 + \beta) \frac{\left(\fint_{\Delta(0,\tilde\delta)} \left[k_{A_0}^{(0,1)}(z)\right]^2 \, dz \right)^{1/2}}{\fint_{\Delta(0,\tilde\delta)} k_{A_0}^{(0,1)}(z) \, dz } 
\\ &\le 1 + \beta',
\end{align*}
for all $\delta \in (0, \tilde\delta/10]$. 
On the other hand, the collection of balls of the form $B_\infty(A_0, \tilde\epsilon_1(A_0)) := \{ A_0' : \|A_0 - A_0'\|_{L^\infty} < \tilde\epsilon_1(A_0)\}$, as $A_0$ ranges over all real, constant-coefficient matrices, forms an open\footnote{This is why we employed the use of $\Lambda'$ and made the restriction $\tilde\epsilon_1 < (5\Lambda)^{-1}$.} cover of the set of all real, constant-coefficient matrices satisfying the $\Lambda$-ellipticity condition (in the $L^\infty$-metric), which is a compact set. We may then extract a finite sub-cover $\{B_i\} = \{B(A_0^i,\tilde\epsilon_1(A_0^i))\}$ from which the conclusion of the theorem follows by letting $\delta_1 := \min_{i} \tilde\delta(A_0^i)/10$.
\end{proof}

Now we are ready to conclude our treatment of perturbations of constant-coefficient operators. We are (finally) able to say, quantitatively, that sufficient proximity to a constant-coefficient operator in the sense of \eqref{def:pert} controls the $B_2$ constant at small scales.

\textit{Proof of Theorem \ref{UHSperturbfinal.thrm}.} Let $\beta' \in (0,1)$ be such that $(1 + \beta')^2 =1 + \widetilde\beta/2$. Next, we choose $\delta \in (0,\delta_1)$, where $\delta_1 =\delta_1(\beta', n, \Lambda)$ is from Lemma \ref{constantcoeffoptimal.thrm}, small enough so that
\[ (1 + C\delta^\mu)(1 + \beta')^2 \le 1 + \widetilde\beta.\]
We apply Corollary \ref{Messcor.cor} with $\beta = \beta'$ and $\delta$ as above.
It follows that if $\epsilon = \epsilon_1(\delta, \beta', n, \Lambda) > 0$ is as in Corollary \ref{Messcor.cor} (with $\beta = \beta'$) we have for all $A$ satisfying the hypothesis of the theorem 
\begin{align*} \frac{\left(\fint_{\Delta(y, \delta'r)} \left[k_A^X(z)\right]^2 \, dz \right)^{1/2}}{\fint_{\Delta(y,\delta' r)} k_A^X(z) \, dz }&\le (1 + C\delta^\mu)(1 + \beta') \frac{\left(\fint_{\Delta(y,\delta)} \left[k_{A_0}^{(0,1)}(z)\right]^2 \, dz \right)^{1/2}}{\fint_{\Delta(0,\delta)} k_{A_0}^{(0,1)}(z) \, dz } 
\\& \le 1 + \widetilde\beta,
\end{align*}
for all $y\in \rn$, $r > 0$, $\delta' \in (0, \delta/10]$ and $X \in D(x,r)$, where we use that $\delta\in (0,\delta_1)$ and Lemma \ref{constantcoeffoptimal.thrm} in the second line. Setting $\tilde\delta = \delta/10$ we obtain the conclusion of the theorem.

\section{Operators with H\"older coefficients in Lipschitz domains}
In this section, we modify H\"older-continuous coefficient matrices (in Lipschitz domains) so that they are small perturbations of constant coefficient matrices, and this will allow us to use the results of Section \ref{sec:pert}. We will rely on the ``flat" transformation between Lipschitz domains and the upper half space. The relevant computations are standard, so we include them in Appendix \ref{pushpull.sect}. But we suggest reader to first read the notation and statements in Appendix \ref{pushpull.sect}, since we rely heavily on them in this section. For instance, we often employ the `flattening map' $\Phi$ defined in Appendix \ref{pushpull.sect}. 

In what follows, if $\tau > 0$ we let $Q_\tau$ be the open $n$-dimensional cube centered at zero with side length $2\tau$, that is, $Q_\tau = \{x \in \rn : |x|_\infty < \tau\}$. 
\begin{lemma}\label{Qtaulocal.lem} Let $\varphi : \rn \to \re$, $\varphi(0) = 0$ be a Lipschitz function with Lipschitz constant $\gamma \le \frac{1}{50n}$. Suppose $A$ is a $\Lambda$-elliptic matrix satisfying the H\"older condition \eqref{cond:Holder}.
Then for $\tau > 0$ the matrix-valued function\footnote{We remind the reader that $R_{Q_\tau, \varphi}$ is the $\varphi$-adapted Carleson box, defined in \eqref{def:Cbox}.}
\[A_{\tau,\varphi} (x,t) := 
\begin{cases}
A(0) & \text{ if } x \not\in Q_\tau
\\ A(x,t) & \text{ if } (x,t) \in R_{Q_\tau, \varphi}
\\ A(x,\varphi(x)) & \text{ otherwise}
\end{cases}
\]
has the decomposition
\[A_{\tau,\varphi}(x,t) = A_1(x) + B(x,t),\]
with 
\begin{equation}\label{Qtaulocalestimatemain.eq}
\| A_1 - A(0) \|_{L^\infty(\rn)} + \|B\|_{\C_\varphi} < C_1\tau^\alpha,
\end{equation}
where $C_1$ depends on $C_A, n$ and $\alpha$. Here 
\[ A_1(x) := A(x,\varphi(x))\mathbbm{1}_{Q_\tau}(x) + A(0)(1 - \mathbbm{1}_{Q_\tau}(x)), \text{ and } B(x,t) :=A_{\tau,\varphi}(x,t) - A_1(x). \]

In particular, $\widetilde{A} = J_\Phi^T(A_{\tau,\varphi}\circ \Phi^{-1})J_\Phi$ is a real, $4\Lambda$-elliptic matrix with the decomposition
\[\widetilde{A}(x,t) = \widetilde{A}_1(x) + \widetilde{B}(x,t),\]
satsifying
\[\| \widetilde{A}_1 - A'_0 \|_{L^\infty(\rn)} + \|\widetilde{B}\|_{\C} < 2C_1\tau^\alpha + 4\sqrt{n}\gamma\Lambda,\]
where $\widetilde{A}_1 := J^T_\Phi(A_1 \circ \Phi^{-1})J_\Phi = J^T_\Phi(A_1)J_\Phi$, $\widetilde{B} := J^T_\Phi(B \circ \Phi^{-1})J_\Phi$ and
$A'_0 := J^T_\Phi(0)A(0)J_\Phi(0)$.
\end{lemma} 
\begin{proof}
Fix $\tau > 0$. The second statement concerning $\widetilde A$, its decomposition and the corresponding bounds follows immediately from Proposition \ref{genopperturbPB.prop} and \eqref{Qtaulocalestimatemain.eq}.

To prove the estimate \eqref{Qtaulocalestimatemain.eq}, we begin with the $L^\infty$ estimate for $A_1(x) - A_0$. For $x \in Q_\tau$ 
\[|A_1(x) - A(0)| = |A(x,\varphi(x)) - A(0)| \le C_A |(x,\varphi(x))|^\alpha \lesssim \tau^\alpha,\]
where the implicit constant depends on $C_A$, $n$ and $\alpha$. 
Since $A_1(x) = A(0)$ for $x \not\in Q_\tau$, we have obtained the estimate $\|A_1 - A(0)\|_{L^\infty(\rn)} < C\tau^\alpha$.

We are left with estimating the $\varphi$-adapted Carleson norm of $B : = A_{\tau,\varphi} - A_1$. For any $X \in R_{Q_\tau,\varphi}$, we write $X=(x, \varphi(x)+t) \in \Omega_\varphi$ and $\hat X = (x, \varphi(x))$. Hence 
\[|B(X)| = |A_{\tau,\varphi}(X) - A_1(X)| = |A(X) - A(\widehat{X})| \le C_At^{\alpha}.\]
Since $B(X) \equiv 0$ in $(R_{Q_\tau,\varphi})^c$ it follows that for any cube $Q \subset \rn$
\begin{align*}
\|B\|_{\C_\varphi} &= \sup_{Q \subset \rn} \left( \frac{1}{|Q|} \iint_{R_{Q,\varphi}} \| B\|_{L^\infty(W_\varphi(y,s))}^2 \, \frac{dy\,ds}{s - \varphi(y)} \right)^{1/2} \\ 
& = \sup_{Q \subset \rn} \left( \frac{1}{|Q|} \int_0^{\ell(Q)} \int_Q \| B\circ \Phi^{-1} \|_{L^\infty(W_(x,t))}^2 \, \frac{dx\,dt}{t} \right)^{1/2} \\
& \le C \sup_{Q\subset \RR^n} \left( \int_0^{\min(\ell(Q),4\tau)} t^{2\alpha -1} \, dt \right)^{1/2} \le C \tau^\alpha,
\end{align*}
where we used the flattening change of variables in the second line. Combining our estimates for $\|A_1 - A(0)\|_{L^\infty(\rn)}$ and $\|B\|_{\C_\varphi}$ we obtain \eqref{Qtaulocalestimatemain.eq}.
\end{proof}

The next Proposition says that the modified coefficient matrix enjoys an almost optimal reverse-H\"older estimate.

\begin{proposition}\label{mainlocalization.prop}
Let $\beta>0$. There exist $\delta_\beta, \gamma_\beta, \tau_\beta>0$ depending on $\beta$ and allowable constants such that the following holds. 

Assume $A$ is a real, $\Lambda$-elliptic matrix-valued function satisfying the H\"older condition \eqref{cond:Holder}, and $\varphi : \rn \to \re$, $\varphi(0) = 0$, is Lipschitz with $\|\nabla \varphi\|_\infty \le \gamma_\beta$. Then for all $\tau\in (0,\tau_\beta)$, the Poisson kernel (denoted by $h_{\tau,\varphi}^X$) for $L_{\tau, \varphi} = -\div A_{\tau, \varphi}\nabla$\footnote{See Lemma \ref{Qtaulocal.lem}.} and the domain $\om_\varphi$ with pole $X \in \om_\varphi$ satisfies
\begin{equation}\label{localizedkerneloptest.eq}
\frac{\left(\fint_{\Delta_\varphi(y, \delta r)} (h^X_{\tau, \varphi})^2\, d\sigma\right)^{1/2}}{\fint_{\Delta_\varphi(y, \delta r)} h^X_{\tau, \varphi}\, d\sigma} \le 1 + \beta,
\end{equation}
for all $y\in \RR^n, r>0, \delta\in (0, \delta_\beta)$ and $X\in \Phi^{-1}(D(y,r))$. Here $\sigma$ is the surface measure to $\Gr(\varphi) = \{(x, \varphi(x)): x \in \rn\}$, 
$\Delta_\varphi(y, \delta r) := B((y, \varphi(y)), \delta r) \cap \Gr(\varphi)$ is the surface ball, and $\Phi$ is the flattening map for $\varphi$. 
\end{proposition}

\begin{proof} 
Let $\widetilde{\beta} = \min\{1, \beta/2\}$ and let $k_{\widetilde A}^Z$ be the ``pulled back" Poisson kernel, that is, $k_{\widetilde A}^Z$ is the Poisson kernel in the upper half space for $\widetilde L= -\div \widetilde A\nabla$, where $\widetilde A  = J_\Phi^T (A_{\tau, \varphi} \circ \Phi^{-1}) J_\Phi$. Using Lemma \ref{Qtaulocal.lem}, there exists $A'_0$, a real, constant $4\Lambda$-elliptic matrix such that $\widetilde A$ has the decomposition
\[\widetilde A(x,t) = \widetilde A_1(x) + \widetilde B(x,t)\]
and 
\[\|\widetilde A - A'_0\|_{L^\infty} + \|\widetilde B\|_{\C} \le 2 C_1\tau_\beta^\alpha + 4\sqrt{n}\gamma_\beta \Lambda.\]
In particular we may initially choose $\tau_\beta$ and $\gamma_\beta$ small depending on $\tilde \beta$ and allowable constants, to guarantee that $\widetilde A$ is an $\epsilon(\tilde\beta)$-perturbation of the constant-coefficient matrix $A'_0$ (as in \eqref{eq:decompA}, \eqref{def:pert}), so that by Theorem \ref{UHSperturbfinal.thrm} we have
\begin{equation}\label{smallnessforpullback.eq}
\frac{\left(\fint_{\Delta(y, \delta'r)} \left[k_{\widetilde{A}}^Z(x')\right]^2 \, dx' \right)^{1/2}}{\fint_{\Delta(y,\delta' r)} k_{\widetilde{A}}^Z(x') \, dx'} < 1 + \widetilde\beta \le 2
\end{equation}
for any $y\in \RR^n, r>0, \delta'\in (0, \tilde\delta)$ and $Z\in D(y,r)$.
The constant $\tilde\delta= \tilde\delta(\tilde\beta, n, \Lambda)$ was determined in Theorem \ref{UHSperturbfinal.thrm}.
We now fix $\tau_\beta$ and set $\delta_\beta = \tilde \delta$, but we will further restrict $\gamma_\beta$ in order to control the errors coming from the flattening change of variables. 

Fix $r > 0$, $\delta \in (0, \delta_\beta)$, $y \in \rn$ and $X \in \Phi^{-1}(D(y,r))$ and set $Z = \Phi(X) \in D(y,r)$. Let $P:\ree \to \rn$ be the projection operator, that is, $P(x,t) := x$ for all $(x,t) \in \ree$. By Proposition \ref{Poissonkernelofpullback.prop} and \eqref{smallnessforpullback.eq}
\begin{equation}\label{initialhestimate.eq}
\frac{\left(\fint_{\Delta_\varphi(y, \delta r)} (h^X_{\tau, \varphi})^2\, d\sigma\right)^{1/2}}{\fint_{\Delta_\varphi(y, \delta r)} (h^X_{\tau, \varphi})^2\, d\sigma} \le 
\frac{\HH^n(\Delta_\varphi(y, \delta r))}{\left| P(\Delta_\varphi(y, \delta r)) \right|}   \frac{\left(\fint_{P(\Delta_\varphi(y, \delta r))} \left[k_{\widetilde{A}}^Z(x')\right]^2 \, dx' \right)^{1/2}}{\fint_{P(\Delta_\varphi(y, \delta r))} k_{\widetilde{A}}^Z(x') \, dx'}, 
\end{equation}
where the error term comes from changing the averages. Let us now make the simple observation that for $x_0 \in \rn$ and $r_0 > 0$
\begin{equation}\label{projectioninclusion.eq}
\Delta(x_0, r_0(1+ \|\nabla \varphi\|_\infty^2)^{-1/2}) \subseteq P(\Delta_\varphi(x_0, r_0)) \subseteq \Delta(x_0,  r_0)),
\end{equation}
where the second inclusion is obvious and the first, similarly, is a consequence of the Pythagorean theorem. Indeed, if $x \in \Delta(x_0, r_0(1+ \|\nabla \varphi\|_\infty^2)^{-1/2})$ then 
\[|(x, \varphi(x)) - (x_0, \varphi(x_0))|^2 \le |x - x_0|^2 + \| \nabla \varphi\|_\infty^2 |x - x_0|^2 < r_0.\]
The inclusions in \eqref{projectioninclusion.eq} give the estimate 
\[\frac{\left| \Delta(y, \delta r) \right|}{\HH^n(P(\Delta_\varphi(y, \delta r)))} \le (\delta r)^n\left(\frac{\delta r}{\sqrt{1+ \|\nabla \varphi\|_\infty^2}}\right)^{-n} = (1+ \|\nabla \varphi\|_\infty^2)^{-n/2}.\]
The above estimate and \eqref{projectioninclusion.eq} yield
\begin{equation}\label{kpullbacknumest.eq}
\left(\fint_{P(\Delta_\varphi(y, \delta r))} \left[k_{\widetilde{A}}^Z(x')\right]^2 \, dx' \right)^{1/2} \le (1 + \|\nabla \varphi\|_\infty^2)^{n/4} \left(\fint_{\Delta(y, \delta r)} \left[k_{\widetilde{A}}^Z(x')\right]^2 \, dx' \right)^{1/2}.
\end{equation}
Again using \eqref{projectioninclusion.eq} we have 
\begin{equation}\label{kpullbackdenomest.eq}
\begin{split}
\fint_{P(\Delta_\varphi(y, \delta r))} k_{\widetilde{A}}^Z(x') \, dx' &\ge  \frac{\omega^Z_{\widetilde{A}}(\Delta(y, \delta r(1+ \|\nabla \varphi\|_\infty^2)^{-1/2}))}{\left| P(\Delta_\varphi(y, \delta r)) \right|}
\\& \ge \frac{\omega^Z_{\widetilde{A}}(\Delta(y, \delta r(1+ \|\nabla \varphi\|_\infty^2)^{-1/2}))}{\omega^Z_{\widetilde{A}}(\Delta(y, \delta r))} \fint_{\Delta(y, \delta r)} k_{\widetilde{A}}^Z(x') \, dx',
\end{split}
\end{equation}
where $\omega^Z_{\widetilde{A}} = k_{\widetilde A}^Z \, dx'$ is the elliptic measure for $\widetilde{L}$ on $\ree_+$.

Combining \eqref{initialhestimate.eq}, \eqref{kpullbacknumest.eq} and \eqref{kpullbackdenomest.eq} with the estimate 
\[H^n(\Delta_\varphi(y, \delta r)) \le \sqrt{1 + \|\nabla\varphi\|_\infty^2} \left| P(\Delta_\varphi(y, \delta r)) \right|\]
we have
\begin{equation}\label{secondtolasthest.eq}
\frac{\left(\fint_{\Delta_\varphi(y, \delta r)} (h^X_{\tau, \varphi})^2\, d\sigma\right)^{1/2}}{\fint_{\Delta_\varphi(y, \delta r)} (h^X_{\tau, \varphi})^2\, d\sigma} \le (1 + \widetilde\beta) (1 + \|\nabla\varphi\|_\infty^2)^{\frac{n+1}{4}} \frac{\omega^Z_{\widetilde{A}}(\Delta(y, \delta r))}{\omega^Z_{\widetilde{A}}(\Delta(y, \delta r(1+ \|\nabla \varphi\|_\infty^2)^{-1/2}))}
\end{equation}
where we used \eqref{smallnessforpullback.eq} to control the ratio of the averages of $k_{\widetilde A}^Z$ by $(1 + \widetilde\beta)$. Clearly we can make $(1 + \|\nabla\varphi\|_\infty^2)^{\frac{n+1}{4}}$ sufficiently close to one by choice of $\gamma_\beta$, so we need to handle the ratios of the elliptic measure. This can be done in a variety of ways, but we choose to do it directly with the following.
\begin{claim}\label{ssprimeavgest.cl} Let $s > 0$ and suppose $\mu = f \, dx$ for $f \ge 0$, $f \in L^1(\Delta(y,s))$ satisfies 
\begin{equation}\label{rh2mu.eq}
\left(\fint_{\Delta(y,s)} f^2 \, dx \right)^{1/2} \le 2 \fint_{\Delta(y,s)} f \, dx.
\end{equation}
Then for $s' \in ([1- (1/4)]^{1/n}s, s)$,  
\[\frac{\mu(\Delta_s)}{\mu(\Delta_{s'})} \le \frac{1}{1- 2\left(\frac{s^n - (s')^n}{s^n} \right)^{1/2}},\]
where $\Delta_{s'} := \Delta(y, s')$ and  $\Delta_s := \Delta(y, s)$.
\end{claim}
\begin{proof}[Proof of Claim \ref{ssprimeavgest.cl}]
The proof is a direct consequence of \eqref{rh2mu.eq} and H\"older's inequality. Indeed,
\begin{align*}
\mu(\Delta_s \setminus \Delta_{s'}) &= \int_{\Delta_s \setminus \Delta_{s'}} f \, dx \le H^n(\Delta_s \setminus \Delta_{s'})^{1/2} \left(\int_{\Delta_s \setminus \Delta_{s'}} f^2 \, dx\right)^{1/2}
\\& \le  \left| \Delta_s \setminus \Delta_{s'} \right|^{1/2} \left(\int_{\Delta_s} f^2 \, dx\right)^{1/2}
\\ &\le 2\left(\frac{\left| \Delta_s \setminus \Delta_{s'} \right|}{| \Delta_s |}\right)^{1/2}\mu(\Delta_s)
\\ &\le  2\left(\frac{s^n - (s')^n}{s^n} \right)^{1/2}\mu(\Delta_s)
\end{align*}
where we used \eqref{rh2mu.eq} in the second to last line. With this inequality in hand, we easily prove the claim by writing
\[\frac{\mu(\Delta_s)}{\mu(\Delta_{s'})} = 1 +  \frac{\mu(\Delta_s \setminus \Delta_{s'})}{\mu(\Delta_{s'})} \le 1 + 2\left(\frac{s^n - (s')^n}{s^n} \right)^{1/2} \frac{\mu(\Delta_s)}{\mu(\Delta_{s'})}. \]
\end{proof}

The estimate \eqref{smallnessforpullback.eq} allows us to apply Claim \ref{ssprimeavgest.cl} to the measure $\mu = \hm_{\widetilde{A}}^Z$, $s = \delta r$ and $s' = \delta r(1+ \|\nabla \varphi\|_\infty^2)^{-1/2}$, where $s'$ will satisfy the hypothesis of the claim by the smallness of $\gamma_\beta$. This yields the estimate 
\[\frac{\omega^Z_{\widetilde{A}}(\Delta(y, \delta r))}{\omega^Z_{\widetilde{A}}(\Delta(y, \delta r(1+ \|\nabla \varphi\|_\infty^2)^{-1/2}))}
\le \frac{1}{1 - 2\left(1 - (1 + \gamma_\beta^2)^{-n/2}\right)^{1/2}}.\]
This estimate in concert with \eqref{secondtolasthest.eq} and choice of $\gamma_\beta$ sufficiently small gives
\begin{align*}
\frac{\left(\fint_{\Delta_\varphi(y, \delta r)} (h^X_{\tau, \varphi})^2\, d\sigma\right)^{1/2}}{\fint_{\Delta_\varphi(y, \delta r)} (h^X_{\tau, \varphi})^2\, d\sigma} &\le (1 + \widetilde\beta) \frac{(1 +\gamma_\beta^2)^{\frac{n+1}{4}}}{1 - 2\left(1 - (1 + \gamma_\beta^2)^{-n/2}\right)^{1/2}}
\\ & \le (1 + \beta),
\end{align*}
where we recall that $1 + \widetilde \beta \le 1 + \beta/2$.
\end{proof}

\section{Proof of Theorem \ref{VCADcorr.thm}}

In this section we present the proof of Theorem \ref{VCADcorr.thm}. To begin, we require an `almost optimal' version of Lemma \ref{agreeingops.lem} (which is a comparison between elliptic measures of two operators who agree locally), after we impose sufficient flatness on the boundary and H\"older regularity on the coefficients. This almost-optimal comparison will allow us to transfer the almost-optimal reverse H\"older estimate in Proposition \ref{mainlocalization.prop} of the modified coefficient matrix to that of the original matrix, in local regions of sufficiently flat Lipschitz domains. We will also use this almost-optimal comparison to relate the Poisson kernel for a $(\delta, R)$-chord arc domain $\om$ and its localized domain $\widetilde\om(x_0,r)$, which can be approximated by Lipschitz domains from inside and outside. The notation used here can be found at the end of Section \ref{sec:sCAD}.

 Until we get to the proof of Theorem \ref{VCADcorr.thm} we largely follow \cite{MT}, tightening the presentation along the way. We begin with a lemma which ressembles \cite[Lemma 3.7]{MT}. 

\begin{lemma}\label{sharptransferPK.lem} Let $\om$ be a $(\delta,R)$-chord arc domain with $\delta < \delta_n$. There exists a constant $M'$ depending on the dimension so that the following holds. Let $M \ge M'$,  $0< s <   R/M$, $x_0 \in \pom$ and $\xi \in (-s/100, s/100)$ and set
\[\mathcal{C} := \mathcal{C}(x_0,Ms, \vec{n}_{x_0, Ms}, \xi), \]
defined in \eqref{adcyl.eq}, and $\om_1 :=\mathcal{C} \cap \om$.
Suppose that
\begin{itemize}
\item $\om_2$ is a chord arc domain satisfying $\om_2 \cap \mathcal{C}= \om_1$.
\item $A_1$ and $A_2$ are coefficient matrices defined on $\om_1$ and $\om_2$, respectively,  such that $A_1 = A_2$ on $\om_1$. 
\end{itemize}
For $i = 1,2$, let $\hm_i$ be the elliptic measure for the operator in $L_i :=- \div A_i \nabla$ in $\om_i$ and $G_i(X,Y)$ be the Green function for $L_i$ in $\om_1$.  If $X \in B(x_0, 10^{-8}(n+1)^{-1/2} Ms)$, it holds 
\begin{equation}\label{sharpmainaux.eq}
\frac{d\hm_1^X}{d\hm_2^X}(y) = \lim_{Y \to y} \frac{G_1(X,Y)}{G_2(X,Y)},
\end{equation}
for $\hm_2^X$ a.e. $y \in B(X, 100 \dist(X, \pom)) \cap \pom_2$\footnote{We remark that by the choice of $X$, the ball $B(X, 100 \dist(X, \partial\Omega))$ is well contained in the cylinder $\mathcal{C}$ and thus $B(X, 100 \dist(X, \partial\Omega)) \cap \partial\Omega_2 = B(X, 100 \dist(X, \partial\Omega)) \cap \partial\Omega = B(X, 100 \dist(X, \partial\Omega)) \cap \partial\Omega_1$.}, where the limit is taken within $\om_1$. 

Additionally, if $X_0 = x_0 + tMs \vec n_{x_0, Ms}$ for some $t \in \left[\tfrac{1}{4\sqrt{n+1}}, \tfrac{3}{4\sqrt{n+1}}\right]$ then
\begin{equation}\label{sharpmain1.eq}\frac{d\hm_1^{X_0}}{d\hm_2^{X_0}}(y) = \lim_{Y \to y} \frac{G_1(X_0,Y)}{G_2(X_0,Y)},
\end{equation}
for $\hm_2^X$ a.e. $y \in B(x_0, 10s) \cap \pom_2$. Moreover, for every $\epsilon > 0$ there exists $M'' > M'$ depending on $\epsilon$, $n$ and the chord arc constants of $\om_2$ such that if $M \geq M''$ it holds that 
\begin{equation}\label{sharpmain2.eq}
(1 + \epsilon)^{-1} \le \frac{d \hm_1^{X_0}}{d\hm_2^{X_0}}(y)\bigg/\frac{d \hm_1^{X_0}}{d\hm_2^{X_0}}(z) \le (1 + \epsilon)
\end{equation}
for $\hm_2^{X_0}$-a.e. $y,z \in B(x_0, 10s) \cap \partial\Omega_2$.
\end{lemma}
\begin{proof}
Recall that the estimates in Section \ref{sec:sCAD} can be applied for the truncated domain ${\om_1}$ as if it were a chord arc domain itself, provided we work away from $\partial \mathcal{C}$. (See the paragraph before Definition \ref{BVMO.def}.) The constant $M'$ is simply to overcome the shift of the cylinder by $\xi$. First we prove \eqref{sharpmainaux.eq}; the proof for \eqref{sharpmain1.eq} is nearly identical and we omit it. 

Fix $X$ as in the lemma and set $d_X: = \dist(X, \pom)$. We recall the following Riesz formula (see \cite[Lemma 2.25]{HMT})
\begin{equation}\label{Rieszform1.eq}
\int_{\pom_i} \psi \hm_i^X - \psi(X) = -\iint_{\om_i} \langle A_i^T(Y) \nabla_Y G_i(X,Y), \nabla_Y \psi(Y)\rangle \, dY,
\end{equation}
for every $\psi \in C_c^\infty(\RR^{n+1})$ and $X \in \om_i$, $i = 1, 2$.
Applying Lemma \ref{QHC.lem} to the Green's functions $G_1(X, \cdot)$ and $G_2(X, \cdot)$, we have that the limit
\begin{equation}\label{eq:Gratio}
	g(y) := \lim_{Y \to y} \frac{G_1(X,Y)}{G_2(X,Y)} 
\end{equation} 
exists for $\omega_2^X$-a.e. $y \in B(X, 100 d_X) \cap \partial\Omega_2$.
On the other hand, similar to Lemma \ref{agreeingops.lem} (and using the Lebesgue differentiation theorem for Radon measures),
\[\frac{d \hm^X_1}{d\hm^X_2}(y):= \lim_{r \to 0^+} \frac{ \hm^X_1(B(y,r))}{\hm^X_2(B(y,r))}\]
exists for $\hm_2^X$-a.e. $y \in B(X, 100 d_X) \cap \pom_2$.

For $y \in B(X, 100 d_X) \cap \pom_2$ and $r \ll d_X$ we let $\psi_{y,r}(X) := \psi(|X-y|/r)$, where $\psi \in C^\infty_c((-2, 2))$ is radially decreasing with $\psi \equiv 1$ on $[-1, 1]$, so that $|\nabla \psi_{y,r}| \leq C/r$. Let $u_{y,r}^i$, $i = 1,2$, be the variational solution to the Dirichlet problem for $L_i$ in $\om_i$ with boundary data $\psi_{y,r}|_{\pom_i}$. A modification of the standard (`harmonic analysis') argument used to prove the Lebesgue differentiation theorem (for doubling measures) also shows for $\hm_2^X$-a.e. $y \in B(X, 100 d_X) \cap \pom_2$
\begin{equation}\label{uquotientsconvae.eq}
\frac{d \hm^X_1}{d\hm^X_2}(y) = \lim_{r \to 0^+} \frac{\int_{\pom_1} \psi_{y,r}(z) \, d\hm_1^X(z)}{\int_{\pom_2} \psi_{y,r}(z) \, d\hm_2^X(z)} = \lim_{r \to 0^+} \frac{u_{y,r}^1(X)}{u_{y,r}^2(X)}.
\end{equation}
Indeed, to make such an observation one should appeal to the techniques in \cite[Chapter 1]{Stein-HA} and \cite[Chapter 1]{Stein-SIOs} by building weighted maximal function out of the radial function $\psi(|X|)$ using the the fact that $\hm_2^X$ is doubling. More specifically, one can introduce the operator 
\[A_{\psi,r}h(x) := \frac{1}{\int_{\pom_2} \psi_{z,r}(z) \, d\hm_2^X(z)}\int_{\pom_2} \psi_{z,r}(z)h(z) \, d\hm_2^X(z), \]
and dominate it by an associated (local) maximal operator. Then note that, for sufficiently small $r$ the quotient inside the limit in the middle term of \eqref{uquotientsconvae.eq} is exactly $A_{\psi,r}h(y)$ with $h = \tfrac{d \hm^X_1}{d\hm^X_2}|_{B(X, 200 d_X) \cap \pom_2}$, which is an $\hm_2^X$-integrable function by Lemma \ref{agreeingops.lem}. Finally, use \cite[Chapter 1: Theorem 1]{Stein-HA}\footnote{Note that condition (iv) therein is explicitly stated in the proof as not necessary and an analog of the required covering lemma holds in our setting.} to provide the (local) weak-type bounds for the maximal operator and follow the ideas in \cite[Chapter 1]{Stein-SIOs}. We leave the details to interested readers.

Now fix $y \in B(X, 100 d_X) \cap \pom_2$ such that \eqref{eq:Gratio} and \eqref{uquotientsconvae.eq} hold. Denote $u^i_{r}:= u^i_{y,r}$ and $\psi_r(X) := \psi_{y,r}(X) = \psi(|X-y|/r)$. The fact that $r \ll d_X$ implies that $X \notin B(y,2r) = \supp \psi_r$ and $\Omega_1 = \Omega_2$ on $\supp \psi_r$. Hence by the Riesz formula \eqref{Rieszform1.eq}, we have 
\begin{equation}\label{Rieszform2.eq}
\begin{split}
u_{r}^i(X) = \int_{\pom_i} \psi_r ~\hm_i^X &= -\iint_{\om_i} \langle A_i^T(Y) \nabla_Y G_i(X,Y), \nabla_Y \psi_r(Y)\rangle \, dY
\\ &= -\iint_{\om_1} \langle A_1^T(Y) \nabla_Y G_i(X,Y), \nabla_Y \psi_r(Y)\rangle \, dY, \quad i = 1,2,
\end{split}
\end{equation}
where we used that  $A_1= A_2$ on $\om_1$ and $\Omega_2 = \Omega_1$ on $\supp \psi_r$.

Notice that by the CFMS estimates and Lemma \ref{agreeingops.lem}, we have $g(y) \approx 1$. We are going to show that 
\begin{equation}\label{differenceestquo.eq}
\left|u_{r}^1(X) - g(y) u_{r}^2(X)\right| \le C \left(\frac{r}{d_X}\right)^\mu g(y) u_{r}^2(X),
\end{equation}
where $\mu$ is the H\"older exponent as in Lemma \ref{QHC.lem}.
Using that $g(y) \approx 1$ we may divide by the positive constant $u_r^2(X)$ on both sides to obtain
\[\left|\frac{u_{r}^1(X)}{u_{r}^2(X)} - g(y) \right| \le C' \left(\frac{r}{d_X}\right)^\mu\]
and letting $r \to 0$ shows 
\[\lim_{r \to 0^+} \frac{u_{r}^1(X)}{u_{r}^2(X)}  = g(y).\]
Then recalling \eqref{eq:Gratio} and \eqref{uquotientsconvae.eq}, the desired equality \eqref{sharpmainaux.eq} follows. Therefore it suffices to show \eqref{differenceestquo.eq} and we do this now.

Using \eqref{Rieszform2.eq} and carefully noting that $g(y)$ is a {\it fixed scalar} since $y$ is fixed, we may use the boundedness of $A_1$ and the properties of $\psi_r$ to conclude
\begin{align}
&\left|u_{r}^1(X) - g(y) u_{r}^2(X)\right| \nonumber \\
& \quad = \left|\iint_{\om_1 \cap B(y,2r)} \left\langle A_1^T\nabla_Z \left[G_1(X,Z) - g(y) G_2(X,Z) \right],  \nabla_Z \psi_r(Z) \right\rangle \, dZ  \right| \nonumber
\\ &\quad \lesssim \left(\iint_{\om_1 \cap B(y,2r)} \left|\nabla_Z \left[G_1(X,Z) - g(y) G_2(X,Z) \right] \right|^2 dZ \right)^{1/2} \nonumber \\
& \qquad \qquad \times \left(\iint_{\om_1 \cap B(y,2r)}|\nabla_Z \psi_r|^2 dZ \right)^{1/2} \nonumber
\\ &\quad \lesssim r^{\frac{n-1}{2}} \left(\iint_{\om_1 \cap B(y,2r)} \left|\nabla_Z \left[G_1(X,Z) - g(y) G_2(X,Z) \right] \right|^2 dZ \right)^{1/2}.
\end{align}
Again noting that $g(y)$ is a fixed scalar, we may use that $U(Z) = G_1(X,Z) - g(y) G_2(X,Z)$ is a solution to $L_1^T U = 0$ in $\om_1 \cap B(y, 10r)$ which vanish on $\pom_1 \cap B(y, 10r)$ so that we may apply the boundary Caccioppoli inequality to the function $U$. It follows that
\begin{equation}\begin{split}
\left|u_{r}^1(X) - g(y) u_{r}^2(X)\right|  &\lesssim r^{n-1} \left(\fiint_{\om_1 \cap B(y,4r)}|G_1(X,Z) - g(y) G_2(X,Z)|^2 dZ\right)^{1/2} \\
& \lesssim r^{n-1} \left(\fiint_{\om_1 \cap B(y,4r)}\left|\frac{G_1(X,Z)}{G_2(X,Z)} - g(y)\right|^2 |G_2(X,Z)|^2 dZ \right)^{1/2} \\
& \lesssim r^{n-1} \sup_{ \om_1 \cap B(y,4r)}  \left|\frac{G_1(X,\cdot )}{G_2(X,\cdot )} - g(y)\right|  \sup_{\om_1 \cap B(y,4r)} G_2(X,\cdot).
\end{split}
\end{equation}
Next, we use Lemma \ref{QHC.lem}, the Carleson estimate (Lemma \ref{cscarlesonest.lem}) to obtain
\begin{align*}
\left|u_{r}^1(X) - g(y) u_{r}^2(X)\right| &\lesssim \left(\frac{r}{d_X}\right)^\mu g(y) r^{n-1} G_2(X,\mathcal{A}((y,r))
\\ & \lesssim \left(\frac{r}{d_X}\right)^\mu g(y) \hm^X_2(B(y,r))
\end{align*}
where we used the CFMS estimate (Lemma \ref{CFMS.lem}) in the second line. Finally, using the local doubling of $\hm^X_2$ we see that $u_r^2(X) \approx \hm^X_2(B(y,r))$, which along with the estimate above yields the estimate \eqref{differenceestquo.eq}. As we had reduced matters to proving \eqref{differenceestquo.eq}, this shows \eqref{sharpmainaux.eq}.

As remarked above, \eqref{sharpmain1.eq} has the same proof as \eqref{sharpmainaux.eq}. To obtain \eqref{sharpmain2.eq} we use  Lemma \ref{QHC.lem} to deduce that the function 
\[g(y) := \lim_{Y \to y} \frac{G_1(X_0,Y)}{G_2(X_0,Y)}\]
satisfies for $y, z \in B(x_0, 10s)\cap \pom$ 
\[|g(y) - g(z)| \lesssim \left(\frac{|y- z|}{d_{X_0}}\right)^\mu g(x_0) \lesssim \left(\frac{s}{Ms}\right)^\mu g(z),\]
where we denote $d_{X_0} := \dist(X_0, \partial\Omega)$.
Thus, \eqref{sharpmain2.eq} readily follows from \eqref{sharpmain1.eq} provided we choose $M''$ sufficiently large.
\end{proof}

\begin{remark}\label{observationforDKP.rmk}
Notice that Lemma \ref{sharptransferPK.lem}, in fact, shows that $\tfrac{d\hm_1^{X_0}}{d\hm_2^{X_0}}(\cdot)$ is locally H\"older continuous with quantitative estimates. This is simply a consequence of the fact that $\tfrac{d\hm_1^{X_0}}{d\hm_2^{X_0}}(y) = g(y)$, where $g(y)$ is as in the proof of Lemma \ref{sharptransferPK.lem}, and $g$ has these estimates by Lemma \ref{QHC.lem}.
\end{remark}

Combining the Lemma \ref{sharptransferPK.lem} with Proposition \ref{mainlocalization.prop} we obtain the following.
\begin{proposition}\label{Smalllipcyloptimal.lem}
Let $\epsilon \in (0,1)$. There exist positive constants $\gamma$, $\tau$ and $\widetilde M \ge M^{''}$ ($M^{''}$ is from Lemma \ref{sharptransferPK.lem}) depending on $\epsilon$ and allowable constants such that the following holds. 

Assume $L = -\div A \nabla$ is a divergence form elliptic operator with $\Lambda$-elliptic coefficients $A$ satisfying the H\"older condition \eqref{cond:Holder}, and $\varphi : \rn \to \re$, $\varphi(0) = 0$, is Lipschitz with $\|\nabla \varphi\|_\infty \le \gamma$. Then for all $M \ge \widetilde M$, $s \in (0, \tau/M)$ and $\xi \in [-s/200, s/200]$, the Poisson kernel, $h$, for the operator $L$ in the domain
\[\widetilde{\om}_{Ms,\xi} = \widetilde{\om}_{\varphi, Ms,\xi} : =\{(x,t): x \in \rn, t > \varphi(x)\} \cap \C(0, Ms, e_{n+1}, \xi)\]
satisfies
\begin{equation}\label{RH2optSmalllipcyl.eq}
\frac{\left(\fint_{B(y,r)}(h^{X})^2\, d\sigma(z)\right)^{1/2}}{\fint_{B(y,r)}h^{X}\, d\sigma(z)} \le (1 + \epsilon)^3,
\end{equation}
for all $y \in B(0, 5s) \cap \Gr(\varphi)$ and $r \in (0, 5s)$. Here $\sigma := \HH^n|_{\partial\widetilde{\om}_{Ms,\xi}}$ and $X = (\{0\}^n, Mts)$ for some $t \in \left[\tfrac{1}{4\sqrt{n+1}}, \tfrac{3}{4\sqrt{n+1}}\right]$.
\end{proposition}
\begin{proof}
We first fix all the constants, the reason for which will become clear shortly. Let $\delta_\epsilon, \gamma_\epsilon, \tau_\epsilon$ be constants from Proposition \ref{mainlocalization.prop} (using $\beta=\epsilon$), and let $M^{''}$ be the constant from Lemma \ref{sharptransferPK.lem}.
We set $\tau = \tau_\epsilon$ and $\widetilde M = 20 \sqrt{n+1}\max\{1/\delta_\epsilon, M^{''},1\}$. Finally, we set $\gamma = \min\{ \gamma_\epsilon, \delta_n,\gamma_n\}$, where $\gamma_n$ is chosen so that
\[\C(0, s', e_{n+1}, \xi) \cap \Omega_\varphi \subset \, R_{Q_{s'}, \varphi}\]
for all $s' > 0$ and $\xi \le s'/100$. In particular, the inclusion holds for $\xi \le s'/(M^{''}100)$. 
For any Lipschitz function $\varphi$ whose Lipschitz constant is bounded by $\gamma$, the domain 
\[\om := \om_\varphi := \{(x,t): x \in \rn, t > \varphi(x)\} \]
clear is a $(\delta_n,\infty)$-chord arc domain.
Let $M \ge \widetilde M$ and $s \in (0, \tau/M)$ be arbitrary. Then $\widetilde{\om}_{Ms,\xi}$ may play the role of $\om_1$ in Lemma \ref{sharptransferPK.lem} (note $x_0 = 0$ here). 
Indeed, while $\rn \times \{0\}$ may not be the plane that minimizes the bilateral distance, the graph of $\varphi$ is $\delta_n$-flat at scale $Ms$ with respect to this plane. 

Let $y\in B(0,5s)\cap \Gr(\varphi)$ and $r\in (0,5s)$ be arbitrary. If we write $y=(y',\varphi(y'))$ with $y'\in \RR^n$, then $|y'| < 5s$.
Notice that $\Phi$, the `flattening map' for $\varphi$, fixes $X=(0, Mts)$ and that 
\[|X- y'| \le \sqrt{|y'|^2 + (Mts)^2} \leq  2Mts\]
by the choice of $M$,
so $X \in \Phi^{-1}(D(y',Mts))$. Moreover, by our choice of $M \ge 20\sqrt{n+1}/\delta_\epsilon$ we have
\[\delta_\epsilon \cdot Mts \geq 5s > r.\]
Thus by Proposition \ref{mainlocalization.prop},
\[\frac{\left(\fint_{B(y,r)}(h_{\tau}^{X})^2\, d\sigma(z)\right)^{1/2}}{\fint_{B(y,r)}h_{\tau}^{X}\, d\sigma(z)} \le 1 + \epsilon,\]
where $h_\tau$ is the Poisson kernel for the operator $L_{\tau,\varphi}$ in $\om = \Omega_{\varphi}$.
On the other hand, the elliptic matrices $A_1 = A_{\tau,\varphi}$ and $A_2 = A$ agree on the cylinder $\C(0,Ms, \vec e_{n+1}, \xi)$ by the inclusion 
\[ \C(0, Ms, \vec e_{n+1}, \xi) \subset R_{Q_{Ms}, \varphi} \subset R_{Q_\tau, \varphi}. \] Therefore applying \eqref{sharpmain2.eq} from Lemma \ref{sharptransferPK.lem} we obtain \eqref{RH2optSmalllipcyl.eq}.
\end{proof}

The following corollary follows immediately by the theory of weights:
\begin{corollary}\label{optAinfLip.cor} 
Let $\epsilon \in (0,1)$. There exist constants\footnote{Here $M''$ is from Lemma \ref{sharptransferPK.lem}.} $M^* \ge M'' \ge 1/\epsilon$ and $\gamma'_\epsilon$, $\tau'_\epsilon$ depending on $\epsilon$ and allowable constants such that the following holds.

Assume $L = -\div A \nabla$ with $\Lambda$-elliptic coefficients satisfying the H\"older condition \eqref{cond:Holder}, and $\varphi : \rn \to \re$, $\varphi(0) = 0$, is Lipschitz with $\|\nabla \varphi\|_\infty \le \gamma'_\epsilon$. Then for all $M \ge M^* $, $s \in (0, \tau'_\epsilon/M)$ and $\xi \in [-s/200, s/200]$ the elliptic measure, $\widetilde\hm$, for the operator $L$ in the domain
\[\widetilde{\om}_{Ms,\xi} = \widetilde{\om}_{\varphi, Ms,\xi} : =\left\{(x,t): x \in \rn, t > \varphi(x)\right\} \cap \C(0, Ms, e_{n+1}, \xi)\]
satisfies
\begin{equation}\label{optAinfLip.eq}
(1 + \epsilon)^{-1}\left(\frac{\sigma(E)}{\sigma(\Delta)}\right)^{1+\epsilon} \le \frac{\widetilde\hm^X(E)}{\widetilde\hm^X(\Delta)} \le (1 + \epsilon)\left(\frac{\sigma(E)}{\sigma(\Delta)}\right)^{1-\epsilon}, \quad \forall E \subset \Delta \subset B(0, 2s),
\end{equation}
where $\sigma := \HH^n|_{\partial\widetilde{\om}_{M s,\xi}}$ and $X = (\{0\}^n, M ts)$ for any $t \in \left[\tfrac{1}{4\sqrt{n+1}}, \tfrac{3}{4\sqrt{n+1}}\right]$ and $\Delta$ is an arbitrary surface ball.
\end{corollary}
\begin{proof}
Proposition \ref{Smalllipcyloptimal.lem} says that $\widetilde\omega^X = h^X d\sigma \in B_2(\sigma)$, and the $B_2$ constant can be made as close to $1$ as possible. For any $\epsilon \in (0,1)$, we choose constants appropriately in the Proposition, so that the $B_2$ constant (i.e. the right hand side of \eqref{RH2optSmalllipcyl.eq}) is bounded above by $\exp((\epsilon/K)^2)$, where $K$ is a dimensional constant as in \cite[Corollary 8]{Korey}. Then by \cite[Corollary 8]{Korey}
	\begin{equation}\label{eq:weightopt}
		\widetilde \omega^X \in A_{1+\epsilon}(\sigma) \text{ with constant } 1+\epsilon, \quad \widetilde \omega^X \in B_{\frac{1}{\epsilon}}(\sigma) \text{ with constant } 1+ \epsilon. 
	\end{equation} 
	(See \cite[Section 3.1]{Korey} for the definition of $A_p(\sigma)$ weight.)
	Both estimates in \eqref{optAinfLip.eq} then follow easily from \eqref{eq:weightopt} and H\"older's inequality.
\end{proof}

We are now ready to give the Proof of Theorem \ref{VCADcorr.thm}. Here we diverge a bit from the techniques in \cite{MT,KT-Duke}, opting for an approach that largely avoids the use of the Poisson kernel and instead works with the elliptic measure more directly. This avoids some of the issues that arise in \cite{MT}.

\begin{proof}[Proof of Theorem \ref{VCADcorr.thm}]
Let $\om$ be a vanishing chord arc domain. Recall that this means for all $\delta \in (0, \delta_n]$, $\om$ is a $(\delta, R_\delta)$ chord arc domain for some $R_\delta > 0$ (see Remarks \ref{ADR.Deltarforallr.rmks}). We set $\hm := \hm^{X_0}$ to be the elliptic measure associated to $L$ for the domain $\om$ with fixed pole $X_0\in \Omega$. 

We first make the following claim which gives a much rougher estimate than what we will produce in the end. 
\begin{claim}\label{orighmisAinfty.cl}
The elliptic measure $\hm$ is locally an $A_\infty$ weight, that is, there exist $\tau_0$ and constants $C_0$, $\theta$ depending on allowable constants such that 
\begin{equation}\label{orighmisAinfty.eq}
\frac{\hm(E)}{\hm(\Delta)} \le C_0 \left(\frac{\sigma(E)}{\sigma(\Delta)}\right)^\theta, \forall E \subset \Delta,
\end{equation}
where $\Delta$ is any surface ball with radius less than or equal to $\tau_0$, that is, $\Delta = B(x, r) \cap \pom$ with $x \in \pom$ and $r \in (0, \tau_0]$.
\end{claim}

To state a more precise estimate, we first fix some constants. For every $\beta\in (0,1)$, we fix a constant $\epsilon\in (0,\beta/2)$ so that
\begin{equation}\label{goodepschoice.eq}
(1+\epsilon) \left( \frac12 + \epsilon \right)^{1-\epsilon} \leq \frac12 (1+\beta),
\end{equation}
with the intention of using the estimate obtained in Corollary \eqref{optAinfLip.cor}. Let $M^{''}>0$ be the constant found in Lemma \ref{sharptransferPK.lem}, and $M^* , \gamma'_{\epsilon}, \tau'_{\epsilon}>0$ be constants found in Corollary \ref{optAinfLip.cor}, and set $M= \max\{M^* ,1/\beta\}$. Let $\delta\in (0, \delta_n]$ be sufficiently small, depending on $\beta, \epsilon, \gamma'_{\epsilon}, M^* $ and allowable constants, and recall $\Omega$ is a $(\delta, R_{\delta})$-chord arc domain.

\begin{claim}\label{halfalmosthalf.cl}

For any $x_0 \in \pom$ and $s$ sufficiently small satisfying
\[ s M \le \min\{\tau'_\epsilon, \tau_0, R_\delta/10, \dist(X_0, \pom)/5)\}, \] 
if $E \subset B(x_0, s) \cap \pom =: \Delta_0$ satisfies
\[\frac{\sigma(E)}{\sigma(\Delta_0)} = \frac12 \]
then
\begin{equation}\label{tmp:omtildehalf}
	\frac{\widetilde\hm(E)}{\widetilde\hm(\Delta_0)} = \frac12 + C\beta,
\end{equation} 
where $C$ is a constant depending on the dimension. 
Here $\sigma$ is the surface measure for the domain $\widetilde\om(x_0, M s)$, which agrees with that of $\om$  on $B(x_0, s)$;
and $\widetilde\hm$ is the elliptic measure for $\widetilde\om(x_0, M s)$ with pole at $X_1 := x_0 + \tfrac{1}{2\sqrt{n+1}}M s \, \vec n_{x_0, M s}$. 
\end{claim}

Let us take this claim for granted momentarily and see how to conclude the theorem. Assuming \eqref{tmp:omtildehalf}, we want to change to domain from $\widetilde\Omega(x_0, Ms)$ to $\Omega$, and change the pole from $X_1$ to $X_0$, in the hope to get a similar estimate for $\omega = \omega^{X_0}$.
By Lemma \ref{sharptransferPK.lem} and the choice of parameters,
\[ \frac{\omega^{X_1}(E)}{\omega^{X_1}(\Delta_0)} \leq (1+\epsilon)^2 \, \frac{\widetilde\hm(E)}{\widetilde\hm(\Delta_0)} \leq (1+\beta)^2 \left( \frac12 + C\beta \right). \]
Here we remind the reader that the pole of $\widetilde\hm$ is at $X_1$.
To change the pole of $\omega$ from $X_1$ to $X_0$, we use Lemma \ref{Kernelcompare.lem} in a similar manner to Lemma \ref{goodpolechange.lem}. Note that the pole $X_1$ is roughly at distance $M s \ge s/\beta$ from the center of $\Delta_0$ a surface ball of radius $s$, and $\dist(X_1, \pom) \leq Ms \le \dist(X_0, \pom)/5$. This allows us to use Lemma \ref{Kernelcompare.lem}  to say 
\begin{equation}
	\frac{\omega(E)}{\omega(\Delta_0)} \leq (1+C\beta^{\mu}) \, \frac{\omega^{X_1}(E)}{\omega^{X_1}(\Delta_0)} \leq \frac12 + C' \beta^\mu.
\end{equation}
To sum up, this combined with Claim \ref{halfalmosthalf.cl} says that for any $x_0\in \pO$ and $s$ sufficiently small (depending on $\beta$),
\begin{equation}\label{cond:d}
	\frac{\sigma(E)}{\sigma(\Delta_0)} = \frac12 \text{ for a Borel set } E \subset \Delta_0 \implies \frac{\omega(E)}{\omega(\Delta_0)} \leq \frac12 + C' \beta^{\mu}. 
\end{equation}

Among other things \cite[Theorem 10]{Korey} establishes the equivalence between the above statement \eqref{cond:d} (that is, condition (e) in \cite[Theorem 10]{Korey}) and the fact that the $B_2$ constant of $\omega$ is close to $1$ (that is, condition (c) in \cite[Theorem 10]{Korey}). We in fact need a quantitative local estimate proved in \cite[Theorems 8 and 9]{Korey}, to see that \eqref{cond:d} implies
\begin{equation}\label{cond:c}
	\left( \fint_{\Delta(x_0,s)} k \, d\sigma \right)^2 \leq (1+C' \beta^{\mu/2}) \left( \fint_{\Delta(x_0,s)}  k^{1/2} \, d\sigma \right),
\end{equation}
where $k = d \omega/d\sigma$ is the Poisson kernel. For the sake of self-containment we include proof of \eqref{cond:c} in Appendix \ref{Koreyexplain.sect}, see Lemma \ref{Koreyexplain.lem}. Since $\beta$ can be chosen arbitrarily small, again by \cite[Theorem 8]{Korey} we conclude that $\log k \in VMO$.

It remains to prove Claims \ref{orighmisAinfty.cl} and \ref{halfalmosthalf.cl}. We start with Claim \ref{halfalmosthalf.cl}, taking Claim \ref{orighmisAinfty.cl} for granted.
\begin{proof}[Proof of Claim \ref{halfalmosthalf.cl}]
Without loss of generality we may assume $x_0 = 0$ and $\vec n_{x_0, Ms} = \vec e_{n+1}$. For notational convenience we set $\widetilde{\om} := \widetilde\om(0, Ms)$.

We make use of the Semmes decomposition to approximate $\Omega$ by Lipschitz domains,  see \cite{Semmes-SmallCAD1, KT-Duke}. To be precise by \cite[Lemma 5.1]{KT-Duke} (replacing $\delta^4$ by $\delta$), when $\delta$ is chosen sufficiently small, there exist two Lipschitz functions $\varphi^\pm : \rn \to \re$ whose graphs approximate $\pom$ from inside(+) and the outside(-). Set $\Gamma_\pm := \Gr(\varphi^\pm)$
and
\[\om^\pm:= \{(x,t) \in \ree: x \in \rn, t >\varphi(x)\}\]
then the following properties hold:
\begin{enumerate}
\item $\|\nabla \varphi^\pm\|_\infty \le C_1 \delta^{1/4}$
\item $D[\Gamma^\pm \cap \overline{B(0, Ms)}; \pom \cap \overline{B(0, Ms)}] \le C_1 \delta^{1/4}Ms$
\item $\HH^n\left(B(0,Ms) \cap \pO \setminus \Gamma_{\pm}  \right) \leq c_1 \exp(-c_2 \delta^{-1/4}) \omega_n (Ms)^n $ 
\item $\om^+ \cap \C(0,Ms) \subseteq \om \cap \C(0,Ms) \subseteq \om^- \cap \C(0,Ms)$,
\end{enumerate}
where $C_1, c_1, c_2$ are all positive constants depending on only on dimension. Now we let $\widetilde \hm_\pm$ be the elliptic measure (for $L$) in the domains $\widetilde \om^\pm := \om^\pm \cap \C(0,Ms)$ with pole at $X_1 : = \tfrac{1}{2\sqrt{n+1}}M se_n$ and $\sigma_\pm$ be the surface measure for $\widetilde \om^\pm$.
We choose $\delta$ sufficiently small, to ensure that $C_1 \delta^{1/4}Ms < s/200$ and $C_1 \delta^{1/4} < \gamma'_\epsilon$. Thus by (2) $|\varphi^\pm(0)| = |\varphi^\pm(0) - 0 |< s/200$. (This was the reason for the using parameter $\xi$ to shift the cylinder appearing in  Corollary \ref{optAinfLip.cor} above.) 
Applying Corollary \ref{optAinfLip.cor} to Lipschitz domains $\widetilde\Omega^{\pm}$, we have 
\begin{equation}\label{optAinflipapprox.eq}
(1 + \epsilon)^{-1}\left(\frac{\sigma_\pm(F)}{\sigma_\pm(\Delta)}\right)^{1+\epsilon} \le \frac{\widetilde\hm_\pm(F)}{\widetilde\hm_\pm(\Delta)} \le (1 + \epsilon)\left(\frac{\sigma_\pm(F)}{\sigma_\pm(\Delta)}\right)^{1-\epsilon}, \quad \forall F \subset \Delta \subset B((0,\varphi^{\pm}(0)), 2s),
\end{equation}
where $\Delta$ is a surface ball of the form $B(x,r) \cap \pom^\pm$. (We will not use $\widetilde \om^+$ for the proof of this claim, but it is used in proving Claim \ref{orighmisAinfty.cl}.)

Now let $E \subset \Delta_0 = B(0,s) \cap \pom$ be such that $\sigma(E) = (1/2)\sigma(\Delta_0)$. Set $\Delta_- := B((0,\varphi^-(0)), s) \cap \Gamma_-$. Then 
\begin{align*}
\frac{\sigma_-(E \cap \Gamma_-)}{\sigma_-(\Delta_-)} &= \frac{\sigma(E\cap \Gamma_-)}{\sigma_-(\Delta_-)} \leq \frac{\sigma(E)}{\sigma(\Delta_0)} \, \frac{\sigma(\Delta_0)}{\sigma_-(\Delta_-)} \leq  \frac{1}{2} \frac{(1+\delta)\hm_ns^n}{(1 + C_1^2 \delta^{1/2})^{-n/2} \hm_n s^n } \le \frac{1}{2} + \epsilon
\end{align*}
by choice of $\delta$ sufficiently small, where we used property (1) of the Semmes decomposition, \eqref{projectioninclusion.eq} for the lower bound on $\sigma_-(\Delta_-)$ and for the upper bound on $\sigma(\Delta_0)$ we used the definition of $(\delta,R)$-chord arc domain (Definition \ref{deltaVCAD.def}). It follows from \eqref{optAinflipapprox.eq} and the choice \eqref{goodepschoice.eq} of $\epsilon$ that 
\begin{equation}\label{almosthalfhmminus.eq}
\frac{\widetilde \hm_-(E \cap \Gamma_-)}{\widetilde \hm_-(\Delta_-)} \le \frac12(1 + \beta).
\end{equation}

We now estimate 
\begin{equation}\label{localhalfstart.eq}
\frac{\widetilde \hm(E)}{\widetilde \hm(\Delta_0)} = \frac{\widetilde \hm(E\cap \Gamma_-)}{\widetilde \hm(\Delta_0)}  + \frac{\widetilde \hm(E\setminus \Gamma_-)}{\widetilde \hm(\Delta_0)} =: I + II.
\end{equation}
Notice that
\begin{align}
	\frac{\sigma(\Delta_0 \setminus \Gamma_-)}{\sigma(\Delta_0)} & \leq \frac{\sigma(M\Delta_0\setminus \Gamma_-)}{\sigma(\Delta_0) } \nonumber \\
	&\leq \frac{c_1 \exp(-c_2\delta^{-1/4}) \omega_n (Ms)^n}{(1-\delta) \omega_n s^n} \nonumber \\
	& = \frac{c_1 \exp(-c_2\delta^{-1/4}) M^n }{1-\delta},\label{tmp:smrsd}
\end{align}
by property (3) and the definition of $(\delta,R)$-chord arc domain (Definition \ref{deltaVCAD.def}).
Therefore by a simple pole change argument (or the shaper version, Lemma \ref{Kernelcompare.lem}), Lemma \ref{agreeingops.lem} and the fact that $\hm$ is $A_\infty$ (Claim \ref{orighmisAinfty.cl}) we obtain 
\begin{align*}
	II 
	\le \frac{\widetilde \hm(\Delta_0 \setminus \Gamma_-)}{\widetilde \hm(\Delta_0)} \lesssim \frac{\omega(\Delta_0\setminus \Gamma_-)}{\omega(\Delta_0)} \lesssim \left( \frac{\sigma(\Delta_0 \setminus \Gamma_-)}{\sigma(\Delta_0)} \right)^{\theta} 
	< \beta
\end{align*} 
by choice of $\delta$ small. It remains to treat term $I$, which is a matter of `removing the minus sign' in \eqref{almosthalfhmminus.eq}. Since $\widetilde\Omega \subset \widetilde\Omega^-$, by the maximal principle 
\begin{align*}
\frac{\widetilde \hm(E\cap \Gamma_-)}{\widetilde \hm(\Delta_0)}  \leq \frac{\widetilde \hm_-(E \cap \Gamma_-)}{\widetilde \hm_-(\Delta_-)}  \frac{\widetilde \hm_-(\Delta_-)}{\widetilde \hm(\Delta_0)} \le \frac12 (1 + \beta)  \frac{\widetilde \hm_-(\Delta_-)}{\widetilde \hm(\Delta_0)}.
\end{align*}
Thus, we may reduce proving Claim \ref{halfalmosthalf.cl} to the estimate
\begin{equation}\label{halfalmosthavered.eq}
\frac{\widetilde \hm_-(\Delta_-)}{\widetilde \hm(\Delta_0)} \le (1 + \beta)^2
\end{equation}
and we do this now. 

Notice $|X_1 - 0| \approx Ms \approx |X_1 - (\varphi^-(0),0)|$ so that Bourgain's estimate (Lemma \ref{Bourgain.lem}) and the Harnack inequality yield the estimate
\begin{equation}\label{cepsLB.eq}
c \le \widetilde \hm_-(\Delta_-),\widetilde \hm(\Delta_0) \le 1,
\end{equation}
where $c$ is a constant depending on $\epsilon$, by way of $M$. Additionally, by \eqref{optAinflipapprox.eq} we have that
\begin{align*}
&\frac{\widetilde \hm_-(\Delta_- \setminus (1 - \delta^{1/16})\Delta_- )}{\widetilde \hm_-(\Delta_-)} 
\\ &\quad \le (1 + \epsilon)\left(\frac{\sigma_-(\Delta_- \setminus (1 - \delta^{1/16})\Delta_- )}{\sigma_-(\Delta_-)} \right)^{1-\epsilon}
\\ &\quad \le (1 + \epsilon)\left(\frac{ 1 - (1 + C_1^2 \delta^{1/2})^{-n/2}(1- \delta^{1/16})^n}{(1 + C_1^2 \delta^{1/2})^{-n/2}} \right)^{1-\epsilon} \\
& \quad \leq \frac{\beta}{1+\beta}
\end{align*}
by choice of $\delta$, where we used the property $(1)$ and the inclusion \eqref{projectioninclusion.eq}. Therefore 
\[\widetilde \hm_-(\Delta_-) \le (1 + \beta)\widetilde \hm_-((1 - \delta^{1/16})\Delta_-)\]
and to prove \eqref{halfalmosthavered.eq}, it is enough to show
\begin{equation}\label{halfalmosthavedoubred.eq}
\widetilde \hm_-((1 - \delta^{1/16})\Delta_-) \le (1 + \beta)\widetilde \hm(\Delta_0).
\end{equation}

Let us now assume that $\delta$ is small enough so that $C_1 \delta^{1/4}Ms < \delta^{1/8}s$. Let $x \in \partial \widetilde \om \setminus \Delta_0 =  \partial \widetilde \om \setminus B(0,s)$, then either $x \in \pO \cap \overline{B(0,Ms)} \setminus B(0,s)$ or $x \in \partial \C(0, Ms)$. 
In the first case, the choice of $\delta$ and the property $(2)$ guarantee that there exists $\hat x \in \Gamma_- \cap \overline{B(0,Ms)}$ such that $|x - \hat x| < \delta^{1/8}s$, and thus 
\[ |\hat x - (\varphi^-(0),0)| \ge |x| - |x- \hat x| - |(\varphi^-(0),0)|  \ge (1 - 2\delta^{1/8})s. \]
In particular, 
\[ B\left(\hat x, \frac{\delta^{1/16}}{2} \right) \cap (1- \delta^{1/16})\Delta_- = \emptyset. \]
This means we may use the H\"older continuity of solutions vanishing at the boundary (Lemma \ref{HCatbdry.lem}) to yield the estimate 
\[\widetilde \hm_-^x((1- \delta^{1/16})\Delta_-) \le C\left(\frac{\delta^{1/8}}{\delta^{1/16}}\right)^{\mu} \le C\delta^{\mu/16},\]
where $\mu$ is from Lemma \ref{HCatbdry.lem}.
In the second case, that is, when $x \in \partial \C(0, Ms) \cap  \partial \widetilde\om$, then $x \in \partial \C(0, Ms) \cap  \partial \widetilde\om^-$ since $\Omega \cap \C(0,Ms) \subset \Omega^- \cap \C(0,Ms)$. 
Hence $\widetilde \hm_-^x((1- \delta^{1/16})\Delta_-) = 0$.
In either case, we have that  
\[\widetilde \hm_-^x((1- \delta^{1/16})\Delta_-) \le C\delta^{\mu/16}\]
whenever $x \in \partial \widetilde \om \setminus \Delta_0$.
Therefore by the maximum principle (and that $\widetilde \hm_-^x((1- \delta^{1/16})\Delta_-) \le 1$ for $x  \in \Delta_0$) we have 
\begin{align*}
\widetilde \hm_-((1- \delta^{1/16})\Delta_-) &\le  \widetilde \hm(\Delta_0) + C\delta^{\mu/16}
\\ & \le (1 + (C/c) \delta^{\mu/16})\widetilde \hm(\Delta_0)
\\ & \le (1 + \beta)\widetilde \hm(\Delta_0)
\end{align*}
by choice of $\delta$ sufficiently small, where we used \eqref{cepsLB.eq} in the second line. This shows \eqref{halfalmosthavedoubred.eq} and the claim follows.
\end{proof}

We are left with proving Claim \ref{orighmisAinfty.cl}, which will be a consequence of the maximum principle, the Semmes decomposition above and the theory of weights. The proof is essentially contained \cite{DJ}, but we include a proof that tracks the constants carefully.
 
\begin{proof}[Proof of Claim \ref{orighmisAinfty.cl}] 
By the work of Coifman and Ferfferman \cite{CF-weights}, or rather its (local) generalization to spaces of homogeneous type, to prove the claim it is enough to show the following. For any $\eta \in (0,1)$, there exists $\gamma = \gamma(\eta)\in (0,1)$ such that for every surface ball $\Delta_0 = B(x,s)\cap \pO$ with $x \in \pom$ and $s \in (0, \tau_0]$, if $E \subset \Delta_0$ is a Borel set satisfying $\sigma(E) \ge \gamma \sigma(\Delta_0)$, then $\hm(E) \ge \eta \hm(\Delta_0)$. Here $\tau_0>0$ is a fixed constant whose value is to be specified later. Without loss of generality we assume $x=0$.

Similar to the proof of Claim \ref{halfalmosthalf.cl}, the idea here is also to use the elliptic measure of Lipschitz domain (this time we use $\Omega^+$ instead of $\Omega^-$) to estimate $\omega$. For that purpose we just need a crude version of Corollary \ref{optAinfLip.cor}. Let $\epsilon = 1/2$ be fixed, thus we fix the constants $M^* , \gamma'_\epsilon, \tau'_\epsilon$ accordingly. Set $M= M^* $. We fix $\delta \in (0, \delta_n]$ satisfying $C_1 \delta^{1/4} \leq \gamma'_\epsilon$, then there exists $R_\delta>0$ such that $\Omega$ is a $(\delta, R_\delta)$-chord arc domain. Let $\tau_0 = \min\{R_\delta, \tau'_\epsilon\}/M$ and $s\in (0,\tau_0]$ be arbitrary. Then $\Omega$ has Semmes decomposition in $B(0,Ms)$; moreover by the choice of $\delta$ we may apply Corollary \ref{optAinfLip.cor} to $\widetilde\Omega^+$. Hence in particular, $\widetilde\omega^+ \in A_\infty(\sigma_+)$ on $\Delta_+:=B((0,\varphi^+(0)),s) \cap \Gamma_+$. For $\widetilde \eta >0$ to be determined later, there exists $\gamma = \gamma(\widetilde \eta) >0 $ such that
\begin{equation}\label{eq:Ainftyomp}
	\frac{\sigma_+(F)}{\sigma_+(\Delta_+)} \geq \frac{\gamma}{2} \text{ for a Borel set } F\subset \Delta_+ \implies \frac{\widetilde \omega_+(F)}{\widetilde \omega_+(\Delta_+)} \geq \widetilde \eta. 
\end{equation}

We then use properties $(1)$ and $(3)$ of Semmes decomposition along with the definition of $(\delta, R)$-chord arc domain to ensure that
\[ \sigma(\Delta_+ \cap \Delta_0) \geq (1-\delta) \omega_n s^n - c_1 \exp(-c_2 \delta^{-1/4}) \omega_n (Ms)^n \geq \left(1-\frac{\gamma}{3} \right)  \sigma(\Delta_0). \]
We may need to choose $\gamma$ slightly bigger (depending on the value of $\delta$), to make sure the last inequality holds.
Suppose that $E \subset \Delta_0$ is a Borel set with $\sigma(E) \ge \gamma \sigma(\Delta_0)$. Then
\[\sigma_+(E\cap \Delta_+) = \sigma(E \cap \Delta_+) \ge \frac23 \gamma \sigma(\Delta_0) \ge \frac{\gamma}{2} \sigma_+(\Delta_+),\]
where we used the estimates of $\sigma(\Delta_0), \sigma_+(\Delta_+)$ and that $\delta$ is chosen sufficiently small.
It follows from \eqref{eq:Ainftyomp} that 
\[ \frac{\widetilde \omega_+(E \cap \Delta_+)}{\widetilde \omega_+(\Delta_+)} \geq \widetilde \eta. \]
By the maximum principle and Bourgain's estimate
\[ \widetilde \omega(E) \geq \widetilde \omega_+(E\cap \Delta_+) \approx \frac{\widetilde \omega_+(E\cap \Delta_+) }{\widetilde \omega_+(\Delta_+)} \geq \widetilde \eta. \]
By Lemma \ref{agreeingops.lem} and a simple change of pole argument (or the shaper version, Lemma \ref{Kernelcompare.lem}) we have
\[ \frac{\omega(E)}{\omega(\Delta_0)} \approx \omega^{X_1}(E) \approx \widetilde \omega(E). \]
Hence
\[ \frac{\omega(E)}{\omega(\Delta_0)} \geq C\widetilde \eta = \eta \]
as desired. Here we chose $\widetilde \eta$ to account for the constant $C$.
\end{proof}
As we had reduced the proof of the theorem to Claims \ref{orighmisAinfty.cl} and \ref{halfalmosthalf.cl}, we have proved the theorem.
\end{proof}

\appendix
\section{A brief explanation of the first order approach}\label{Firstorderapp.app}
The goal of this appendix is twofold. Firstly, we justify that the solution to the $L^2$-Dirichlet problem obtained in \cite{AA} satisfies the estimates \eqref{1starpprime.eq} and \eqref{2starpprime.eq}. Whereas, modulo verification of the assumptions, \eqref{1starpprime.eq} is explicitly stated in \cite[Theorem 2.4(ii)]{AA}, the estimate \eqref{2starpprime.eq} requires a deeper understanding of \cite{AA} and is obtained by combining several lemmas and theorems in various sections. We recommend to the interested reader to read \textit{Section 3, Road map to the proofs} in \cite{AA} for the general idea of their first order approach. We give a brief overview here. After the general approach and all relevant notation are laid out, we give a rigorous proof of \eqref{2starpprime.eq} in Subsection \ref{sec:pfntmpert}. Then, we show that when the boundary data satisfies $0\leq f\in C_c^\infty(\mathbb{R}^n)$, the solution above (namely, the solution to the $L^2$-Dirichlet problem obtained in \cite{AA}) agrees with the classical elliptic measure solution (see \eqref{eq:hmsol}). This is carried out in Subsection \ref{sec:uniquesol}.

In what follows we will adopt the notation of \cite{AA} and write\footnote{This is done through a change of basis, so we keep the notion that $\ree_+ =\{(t,x): t > 0\}$.}$(t,x)$ instead of $(x,t)$ and for a vector $\vec{v} \in \co^{n+1}$ write $v = (v_\perp, v_\parallel)$ with $v_\perp$ being a scalar\footnote{The authors work with elliptic systems in \cite{AA}, but our overview will be for equations for simplicity, i.e. $m=1$.}. This change of basis is also reflected in the definition of the coefficient matrix $A$ in any divergence form elliptic operator. Note also that some of the conditions here are often easier to state in the scalar case, but we maintain the notation in \cite{AA} for the ease of identifying estimates therein. 
 
The first-order approach takes its name from the following: Suppose that $u$ is a solution to 
\begin{equation}\label{eq:2nd}
	L_A u := -\div A \nabla u =0 \text{ in } \ree_+ 
\end{equation}  
with $\nabla u \in L^2_{loc}(\mathbb{R}_+; L^2(\rn,\co^{n+1}))$, then its conormal gradient 
\begin{equation}\label{def:fcnm}
	f= \nabla_A u := \begin{bmatrix}
	\partial_{\nu_A} u \\
	\nabla_x u
\end{bmatrix}  \quad \text{ with } \partial_{\nu_A} u = (A \nabla_{t,x} u)_{\perp}
\end{equation} 
 solves the first-order equation
\begin{equation}\label{eq:1st}
\partial_t f + DBf = 0,
\end{equation}
where $D= \left[ \begin{matrix} 
0 & \div_x \\
-\nabla_x & 0
\end{matrix}\right]$ and $B$ is a matrix defined by $A$ as follows. 
We first write $A = \left[
\begin{matrix}A_{\perp\perp} & A_{\perp\parallel} \\
A_{\parallel\perp} & A_{\parallel\parallel}
 \end{matrix}\right]$, where $A_{\perp\perp}$ is a scalar. The strong ellipticity of the matrix $A$ implies $A_{\perp \perp}$ is strictly positive. We thus define
 \begin{equation}\label{def:BbyA}
 	B = \hat A:= \underline{A}(\overline{A})^{-1}= 
\begin{bmatrix}1& 0 \\
A_{\parallel \perp} & A_{\parallel\parallel}
 \end{bmatrix}
\begin{bmatrix}A_{\perp\perp} & A_{\perp\parallel} \\
0& I
 \end{bmatrix}^{-1} = \left[
\begin{matrix}1& 0 \\
A_{\parallel\perp} & A_{\parallel\parallel}
 \end{matrix}\right]
  \left[\begin{matrix}A_{\perp\perp}^{-1} & -A_{\perp\perp}^{-1}A_{\perp\parallel} \\
0 &I 
 \end{matrix}\right].
 \end{equation} 
In fact, if we make the restriction that solutions $f$ to \eqref{eq:1st} belongs to the space $ L^2_{loc}(\mathbb{R}_+; \mathcal{H})$ where
\[ \mathcal{H}: = \left\{g \in L^2(\rn, \co^{n+1}): \curl_x g = 0\right\}, \]
we have that the solutions $u$ to divergence form elliptic operator \eqref{eq:2nd} (with the `gradient bound on slices' as above) are in one-to-one correspondence with solutions $f$ to the first-order equation \eqref{eq:1st}. See \cite[Proposition 4.1]{AA}, where this is worked out in detail and note that this does not require the operator to be $t$-independent. 
Therefore to find solutions to boundary value problems (Neumann BVPs, regularity BVPs, or Dirichlet BVPs) of \eqref{eq:2nd} using first order approach, we need to
\begin{enumerate}
	\item study the well-posedness of the first-order equation \eqref{eq:1st} in an appropriate functional space;
	\item relate and determine the trace of $f$ at $t=0$ from the boundary value to \eqref{eq:2nd}.
\end{enumerate}
Notice that the first task depends on the operator $DB$ and it has nothing to do with what type of BVPs we consider.

\subsection{Well-posedness of the first order equation}
When the operator $A(t,x) = A_0(x)$ is $t$-independent and is a small $L^\infty$-perturbation of a constant matrix (or a real symmetric matrix), the well-posedness of the corresponding first order equation
\begin{equation}\label{eq:1st0}
	\partial_t f + DB_0 f =0, \qquad \text{ with } B_0 = \hat{A}_0
\end{equation} 
has been established in \cite{AAMc}. (This result has been obtained before, in \cite{FJK, AAAHK, AAH}, but we appeal to \cite{AAMc} since the notion of well-posedness there is compatible with \cite{AA}.) The authors remark that the operator $DB_0$ is not a sectorial operator, but instead bisectorial, that is, its spectrum is contained in a double sector around the real axis. This means that the natural operator $e^{-t DB_0}$ associated with the free evolution equation \eqref{eq:1st0} is not well-defined on all of $\mathcal{H} \subset L^2(\rn; \mathbb{C}^{1+n})$ for any $t\neq 0$. Thus we need to split $\mathcal{H}$ into the spectral subspace $E_0^+\mathcal{H}$ for the sector in the right half plane and the spectral subspace $E_0^- \mathcal{H}$ for the sector in the left half plane\footnote{$E^{\pm_0} := \chi^\pm(DB_0)$ provided by the bounded holomorphic calculus, where $\chi^+$ is the indicator of the right-half of the complex plane and $\chi^-$ is the indicator of the left-half of the complex plane. See \cite[Section 4]{AA}, between Proposition 4.1 and Proposition 4.3.}. Moreover, the authors in \cite{AAMc} show that any solution $f$ to \eqref{eq:1st0} in the appropriate functional space (such that $f\in L_{loc}^2(\mathbb{R}_+; \mathcal{H})$ and $\ntt f \in L^2 $) is given by the \textit{generalized Cauchy reproducing formula} 
\begin{equation}\label{eq:freeevol}
	f = C_0^+ f_0 = e^{-tDB_0}E_0^+f_0, \qquad \text{ for some } f_0 \in E_0^+ \mathcal{H}.
\end{equation}   
Additionally, using the notation $f_t = f(t,\cdot)$ it satisfies
\[ \lim_{t\to 0} f_t = f_0 \quad \text{ and } \quad \lim_{t\to \infty} f_t = 0 \]
in the $L^2$ sense.

Since every solution to \eqref{eq:1st0} is given by an explicit formula, to study the well-posedness when the matrix $A(t,x)$ is a (Carleson measure) perturbation of a $t$-independent operator $A_0(x)$, we rewrite its first order equation \eqref{eq:1st} as
\begin{equation}\label{eq:1stptb}
	\partial_t f + DB_0 f = D\E f,
\end{equation}
where $\E(t,x) = B_0(x) - B(t,x)$ denotes the perturbation. We remind the reader that $B= \hat A$ and $B_0 = \hat{A}_0$ are defined as in \eqref{def:BbyA}.
We also remark that if $\|A - A_0\|_{\mathcal{C}}< \infty$ then using the notation in \cite[Lemma 5.5]{AA} (see \cite[Definition 5.4]{AA} for the definition of the operator norm $\|\cdot\|_{*}$)
\[\|\mathcal{E}\|_{*} \lesssim \|\mathcal{E}\|_{\mathcal{C}} =  \|B - B_0\|_{\mathcal{C}} \lesssim \|A - A_0\|_{\mathcal{C}},\]
with implicit constants depending on the dimension and ellipticity constants. 
This is because 
\[ B - B_0 = \underline{A}(\overline{A})^{-1} - \underline{A_0}(\overline{A_0})^{-1}  = [(\underline{A}- \underline{A}_0)(\overline{A})^{-1}] + [\, \underline{A}_0((\overline{A})^{-1} - (\overline{A}_0)^{-1})] \]
and $\Lambda \lesssim \operatorname{Re} A_{\perp\perp}, \operatorname{Re} (A_0)_{\perp\perp} \leq \max\{ \|A\|_{\infty}, \|A_0\|_{\infty}\}$,
which allows us to control both bracketed terms using the Carleson norm of $A - A_0$. Hence under the assumption $\|A- A_0\|_{\mathcal{C}} \ll 1$, the right hand side of \eqref{eq:1stptb} can be thought of as a small error to the free evolution of $\partial_t f + DB_0 f $. The above discussion formally justifies \cite[Theorem 8.2]{AA}, which says solutions to \eqref{eq:1st}, or equivalently \eqref{eq:1stptb}, in the appropriate functional space is of the form
\begin{equation}\label{eq:nfreeevoltmp}
	f_t = C_0^+ h^+ + S_A f_t = e^{-t D B_0}E_0^+ h^+ + S_A f_t, \qquad \text{ for some } h^+ \in E_0^+ \mathcal{H}; 
\end{equation} 
and formally speaking as $t\to 0$ we have $f_t \to f_0 = h^+ + h^-$, where $h^- \in E_0^- \mathcal{H}$ and is determined explicitly.
(Notice that when $h^+ \in E_0^+ \mathcal{H}$, we have $e^{-DB_0} E_0^+ h^+ = e^{-t|DB_0|} h^+ $, see for example the discussion in page 62 of \cite{AA}. So the above formula is the same as that in \cite[Theorem 8.2]{AA}.)
The operator $S_A$ is given formally by 
\begin{equation}\label{def:SA}
	S_A f_t := \int_0^t e^{-(t-s) DB_0} E_0^+ D \E_s f_s ds - \int_t^{\infty} e^{(s-t)DB_0} E_0^- D\E_s f_s ds, 
\end{equation} 
see \cite[Equation (1)]{AA}; for a rigorous treatment see \cite[Proposition 7.1]{AA}. Moreover they also show
\[\|\widetilde{N}_\ast (S_A f)\|_{L^2(\rn)} \lesssim \|\E\|_* \|\ntt f\|_{L^2(\rn)} \lesssim \|A - A_0\|_{\mathcal{C}}\| \widetilde{N}_\ast f\|_{L^2(\rn)}, \]
whereby one concludes the boundedness of $(1 - S_A)^{-1}$  on the space 
\[ \mathcal{X}:= \{f: \ree_+ \to \mathbb{C}^{1+n}; \, \widetilde{N}_\ast f \in L^2\} \] for $\|A - A_0\|_{\mathcal{C}}$ sufficiently small. Therefore one seeks solutions of the form 
\begin{equation}\label{eq:nfreeevol}
	f = (1 - S_A)^{-1}C_0^+ h^+, \qquad \text{ for some } h^+ \in E_0^+ \mathcal{H}.
\end{equation}

We summarize the above discussion. Our goal is to find solutions to divergence form elliptic equations \eqref{eq:2nd} with prescribed boundary values (they could be Neumann, regularity, or Dirichlet boundary values). By the one-to-one correspondence between solutions to \eqref{eq:2nd} and solutions to the first-order equation \eqref{eq:1st}, it suffices to determine $h^+$ using the prescribed boundary value and then use the ansatz \eqref{eq:nfreeevol} to compute $f$. In other words, we need to determine the trace of $f_t$ using the boundary value of $u$.

\subsection{Determining the trace of $f$ by Neumann or regularity boundary values of $u$}
By looking at \eqref{def:fcnm}, how $f$ is defined using $u$, it should be intuitively clear that Neumann BVPs and regularity BVPs are more natural in the first order approach, since the trace of $f_t$ is related naturally to the Neumann or regularity boundary value. To be more precise, when $A= A_0$ is $t$-independent, a solution $f = C_0^+ h^+$ satisfies the Neumann boundary value $(f_0)_{\perp} = \partial_{\nu_{A_0}} u = \varphi $ if and only if its trace $h^+$ solves the equation $\Gamma_{A_0} h^+ = \varphi$, where
\begin{align*}
	\Gamma_{A_0}: & E_0^+ \mathcal{H}  \to L^2(\rn, \mathbb{C}) \\
	& h^+ \mapsto (h^+)_{\perp}
\end{align*}
is the identification map from the trace of $f$ to the Neumann boundary value of the solution $u$ to the divergence form elliptic operator \eqref{eq:2nd}. In other words the well-posedness of the Neumann BVP is equivalent to $\Gamma_{A_0}$ being an isomorphism.
Similarly the well-posedness of the regularity BVP is equivalent to the identification map (with regularity boundary value)
\begin{align*}
	\Gamma_{A_0}: & E_0^+ \mathcal{H} \to \{g \in L^2(\rn, \co^n): \curl_x g = 0\} \\
	& h^+ \mapsto (h^+)_{\parallel} 
\end{align*}  
being an isomorphism. We remark that even for $t$-independent operators, these maps are not always invertible; however $\Gamma_{A_0}$ is invertible if we assume a-priori the well-posedness of BVPs for $L_{A_0}$ (for example see \cite[Corollary 8.6]{AA}).

Now we begin to consider operators with $t$-dependent coefficients. Recall in the above discussion on the well-posedness of equation \eqref{eq:1st}, or equivalently \eqref{eq:1stptb}, a solution $f$ of the form \eqref{eq:nfreeevoltmp} (for some $h^+ \in E_0^+ \mathcal{H}$) has trace
\[ f_0 = h^+ + h^- \text{ and } h^- = \int_0^{\infty} \Lambda e^{-s\Lambda} \widehat{E}_0^- \E_s f_s \, ds \in E_0^- \mathcal{H}, \]
where $\Lambda = |DB_0|$ and $\widehat{E}_0^-$ is defined as \cite[Equation (22)]{AA}. See \cite[Theorem 8.2]{AA} for the precise statement. Therefore $f$ satisfies the Neumann boundary condition $(f_0)_{\perp} = \partial_{\nu_A} = \varphi$ if and only if $h^+$ solves the equation $\Gamma_A h^+ = \varphi$, where $\Gamma_A$ is the map
\begin{align*}
	\Gamma_A: & E_0^+ \mathcal{H} \to L^2 (\rn, \mathbb{C}) \\
	& h^+ \mapsto (f_0)_{\perp} = \left( h^+ + \int_0^{\infty} \Lambda e^{-s \Lambda} \widehat{E}_0^- \E_s f_s ds \right)_{\perp}.
\end{align*}
It follows that the well-posedness\footnote{Here one is seeking solutions in the space $\mathcal{X}$ and the boundedness of $(1 - S_A)^{-1}C_0^+$ as an operator from $E_0^+$ into $\mathcal{X}$ are provided in \cite{AA}.} of Neumann BVPs in appropriate function space $\mathcal{X}$ is equivalent to $\Gamma_A $ being an isomorphism. On the other hand, one can show\footnote{See the proof of \cite[Corollary 8.6]{AA}.}
\[\|\Gamma_A - \Gamma_{A_0}\|_{L^2 \to L^2} \lesssim \|\E\|_* \lesssim \|A-A_0\|_{\mathcal{C}}\]
so that one may deduce the invertibility of $\Gamma_{A}$ from that of $\Gamma_{A_0}$, provided the Carleson norm here is small. 
Thus, the Neumann BVP is well-posed for $L_A$, provided the smallness of the Carleson norm and the well-posedness of the Neumann BVP for $L_{A_0}$. 
The method for the regularity problem is similar (where we instead wish to invert the tangential trace).

\subsection{Dirchlet boundary value problems}
Solving Dirichlet BVPs using the above first order approach is more complicated compared to Neumann BVPs or regularity BVPs; and we will give more details in this section since it is directly related to our proof of \eqref{2starpprime.eq}. It is not obvious how the first order approach applies, due to the lack of identification between the trace of $f$ and the Dirichlet boundary value of $u$.
Instead of the equation \eqref{eq:1st} for the conormal gradient $f$, we consider vector-valued solutions to first order equation
\begin{equation}\label{eq:1stDir}
	\partial_t v + BD v = 0.
\end{equation}
Heuristically, applying $D$ to the equation \eqref{eq:1stDir} gives $(\partial_t + DB)(Dv) = 0$. On the other hand, $u$ solves the divergence form elliptic equation if and only if $(\partial_t + DB)(\nabla_A u) = 0$. By comparing $f= \nabla_A u$ with $f= Dv$ we find that $u=-v_{\perp}$. 
More precisely, for coefficients $A(t,x)$ which are (Carleson measure) perturbations of $t$-independent coefficient $A_0(x)$, we have that solutions to the divergence form elliptic equation \eqref{eq:2nd} obeying a certain square function estimate (that is, $\nabla u \in \mathcal{Y}$ defined below in \eqref{def:Y}) are of the form
\[ u = c- v_{\perp} \]
where $v$ solves the first order equation \eqref{eq:1stDir} and $c\in \mathbb{C}$. 
Moreover, we have
\[ \lim_{t\to 0} u_t = c-(v_0)_{\perp} \quad \text{ and } \lim_{t\to \infty} u_t = c \] in the $L^2$ sense. In particular, if we impose that $u$ has Dirichlet boundary value in $L^2(\rn, \mathbb{C})$, the constant $c$ is zero. 
See \cite[Theorem 9.3]{AA}.
We will follow the first order approach as before to study solutions to equation \eqref{eq:1stDir} and then find an ansatz for solutions to Dirichlet BVPs.

For $t$-independent operators we note that $B_0 D$ is another bisectorial operator, just like $DB_0$. So we define the spectral projections $\widetilde{E}_0^{\pm} = \chi^{\pm}(B_0 D)$ as before, which splits the space $\mathcal{H}$. We make an important remark with regards to $DB_0$ and $B_0 D$:
\begin{equation}\label{eq:cmt}
	B_0 \, b(DB_0) = b(B_0 D) \, B_0
\end{equation}
where $b(\cdot)$ denotes the functional calculus to an operator on $L^2$. (See the discussion in \cite[Section 7]{AA}.) This observation allows us to switch between $DB_0$ and $B_0 D$.
Similar to the argument for $DB_0$, we have that solutions to \eqref{eq:1stDir} obeying a square function estimate are of the form
\begin{equation}\label{eq:vfreeevol}
	v = \widetilde{C}_0^+ v_0 + c = e^{-t B_0 D} \widetilde{E}_0^+ v_0 + c 
\end{equation} 
for a unique $v_0 \in \widetilde{E}_0^+ L^2$ and some $c\in \mathbb{C}^{1+n}$. Therefore, for $t$-independent operators we have the representation formula
\begin{equation}
	u= c- \left(\widetilde{C}_0^+ v_0 \right)_{\perp}, \qquad  v_0 \in \widetilde{E}_0^+ L^2, c\in \mathbb{C}
\end{equation}
for solutions $u$ to Dirichlet BVPs obeying a square function estimate. See \cite[Corollary 9.4]{AA}.

Now we consider perturbations of $t$-independent operators. Recall that solutions $f$ to \eqref{eq:1st} are of the form
\begin{equation}\label{tmp:fDir}
	f_t = e^{-t D B_0} h^+ + S_A f_t, \qquad \text{ or equivalently } f_t = (I-S_A)^{-1} e^{-t DB_0} h^+, 
\end{equation} 
for some $h^+ \in E_0^+ \mathcal{H}$. We remark that to adapt to Dirichlet BVPs, we work with the functional space
\begin{equation}\label{def:Y}
	\mathcal{Y} := \left\{ f: \ree_+ \to \mathbb{C}^{1+n}; \int_0^{\infty} \|f_t\|_{L^2(\rn)}^2 tdt < \infty  \right\}  = L^2(\mathbb{R}_+, tdt; L^2(\rn, \mathbb{C}^{1+n})). 
\end{equation} 
instead of $\mathcal{X}$. Moreover \cite[Proposition 7.1]{AA} shows that 
\[\|S_A\|_{\mathcal{Y} \to \mathcal{Y}} \lesssim \|A - A_0\|_{\mathcal{C}}.\]
Provided that $\|A - A_0\|_{\mathcal{C}}$ is sufficiently small, $(1 - S_A)^{-1}$ exists as a bounded operator on the space $\mathcal{Y}$.

We want to find $v$ satisfying $f= Dv$, so we basically need to factor out $D$ in \eqref{tmp:fDir} (even though $D$ is not injective).
Indeed by \eqref{eq:cmt}, imposing the free evolution term $g:= e^{-tDB_0} h^+$ (the first term in \eqref{tmp:fDir}) in $\mathcal{Y}$ allows us to rewrite 
\begin{equation}\label{tmp:g}
	g= D e^{-t B_0 D} \widetilde{E}_0^+ \widetilde{h}^+ = D \widetilde{C}_0^+ \widetilde{h}^+
\end{equation} 
for some $\widetilde{h}^+ \in \widetilde{E}_0^+ L^2$ determined by $h^+ = D\widetilde{h}^+$.
And formally for $S_A$, we obtain starting from \eqref{def:SA} that $S_A = D \widetilde{S}_A$ where
\begin{equation}\label{def:StA}
	\widetilde{S}_A f_t := \int_0^t e^{-(t-s) B_0 D} \widetilde{E}_0^+ D \E_s f_s ds - \int_t^{\infty} e^{(s-t)B_0 D}\widetilde{E}_0^- D\E_s f_s ds.
\end{equation}
See \cite[Proposition 7.2]{AA} for the rigorous treatment of $\widetilde{S}_A$.
Putting them together we get
\[ f_t = e^{-t D B_0} h^+ + S_A f_t = D \widetilde{C}_0^+ \widetilde{h}^+ + D \widetilde{S}_A f_t,  \]
and thus we can set 
\[ v= \widetilde{C}_0^+ \widetilde{h}^+ + \widetilde{S}_A f_t = \widetilde{C}_0^+ \widetilde{h}^+ + \widetilde{S}_A (I-S_A)^{-1} D \widetilde{C}_0^+ \widetilde{h}^+. \]
In the second equality we substitute the expression for $f_t$ in \eqref{tmp:fDir} and we also use \eqref{tmp:g}. Notice that the first term is exactly the same as \eqref{eq:vfreeevol}, the solution to $\partial_t v + B_0 D v=0$ (the constant $c$ there is always zero for $L^2$-Dirichlet BVPs).
Therefore, to solve the Dirichlet problem to \eqref{eq:2nd}, we make the ansatz
\begin{equation}\label{eq:uDir}
	u= \left( \widetilde{C}_0^+ \widetilde{h}^+ + \widetilde{S}_A (I-S_A)^{-1} D \widetilde{C}_0^+ \widetilde{h}^+ \right)_{\perp}
\end{equation}
for some $\widetilde{h}^+ \in \widetilde{E}_0^+ L^2$.
Moreover, 
\[ \lim_{t\to 0} v_t = v_0 \quad \text{ and } \quad \lim_{t\to \infty} v_t = 0 \]
in the $L^2$ sense. The trace $v_0$ can be written explicitly as
\begin{equation}
	v_0 = \widetilde{h}^+ + \widetilde{h}^- \text{ and } \widetilde{h}^- = - \int_0^{\infty} e^{-s \widetilde{\Lambda}} \widetilde{E}_0^- \E_s f_s \, ds \in \widetilde{E}_0^- L^2,
\end{equation}
with $\widetilde{\Lambda} = |B_0 D|$.
See \cite[Theorems 9.2, 9.3]{AA} and the proof of \cite[Corollary 9.5]{AA}.

It then follows that the solution $u$ has Dirichlet boundary value $\varphi \in L^2(\rn)$ if and only if $\widetilde{h}^+ \in \widetilde{E}_0^+ L^2$ satisfies
\[ \varphi = \lim_{t\to 0} u_t = -(v_0)_{\perp} = \left( - \widetilde{h}^+ + \int_0^{\infty} e^{-s \widetilde{\Lambda}} \widetilde{E}_0^- \E_s f_s \, ds \right)_{\perp}, \]
where $f = (I-S_A)^{-1} D \widetilde{C}_0^+ \widetilde{h}^+$ by the formula \eqref{tmp:fDir}.
In other words, if we define the map
\begin{align*}
	\widetilde{\Gamma}_A: & \widetilde{E}_0^+ L^2 \to L^2(\rn, \mathbb{C}) \\
	& \widetilde{h}^+ \mapsto \left( - \widetilde{h}^+ + \int_0^{\infty} e^{-s \widetilde{\Lambda}} \widetilde{E}_0^- \E_s f_s \, ds \right)_{\perp},
\end{align*}
then $\widetilde{h}^+$ is determined by the equation $\widetilde{\Gamma}_A \widetilde{h}^+ = \varphi$. The well-posedness of Dirichlet BVPs is equivalent to $\widetilde{\Gamma}_A$ being an isomorphism.
In particular, when $A= A_0$ is $t$-independent, this map is just
\begin{align*}
	\widetilde{\Gamma}_{A_0}: & \widetilde{E}_0^+ L^2 \to L^2(\rn, \mathbb{C}) \\
	& \widetilde{h}^+ \mapsto - \left( \widetilde{h}^+ \right)_{\perp}.
\end{align*}
Assuming the Dirichlet problem for $L_{A_0}$ is well-posed, the map $\widetilde{\Gamma}_{A_0}$ is invertible. Since (see the proof of \cite[Corollary 9.5]{AA})
\begin{equation}\label{eq:idDir}
	\|\widetilde{\Gamma}_{A_0} - \widetilde{\Gamma}_{A}\|_{L^2 \to L^2} \lesssim \|\mathcal{E}\|_* \lesssim \|A - A_0\|_{\mathcal{C}},
\end{equation} 
it follows that $\widetilde{\Gamma}_A$ is also invertible provided $\|A-A_0\|_{\mathcal{C}} \ll 1$. Therefore the Dirichlet problem for $L_A$ is also well-posed.

\subsection{Proof of estimates for non-tangential maximal function}\label{sec:pfntmpert}
We are now ready to use the results in \cite{AA} to prove the desired estimates \eqref{1starpprime.eq} and \eqref{2starpprime.eq} for the non-tangential maximal function. Recall that the elliptic matrix we consider can be written of the form 
\[ A(t,x) = A_0(x) + B(t,x) \]
where $A_0(x)$ is a small ($t$-independent) perturbation of a \textit{constant} real elliptic matrix and $B(t,x)$ has small Carleson norm, see \eqref{eq:decompA} and \eqref{def:pert}. (Here we denote the $t$-independent matrix by $A_0(x)$ instead of $A_1(x)$ in \eqref{eq:decompA}, in order to be consistent with the notation in the rest of the Appendix.) The well-posedness for Dirichlet problems is established in \cite{AAH} for elliptic operators whose $t$-independent coefficient matrices are small perturbations of the constant matrix. Thus the Dirichlet problem for $L_{A_0}$ is well-posed if $\epsilon$ in \eqref{def:pert} is sufficiently small. We claim that this solution agrees with the \textit{elliptic measure solution} (as in \eqref{eq:hmsol}) for smooth, compactly supported data and we postpone the proof of this fact to Subsection \ref{sec:uniquesol}.

The estimate \eqref{1starpprime.eq} just follows directly from \cite[Theorem 2.4(ii)]{AA}.
Now we set out to prove \eqref{2starpprime.eq}. We write corresponding solutions to $L_A$ and $L_{A_0}$ as $u_{\varphi, A}$ and $u_{\varphi, A_0}$, respectively.
By the discussion in the above section, we can write
\[ u_{\varphi, A_0} = \left( \widetilde{C}_0^+ \widetilde{\Gamma}_{A_0}^{-1} \varphi \right)_{\perp}, \]
\[ u_{\varphi, A} = \left( \widetilde{C}_0^+ \widetilde{\Gamma}_A^{-1} \varphi + \widetilde{S}_A (I-S_A)^{-1} D \widetilde{C}_0^+ \widetilde{\Gamma}_A^{-1} \varphi \right)_{\perp}. \]
Hence
\begin{align*}
u_{\varphi,A} - u_{\varphi,A_0} &= \left(\widetilde{C}_0^+ (\widetilde{\Gamma}_{A}^{-1}- \widetilde{\Gamma}_{A_0}^{-1})\varphi\right)_\perp 
+ \left( \widetilde{S}_A( I - S_A)^{-1}D \widetilde{C}_0^+ \widetilde{\Gamma}_{A}^{-1}\varphi\right)_\perp
\\& = \Ir + \IIr.
\end{align*}
Notice that $\Ir$ is just $ \widetilde{C}_0^+\tilde{h}^+$ with
\[\tilde{h}^+ := (\widetilde{\Gamma}_{A}^{-1} - \widetilde{\Gamma}_{A_0}^{-1})\varphi \in \widetilde{E}^+_0 L^2. \]
Recall \eqref{eq:idDir} and the invertibility of $\widetilde{\Gamma}_{A_0}$, we have
\begin{align*}
	\|\widetilde{\Gamma}_A^{-1} - \widetilde{\Gamma}_{A_0}^{-1} \|_{L^2 \to L^2} & = \left\|\widetilde{\Gamma}_A^{-1} \left( \widetilde{\Gamma}_{A_0} - \widetilde{\Gamma}_A \right) \widetilde{\Gamma}_{A_0}^{-1} \right\|_{L^2 \to L^2} \\
	& \lesssim \|\widetilde{\Gamma}_A - \widetilde{\Gamma}_{A_0} \|_{L^2 \to L^2} \\
	& \lesssim \|A-A_0\|_{\mathcal{C}},
\end{align*}
provided $\|A-A_0\|_{\mathcal{C}}$ is sufficiently small. It follows then
\[ \|\widetilde{h}^+ \|_{L^2} \lesssim \|A-A_0\|_{\mathcal{C}} \|\varphi\|_{L^2}. \]
As in the proof of \cite[Theorem 10.1]{AA}, we may write $\tilde{h}^+  $ as $\tilde{h}^+ = B_0h^+$ with $h^+ \in E^+_0 \mathcal{H}$ so that
\[ \|\ntt(\Ir)\|_{L^2} =  \|\ntt ( \widetilde{C}_0^+ B_0h^+)\|_{L^2} = \|\ntt (B_0 C_0 h^+)\|_{L^2} \approx \|C_0 h^+\|_{\mathcal{X}} \lesssim  \|h^+\|_{L^2} \approx  \|\tilde{h}^+\|_{L^2},\]
where we use \eqref{eq:cmt}, \cite[Theorem 5.2]{AA} and the accretivity of $B_0$.
Thus
\[\|\ntt^{3/2} (\Ir)\|_{L^2} \le \|\ntt (\Ir)\|_{L^2} \lesssim \|A - A_0\|_{\mathcal{C}}\|\varphi\|_{L^2}, \]
as desired. 
To handle $\IIr$ we use \cite[Lemma 10.2]{AA} which says that 
\[\|\ntt^{3/2}((\widetilde{S}_A f)_\perp)\|_{L^2} \lesssim \|\E\|_{*} \|f\|_{\mathcal{Y}} \lesssim \|A - A_0\|_{\mathcal{C}}\|f\|_{\mathcal{Y}}. \]
(This is why we have non-tangential maximal function with power $3/2$ in \eqref{2starpprime.eq}.) 
 Thus, to obtain a desirable bound for $\IIr$ it is enough to show
\begin{equation}\label{eq:tmpNII}
	\|(I - S_A)^{-1}D \widetilde{C}_0^+ \tilde{h}^+ \|_{\mathcal{Y}} \lesssim \|\tilde{h}^+\|_{L^2}, \qquad \text{ for any } \tilde{h}^+ \in \widetilde{E}^+_0 L^2,
\end{equation} 
where we used that $\widetilde{\Gamma}_{A}^{-1}$ is an isomorphism to exchange $\widetilde{\Gamma}_{A}^{-1}\varphi$ for $\tilde{h}^+$. Recall (see \cite[Proposition 7.1]{AA}) that $(I - S_A)^{-1}: \mathcal{Y} \to \mathcal{Y}$  is bounded; moreover by the accretivity of $B_0$ and the square function estimate for the operator $B_0 D$ we have
\[\|D \widetilde{C}_0^+ \tilde{h}^+\|_{\mathcal{Y}} \approx \|B_0De^{-tB_0 D}\tilde{h}^+\|_{\mathcal{Y}} \approx \|\tilde{h}^+\|_{L^2}, \quad \forall \tilde{h}^+ \in \widetilde{E}^+_0 L^2.\]
This finishes the proof of \eqref{eq:tmpNII} and therefore
\begin{align*}
	\|\ntt^{3/2}(\IIr) \|_{L^2} & \lesssim \|A-A_0\|_{\mathcal{C}} \| (I - S_A)^{-1}D \widetilde{C}_0^+ \widetilde{\Gamma}_A^{-1} \varphi  \|_{\mathcal{Y}} \\
	& \lesssim \|A-A_0\|_{\mathcal{C}} \|\widetilde{\Gamma}_A^{-1} \varphi \|_{L^2} \\
	& \lesssim \|A-A_0\|_{\mathcal{C}} \|\varphi\|_{L^2}.
\end{align*} 
This concludes the proof of \eqref{2starpprime.eq} modulo showing that the solution here agrees with the elliptic measure solution, when the boundary data is smooth.

\subsection{For smooth data the solution in \cite{AA} agrees with the elliptic measure solution}\label{sec:uniquesol} It remains to show that given $0\leq f \in C_c^\infty(\rn)$ the unique $L^2$-solution $u$ in the sense of \cite{AA} agrees with the elliptic measure solution $u_f$ as in \eqref{eq:hmsol}.\footnote{This is not without cause, since in general (for example, when the coefficient matrix is non-symmetric), even with smooth boundary data different notions of solutions may not agree, see the example in \cite{Axel10}. Even for the Laplacian in the upper half space, it is well known that solutions to the Dirichlet problem are not unique.} (It suffices to consider functions $0\leq f\in C_c^\infty(\mathbb{R}^n)$, since this is enough to characterize the elliptic measure, or equivalently the Poisson kernel. Assuming that $u\in C(\overline{\mathbb{R}^{n+1}_+})$, by the maximum principle it suffices to show $u-u_f(X) \to 0$ as $|X| \to \infty$. Since we know that $u_f(X) \to 0$ as $|X|\to \infty$\footnote{This can be proven by comparing it with the elliptic measure of a compact set. Let $K = \supp f$ and assume $K \subset B_R$ for some $R>0$. Since $0\leq f \in C_c^\infty(\RR^n)$, we have $u_f(X) \leq \omega^X(\Delta_R) \cdot \sup f$. On the other hand, let $A_R$ denote the interior corkscrew point for the ball $B_R$. By Lemma \ref{CFMS.lem} and the estimate of the Green's function, when $X \in \RR^{n+1}_+ \setminus B_{4R}$ we have $\omega^X(\Delta_R) \approx G(X, A_R ) \cdot R^{n-1} \lesssim \frac{R^{n-1}}{|X-A_R|^{n-1}}$. The right hand side converges to zero as $|X| \to \infty$. Therefore $\omega^X(\Delta_R) \to 0$, and thus $u_f(X) \to 0$ as $|X| \to \infty$.}, 
 to prove $u\equiv u_f$ it suffices to show $u$ is continuous all the way to the boundary and $u(X) \to 0$ as $|X| \to \infty$. In fact, we will show in the following lemma that $u \in \dot{C}^\beta(\overline{\ree_+})$, where $\dot{C}^\beta$ is the homogeneous H\"older space:
\[\|v\|_{\dot{C}^\beta(E)} := \sup_{\substack{x, y \in E \\ x \neq y}} \frac{|v(x) - v(y)|}{|x - y|^\beta}.\]
for a function $v:E \to \mathbb{R}$.

\begin{lemma}[\cite{HMiMo,HMayMour}]\label{Cbetasolv.lem}
Let 
\[A(t,x) = A_0(x) + B(t,x)\]
be a real matrix such that $A_0(x) = A_c + \widetilde{B}(x)$, $A_c$ is a real constant elliptic matrix, $\|\widetilde{B}\|_\infty < \epsilon_0$ and $\|B(t,x)\|_{\mathcal{C}} < \epsilon_0$. If $\epsilon_0$ is small depending on the ellipticity of $A_c$ and dimension then there exists $\beta \in [0, 1)$ depending on ellipticity, the dimension and $\epsilon_0$ such that the following holds. If $f \in C_c^\infty(\rn)$ and $u$ is the (unique) solution to the $L^2$-Dirichlet problem for $L = -\div A\nabla$ on $\ree_+$ produced above, then 
\begin{equation}
\|u\|_{\dot{C}^\beta(\overline{\ree_+})} \lesssim \|f\|_{\dot{C}^\beta(\rn)},
\end{equation}
where the implicit constant depends on dimension and ellipticity of $A_c$ 

\end{lemma}
\begin{proof} Here we appeal to the `second order methods' noting that the $L^2$ solution to the Dirichlet problem above is unique so that we are working with the same solution.
The lemma is, in fact, a `direct' result of \cite[Theorem 1.4]{HMiMo} and \cite[Theorem 1.35]{HMayMour}. First note that $A_c$ is constant so that we may assume $A_c$ is symmetric as demonstrated in Lemma \ref{constnondegen.lem} above. Thus, \cite[Theorem 1.4]{HMiMo} gives the solvability of the $\dot{C}^{\beta'}$ Dirichlet problem for coefficients $A_0$ with $\beta' \in (0, \beta'_0)$ depending on dimension, ellipticity of $A_c$ and $\epsilon_0$, provided $\epsilon_0$ is small enough. Similarly, (at the cost of a smaller $\beta$) \cite[Theorem 1.35]{HMayMour} can be applied to perturb from the coefficients $A_0$ to $A$, giving the solvability of the $\dot{C}^\beta$ Dirichlet problem with $\beta \in (0, \beta_0'/2)$ for $L = -\div A\nabla$ provided $\epsilon_0$ is small enough. Here is where one should be careful: We need to check that the $\dot{C}^\beta$ and $L^2$ solutions agree when the data is in $\dot{C}^\beta(\rn) \cap L^2(\rn)$, we will call this `compatibility'. Note that the $\dot{C}^\beta$ and $L^2$ solutions agree for the operator with coefficients $L_c = -\div A_c \nabla$ because they are both given by convolution with a elliptic-Poisson kernel (see Theorems 1.4 and 3.3 in \cite{MarinMartellMitrea}). The interested reader can carefully check\footnote{This will be a result of the boundary trace of the layer potentials being perturbative in the norms $\dot{C}^\alpha$ and $L^2$: In the case of \cite{HMiMo} this is done in \cite[Section 4]{HMiMo} and in the case of \cite{HMayMour} in \cite[Proposition 7.21]{HMayMour}.} that the perturbations \cite[Theorem 1.4]{HMiMo} and \cite[Theorem 1.35]{HMayMour} preserve this compatibility.  

\end{proof}

With the lemma in hand, we are going to prove $u(X) \to 0$ as $|X| \to \infty$.
By definition it holds that
\[\|\widetilde{N}_\ast(u)\|_{L^2(\rn)} \lesssim \|f\|_{L^2(\rn)}\]
which (by interior estimates) implies that 
\[\sup_{t > 0} \|u(t,\cdot)\|_{L^2(\rn)} \lesssim\|f\|_{L^2(\rn)}.\]

Next, we see for $s \in \mathbb{R}_+$, by breaking up $\rn$ into cubes of side length roughly $s$ and using Caccioppoli's inequality, that 
\[\int_s^{2s} \int_{\rn} |\nabla u|^2 \, dx \, dt \lesssim \frac{1}{s^2} \int_{s/2}^{5s/2} \int_{\rn} |u|^2\, dx \, dt \lesssim \frac{\|f\|^2_{L^2(\rn)}}{s},\] 
where we used $\sup_{t > 0} \|u(t,\cdot)\|_{L^2(\rn)} \lesssim\|f\|_{L^2(\rn)}$.
Thus, 
\[\int_{s}^\infty \int_{\mathbb{R}^n} |\nabla u|^2 \, dx \, dt  = \sum_{k = 0}^\infty \int_{2^ks}^{2^{k+1}s} \int_{\mathbb{R}^n} |\nabla u|^2 \, dx \, dt \lesssim \sum_{k = 0}^\infty \frac{\|f\|^2_{L^2(\rn)}}{2^k s}  \lesssim \frac{\|f\|^2_{L^2(\rn)}}{s},\]
written compactly $\|\nabla u\|_{L^2(\ree_{s})} \lesssim \tfrac{\|f\|_{L^2(\rn)}}{\sqrt{s}}$, where $\ree_s = \{(t,x): t > s, x \in \mathbb{R}^n\}$. It is a fact\footnote{See \cite[Theorem 1.78]{MalyZiemer}. One can adapt the proof there using nested cubes (this time not concentric dilates though) which exhaust $\ree_s$.}  that there exists a constant $c$ such that $u - c \in Y^{1,2}(\ree_s):= \{v \in L^{\frac{2n}{n -1}}(\ree_s): \nabla v \in L^2(\ree_s)\}$, but since $\sup_{t > 0} \|u(t,\cdot)\|_{L^2(\rn)} < +\infty$ it must be the case that $c = 0$ and hence $u \in Y^{1,2}(\ree_s)$. Moreover, by the Poincar\'e Sobolev inequality 
\begin{equation}\label{PSbound.eq}
\|u\|_{L^{\frac{2n}{n -1}}(\ree_s)} \lesssim \|\nabla u\|_{L^2(\ree_s)} \lesssim  \|f\|_{L^2(\rn)}/\sqrt{s},
\end{equation}
where we used our estimate established above.

Now fix $\gamma > 0$ and let $C_1 := \|u\|_{\dot{C}^\beta(\overline{\ree_+})} < \infty$. Our goal is to show that there exists $M_\gamma$ so that if $|X| > M_\gamma$ then $|u(X)| \le \gamma$. Note that if there exists $X= (t,x)$ such that $|u(X)| > \gamma$ then we have that $|u(Y)| > \gamma/2$ for all $Y \in \ree_+$ such that $|X - Y| < (\gamma/2C_1)^{1/\beta}$. In particular, if $s \in \re_+$ and there exists $X= (t,x)$ such that $|u(X)| > \gamma$ with $t \ge s$ then
\[\|u\|_{L^{\frac{2n}{n -1}}(\ree_s)} \gtrsim  \gamma^{\frac{n^2-1}{\beta 2n} + 1}\]
where we used that $|u(Y)| > \gamma/2$ in $B(X, (\gamma/2C_1)^{1/\beta}) \cap \ree_s$. Thus, choosing $s_0$ large enough so that 
\[\frac{\|f\|_{L^2(\rn)}}{\sqrt{s_0}} \ll \gamma^{\frac{n^2-1}{\beta 2n} + 1} \]
we have from \eqref{PSbound.eq} that $\|u\|_{L^\infty(\ree_{s_0})} \le \gamma$. 

Having established the bound for $t \ge s_0$, it suffices to show that there exists $A$ so that if $|x| > A$ and $t \in [0, s_0]$ then $|u(t,x)|\le \gamma$. To do this, we use that $\sup_{t > 0} \|u(t,\cdot)\|_{L^2(\rn)} < \|f \|_{L^2(\rn)}$ and hence
\[\|u\|_{L^2([0, s_0] \times \rn)} \lesssim \sqrt{s_0} \|f\|_{L^2(\mathbb{R}^n)}.\]
Arguing as above, it would be impossible to have $X_k = (x_k, t_k)$ with $|x_k| \to \infty$ and $t_k \in [0,s_0]$ such that $|u(X_k)| \ge \gamma$. Indeed, for this would give that $\|u\|_{L^2([0, s_0] \times \rn)} = \infty$. This concludes the proof that $u(X) \to 0$ as $|X| \to \infty$.

\section{From \eqref{cond:d} to \eqref{cond:c}}\label{Koreyexplain.sect}
The goal of this section is to prove Lemma \ref{Koreyexplain.lem}, which will immediately show \ref{cond:d} implies \ref{cond:c}.
We first make an observation that allows us to use the work of Korey \cite{Korey} with impunity in the setting of this work. Specifically, we would like to use the consequences of \cite[Theorem 6]{Korey} and the portion of \cite[Theorem 10]{Korey} unique to the work there.
\begin{lemma}\label{sigmadiffuse.lem}
Let $\Gamma \subset \ree$ be a closed set with locally finite $\HH^n$ measure, that is $\HH^n(\Gamma \cap B(0, r)) < \infty$ for every $r > 0$. Set $\sigma = \HH^n|_\Gamma$. Then for every Borel set of $E \subseteq \Gamma$ and $\tau \in [0,1]$ there exists a Borel $F$ with $\sigma(F) = \tau\sigma(E)$.
\end{lemma}
\begin{proof}
Fix $E$ as above. Clearly, we only need to show result for $\tau = 1/2$. We see by monotonicity of measure $\sigma(B(0,R) \cap E) > (1/2)\sigma(E)$ for some $R$ large enough. Let $r_0 := \sup\{r : \sigma(B(0,r) \cap E) < (1/2)\sigma(E)\}$.  By monotonicity of measure $\sigma(E \cap \partial B(0,r_0)) +  \sigma(B(0,r_0) \cap E)  \ge (1/2)\sigma(E)$ and $\sigma(B(0,r_0) \cap E) \le (1/2)\sigma(E)$. Thus, $\tau' \sigma(E \cap \partial B(0,r_0)) + \sigma(B(0,r_0) \cap E) = (1/2)\sigma(E)$ for some $\tau' \in [0,1]$. Next we note that $\widetilde\sigma := H^n|_{\partial B(0,r_0)}$ already has the diffusivity property, that is, for every $E' \subset \partial B(0,r_0)$ and $\tau' \in [0,1]$ there exists $F' \subseteq E'$ with $\widetilde \sigma(F') = \tau' \widetilde \sigma(E')$ so that we may take $E' = E \cap \partial B(0,r_0)$ and find $F' \subseteq E \cap \partial B(0,r_0)$ so that $\sigma(F \cap \partial B(0,r_0)) = \tau'\sigma(E \cap \partial B(0,r_0))$. Setting $F = F' \cup (B(0,r_0) \cap E)$ we have that $\sigma(F) = (1/2)\sigma(E)$.
\end{proof}

\begin{lemma}\label{Koreyexplain.lem}
Let $\Gamma \subset \ree$ be a closed set with locally finite $\HH^n$ measure and set $\sigma = \HH^n|_\Gamma$. Suppose $x_0 \in \Gamma$ and $r_0 > 0$ are such that $\sigma(\Delta_0) > 0$, where $\Delta_0 = B(x_0,r_0) \cap \Gamma$. Suppose $k \in L^1_{loc}(d\sigma)$ with $k \ge 0$ and set $w = k \, d\sigma$. There exists $\epsilon_0$ and $c$ ,absolute constants, so that the following holds. If there exists $\epsilon > 0$ such that for every $F \subset \Delta_0$ with $\sigma(F)/\sigma(\Delta_0) = 1/2$ it holds that $w(F)/w(\Delta_0) \le 1/2 + \epsilon$ for some $\epsilon \in (0, \epsilon_0)$ then
\[\left( \fint_{\Delta(x_0,s)} k \, d\sigma \right)^2 \leq (1+c\epsilon) \left( \fint_{\Delta(x_0,s)}  k^{1/2} \, d\sigma \right).\]
\end{lemma}

\begin{proof}
We start the proof exactly as in the proof (d) to (c) in \cite[Theorem 10]{Korey}. 
Set $f: = \sqrt{ k}$ an $L_{loc}^2(d\sigma)$ function.
Set $m_{ k, \Delta_0}$ to be the {\bf median} of the function $ k$ on $\Delta_0$, that is
\[\sigma(\{x \in \Delta_0:  k > m_{ k, \Delta_0}\}),\sigma( \{x \in \Delta_0:  k < m_{ k, \Delta_0}\}) \le (1/2)\sigma(\Delta_0)\]
Set $E' = \{x \in \Delta_0:  k > m_{ k, \Delta_0}\}$ and $F' = \{x \in \Delta_0:  k < m_{ k, \Delta_0}\}$. Then by Lemma \ref{sigmadiffuse.lem} there exists $G \subseteq \Delta_0 \setminus (E' \cup F')$ such that 
$\sigma(E' \cup G) =  (1/2)\sigma(\Delta_0)$. Setting $E := E' \cup G$ and $F := \Delta_0 \setminus E$ we have that $\sigma(E) = \sigma(F) = (1/2)\sigma(\Delta_0)$,
\[ k(x) \le m_{ k, \Delta_0}, \quad \sigma-a.e. x \in F\]
and
\[ k(x) \ge m_{ k, \Delta_0}, \quad \sigma-a.e. x \in E.\]
By hypothesis 
\[\frac{ w(F)}{ w(\Delta_0)} = 1 - \frac{ w(E)}{ w(\Delta_0)} \ge \frac12 - \epsilon.\]
Then 
\begin{align*} \fint_{\Delta_0} f^2 \, d\sigma &= \frac{1}{\sigma(\Delta_0)} \int_{\Delta_0}  k \, d\sigma = \frac{  w(\Delta_0)}{\sigma(\Delta_0)}
\\ & \le (1 - 2\epsilon)^{-1} \frac{  w(F)}{\sigma(F)} = (1 - 2\epsilon)^{-1}\frac{1}{\sigma(F)} \int_F   k \, d\sigma
\\ & = (1 - 2\epsilon)^{-1}\frac{1}{\sigma(F)} \int_F f^2\, d\sigma \le (1 - 2\epsilon)^{-1}\frac{1}{\sigma(F)} \int_F f  \, d\sigma \sqrt{m_{ k, \Delta_0}},
\end{align*}
where we used $f \le \sqrt{m_{ k, \Delta_0}}$ $\sigma$-a.e. on $F$. On the other hand,
\begin{align*}
	m_{k,\Delta_0} \leq \fint_E k \, d\sigma = \frac{w(E)}{\sigma(E)} \leq \frac{\frac12 + \epsilon}{\frac12} \,\frac{w(\Delta_0)}{\sigma(\Delta_0)} = (1+2\epsilon) \fint_{\Delta_0} k\, d\sigma,
\end{align*}
where we used $\sigma(E)/\sigma(\Delta_0) = 1/2$ and the hypothesis of the lemma. Combining these two inequalities we have
\begin{align*}
\fint_{\Delta_0} f^2 \, d\sigma &\le (1 - 2\epsilon)^{-1} \left( \fint_F f \, d\sigma \right) \sqrt{m_{ k, \Delta_0}} 
\\ & \le (1 + 2\epsilon)^{1/2}(1 - 2\epsilon)^{-1} \left( \fint_{\Delta_0} f  \, d\sigma \right) \left(\fint_{\Delta_0}  k \, d\sigma \right)^{1/2}
\\ & = (1 + 2\epsilon)^{1/2}(1 - 2\epsilon)^{-1}\left(\fint_{\Delta_0} f   \, d\sigma \right) \left(\fint_{\Delta_0} f^2 \, d\sigma \right)^{1/2},
\end{align*}
where we used that the average over $F$ of $ k$ is less than the average over $\Delta_0$ of $ k$, by the definition of $F$ and $E = \Delta_0 \setminus F$. Thus, we have
\begin{equation}\label{optRH2forsqrt.eq}
\left(\fint_{\Delta_0} f^2 \, d\sigma\right)^{1/2} \le (1 + c\epsilon)  \fint_{\Delta_0} f   \, d\sigma,
\end{equation}
provided that $\epsilon_0$ is sufficiently small. 
\end{proof}

To show \ref{cond:d} implies \ref{cond:c}, we apply the previous lemma with $\epsilon = C'(\beta')^\mu$. We also remark here that in order to conclude that $\log k$ in $VMO$ from \ref{cond:c} one can use the methods in \cite{Korey} (armed with Lemma \ref{sigmadiffuse.lem}).

\section[Pullbacks and pushforwards for Lipschitz domains]{Pullbacks and pushforwards for Lipschitz domains}\label{pushpull.sect}
In this appendix, we recall the well-known applications of arguments concerning the ``flat" transformation of solutions to (second order divergence form) elliptic equations on Lipschitz domains to solutions of a transformed elliptic operator on the upper half space. This transformation is well adapted to the perturbations under consideration in this work as many have noted. We continue to work in $\ree$, identified as $\ree = \{(x,t) \in \rn \times \re\}$. For a Lipschitz function $\varphi : \rn \to \re$, with $\varphi(0) = 0$\footnote{We can always arrange for this by shifting all of the objects under consideration.} we set 
\[ \om_\varphi := \{(x,t) \in \ree: t > \varphi(x)\}. \] We will often just write $\om$ except when the dependence on $\varphi$ is important. We also define 
\[ \Gr(\varphi) := \{(x,t)\in \ree: t=\varphi(x) \}.\]

The following proposition follows by a simple change of variables.

\begin{proposition}\label{Poissonkernelofpullback.prop} Let $\varphi : \rn \to \re$, $\varphi(0) = 0$ be a Lipschitz function and $\Omega= \om_\varphi$ be as above. Let $L = - \div A \nabla$ be an elliptic operator with real coefficient matrix $A$. Let $\Phi(x,t)$ be the {\bf flattening map for $\varphi$}, $\Phi(x,t) := (x, t - \varphi(x))$, so that $\Phi(\om_\varphi) = \ree_+$ and $\Phi(\Gr(\varphi)) = \rn \times \{0\}$. Then $u : \om \to \re$ solves  the the differential equation
\[(D)_{L,\om} 
\begin{cases}
Lu = 0 \in \om \\
u = f \in C_c(\pom)
\end{cases}\]
if and only if $\tilde u: \ree_+ \to \re$ solves the differential equation
\[(D)_{\widetilde L, \ree_+}
\begin{cases}
\widetilde L \tilde u = 0 \in \ree_+ \\
\tilde u = \tilde f,
\end{cases}\]
where $\tilde f(x) = f(x,\varphi(x)) \in C_c(\rn)$, $\tilde u(X) = (u \circ \Phi^{-1})(X)$ and $\widetilde L = -\div \widetilde A \nabla$ with $\widetilde A = J_\Phi^{T}(A \circ \Phi^{-1})J_\Phi$. Here $J_\Phi$ is the Jacobian matrix of $\Phi$, given by 
\[
J_\Phi(X) = J_\Phi(x,t) =  \left[ \begin{array}{c|c}
I_{n \times n} & -\nabla \varphi(x)
\\ \hline
0 & 1
\end{array}\right],
\]
so that, in fact, $J_\Phi$ is a function of $x$. In particular, if $A$ is $t$-independent so is $\tilde A$.

Moreover, if $h_A^X = \frac{d\hm^X}{d\sigma}$ is the Poisson kernel for $L$ in $\om$ with pole at $X$ exists then
\[h_A^X(y,\varphi(y)) = \frac{1}{\sqrt{1 + |\nabla \varphi(y)|^2}} \,k_{\widetilde A}^{\Phi(X)}(y),\]
where $k_{\tilde A}^{\Phi(X)}(y)$ is the Poisson kernel for $\widetilde L$ in $\ree_+$ with pole at $\Phi(X)$.
\end{proposition}

To see the last fact, note that for $X \in \om$ and $f \in C_c(\pom)$
\[ u(X) = \int_{\pO} h_A^X(Z) f(Z) d\sigma(Z) = \int_{\RR^n} h_A^X(y,\varphi(y)) f(y, \varphi(y)) \sqrt{1+ |\nabla \varphi(y)|^2} dy; \]
and on the other hand,
\[ u(X) = \tilde u(\Phi(X)) = \int_{\rn} k_{\widetilde A}^{\Phi(X)}(y) \tilde f(y) \, dy = \int_{\RR^n} k_{\widetilde A}^{\Phi(X)} f(y, \varphi(y)) \, dy. \]

From the above we can see that when $\|\nabla \varphi\|_{\infty} \ll 1$ the Poisson kernels $h_A$ and $k_{\widetilde A}$ are very similar. We would like to say that perturbed operators (in the analogous sense of that in Proposition \ref{perturbconstop.prop}) remain so under pullback. This amounts to look at how $J_{\Phi}$ acts on vectors and matrices. The following lemma can be directly verified via computation. We provide some brief details.

\begin{lemma}\label{JACprop.lem}
 Let $\varphi : \rn \to \re$, $\varphi(0) = 0$ be a Lipschitz function with Lipschitz constant $\gamma:= \|\nabla \varphi\|_{L^\infty}$. 
 For almost every $x \in \rn$, $J_\Phi = J_\Phi(x)$ has the following properties.
\begin{itemize}
\item[(a)] For any $\xi \in \ree$,
\[ \begin{array}{ll}
 |J_\Phi \xi - \xi|_2 \le \sqrt{n} \gamma|\xi|_2, & |J_\Phi \xi - \xi|_\infty \le \gamma|\xi|_\infty,
 \\ |J_\Phi^T \xi - \xi|_2 \le \gamma|\xi|_2, & |J_\Phi^T \xi - \xi|_\infty \le \sqrt{n} \gamma|\xi|_\infty.
\end{array} \]
\item[(b)] For any $(n+1)\times (n+1)$ matrix $A$,
\[|J^T_\Phi A J_\Phi - A|_\infty \le (\sqrt{n} \gamma + n\gamma^2) |A|_\infty.\]
\item[(c)] For any $\xi \in \ree$,
\[|J_\Phi \xi|_2 \ge \min\left\{ \frac12, (1+4n\gamma^2)^{-\frac12} \right\}|\xi|_2.\]
\end{itemize} 
\end{lemma} 
Here $|\cdot|_2$ and $|\cdot|_\infty$ are the $\ell^2$ and $\ell^\infty$ norms, respectively.
\begin{proof} For $\xi \in \ree$, $\xi = (\xi_1, \dots, \xi_n, \xi_{n+1})$ we write $\xi_\parallel = (\xi_1, \dots, \xi_n)$ and $\xi_\perp = \xi_{n+1}$. To prove the assertions of $(a)$ concerning $J_\Phi$, we write
$J_\Phi \xi = (\xi_\parallel - \xi_\perp\nabla \varphi(x), \xi_\perp)$, so that $J_\Phi \xi - \xi = (-\xi_\perp\nabla \varphi(x), 0)$. Similarly, to treat the estimates involving $J_\Phi^T$ we write $J_\Phi^T \xi = (\xi_\parallel, - \nabla \varphi(x) \cdot \xi_\parallel + \xi_\perp)$, so that $J_\Phi^T \xi - \xi = (0, -\nabla \varphi(x)\cdot \xi_\parallel)$. Property $(b)$ follows from property $(a)$ after writing 
\[J^T_\Phi A J_\Phi - A = (J^T_\Phi A J_\Phi   - J^T_\Phi A) + (J^T_\Phi A  - A)\]
so that
\begin{align*}
|J^T_\Phi A J_\Phi - A|_\infty &\le  |J^T_\Phi A J_\Phi   - J^T_\Phi A|_\infty + |J^T_\Phi A  - A|_\infty
\\& \le \sqrt{n}\gamma|AJ_\Phi - A|_\infty + \sqrt{n} \gamma|A|_\infty
\\& \le  (\sqrt{n} \gamma + n\gamma^2 ) |A|_\infty.
\end{align*}
Finally, to see $(c)$, we again write $J_\Phi \xi = (\xi_\parallel - \xi_\perp\nabla \varphi(x), \xi_\perp)$ so that
$|J_\Phi \xi|_2 \ge \max(|\xi_\parallel - \xi_\perp\nabla \varphi|_2, |\xi_\perp|_2)$. If $|\xi_\perp\nabla \varphi|_2 < \tfrac{1}{2}|\xi_\parallel|_2$ then the estimate in $(c)$ follows readily. If not, then $|\xi_\perp\nabla \varphi|_2 \ge \tfrac{1}{2}|\xi_\parallel|_2$ and hence by the Lipschitz condition $|\xi_\perp|_2 \ge \frac{1}{2\sqrt{n} \gamma}|\xi_\parallel|_2$ so that
\[ |J_\Phi \xi|_2 \ge |\xi_{\perp}|_2 \geq \frac{1}{\sqrt{1+4n\gamma^2}} |\xi|_2, \] from which the estimate in $(c)$ again follows readily.
\end{proof}

Next, we define Whitney and Carleson-type regions, which are well-suited for our purposes. For $\varphi$ and $\Omega_\varphi$ as above, and $X = (x,\varphi(x) + t) \in \om_\varphi$ we define the {\bf $\varphi$-adapted Whitney region} 
\[W_\varphi(X): = \left\{(y,s): |y-x|<t,\, \varphi(y) + t/2 < s < \varphi(y) + 3t/2 \right\}.\]
Note that $\Phi(W_\varphi) = W(x,t)$, the Whitney region in $\RR^{n+1}_+$. Next, for a  cube $Q \subset \rn$ we define the {\bf $\varphi$-adapted Carleson box}
\begin{equation}\label{def:Cbox}
	R_{Q, \varphi} = \{(y,s): y \in Q, \, \varphi(y)< s < \varphi(y) + \ell(Q)\}.
\end{equation} 
Finally, for a measurable $(n+1) \times (n+1)$ matrix-valued function $B$, we define the (FKP) {\bf $\varphi$-adapted Carleson norm} of $B$ as
\begin{equation}\label{FKPphiadapt.def}
\begin{split}
\|B\|_{\C_\varphi} &:= \sup_{Q \subset \rn} \left( \frac{1}{|Q|} \iint_{R_{Q,\varphi}} \| B\|_{L^\infty(W_\varphi(y,s))}^2 \, \frac{dy\,ds}{s - \varphi(y)} \right)^{1/2} 
\\& =  \sup_{Q \subset \rn} \left( \frac{1}{|Q|} \iint_{R_{Q}} \| B'\|_{L^\infty(W(x,t))}^2 \, \frac{dx\,dt}{t} \right)^{1/2} =\|B'\|_{\C},
\end{split}
\end{equation}
where $B' = B\circ \Phi^{-1}$ and we used the flattening change of variables in the second line.

The following lemma is a direct consequence of the definitions above and Lemma \ref{JACprop.lem}.

\begin{lemma}\label{carlpull1.lem} Let $\varphi: \rn \to \re$ be a Lipschitz function with Lipschitz constant $\gamma$ and $\varphi(0) = 0$ and suppose that $B$ is a $(n +1) \times (n+1)$ matrix-valued function on $\Omega_\varphi$ with $\|B\|_{\C_\varphi} < \infty$. Then the matrix $\widetilde B = J_\Phi^T (B \circ \Phi^{-1})J_\Phi$ satisfies
\[\|\widetilde{B}\|_{\C} \le (1 + \sqrt{n}\gamma + n\gamma^2)\|B\|_{\C_\varphi}.\]
\end{lemma}
\begin{proof}
Recall that $J_\Phi (x,t) = J_\Phi(x)$ so that \eqref{FKPphiadapt.def} with $B' = \widetilde{B}$ (which means $ J_\Phi^T B J_\Phi$ is in place of $B$)
\[\|\widetilde{B}\|_{\C} \le \| J_\Phi^T B J_\Phi\|_{\C_\varphi} \le (1 + \sqrt{n}\gamma + n\gamma^2)\|B\|_{\C_\varphi},\]
where we used Lemma \ref{JACprop.lem}(b) in the second inequality.
\end{proof}

Combining the previous Lemma with Lemma \ref{JACprop.lem}, one easily obtains the following.

\begin{proposition}\label{genopperturbPB.prop} Let $\Lambda \ge 1$. Let $\varphi : \rn \to \re$ be a Lipschitz function with Lipschitz constant $\gamma \le \frac{1}{50n}$ and $\varphi(0)=0$, and $A_0$ a real, constant $\Lambda$-elliptic $(n+1) \times (n+1)$ matrix. Suppose $A(X)$ is a real, $\Lambda$-elliptic, matrix-valued function on $\ree$ with the decomposition
\[A(x,t) = A_1(x) + B(x,t)\]
satisfying
\[\| A_1 - A_0 \|_{L^\infty(\rn)} + \|B\|_{\C_\varphi} < \kappa\]
for some $\kappa \ge 0$. Then $\widetilde{A} = J_\Phi^T(A\circ \Phi^{-1})J_\Phi$ is a real, $8\Lambda$-elliptic matrix with the decomposition
\[\widetilde{A}(x,t) = \widetilde{A}_1(x) + \widetilde{B}(x,t),\]
satsifying
\[\| \widetilde{A}_1 - A'_0 \|_{L^\infty(\rn)} + \|\widetilde{B}\|_{\C} < 2\kappa + 4\sqrt{n}\gamma\Lambda,\]
where $\widetilde{A}_1 := J^T_\Phi(A_1 \circ \Phi^{-1})J_\Phi = J^T_\Phi(A_1)J_\Phi$, $\widetilde{B} := J^T_\Phi(B \circ \Phi^{-1})J_\Phi$ and\footnote{Note that ${A}'_0$ is a constant matrix and one can show that $\widetilde{A}_0$ is $8\Lambda$ elliptic in the same manner as $\widetilde{A}$.}
$A'_0 := J^T_\Phi(0)A_0J_\Phi(0)$.
\end{proposition}
\begin{proof}
The form of the decomposition $\widetilde{A}(x,t) = \widetilde{A}_1(x) + \widetilde{B}(x,t)$ is immediate from the form of $A$. In particular, notice that $A_1 \circ \Phi^{-1} = A_1$ since $A_1 = A_1(x)$. The $8\Lambda$-ellipticity of $\widetilde{A}$ is a consequence of Lemma \ref{JACprop.lem}(b) and (c). Indeed, the boundedness of $\tilde{A}$ follows from Lemma \ref{JACprop.lem}(b), here one recalls that $\gamma < 1/(50n)$, so that $\sqrt{n}\gamma + n\gamma^2 \leq 1$ (we will use this several times). To see the lower ellipticity bound, we use the $\Lambda$-ellipticity of $A$ and Lemma \ref{JACprop.lem}(c) to obtain
\[\langle \widetilde{A}\xi , \xi \rangle \ge \Lambda^{-1}|J_\Phi \xi|_2^2 \ge \frac{1}{4} \Lambda^{-1}|\xi|_2^2,\]
for almost every $x$ and all $\xi \in \ree$.

To obtain the desired estimate for $\| \widetilde{A}_1 - A'_0 \|_{L^\infty(\rn)}$, we write
\begin{equation}\label{A1estdag.eq}
\widetilde{A}_1 - A'_0 = (\widetilde{A}_1 - \widetilde{A}_0) + (\widetilde{A}_0 - A_0'),
\end{equation}
where $\widetilde{A}_0$ is the (variable) matrix-valued function $J^T_\Phi(x) A_0 J_\Phi(x)$. 
Using Lemma \ref{JACprop.lem}(b), and the triangle inequality
\begin{equation}\label{A1estdagp.eq}
\| \widetilde{A}_1 - \widetilde{A}_0\|_{L^\infty} \le (1 + \sqrt{n}\gamma + n\gamma^2) \| A_1 - A_0\|_{L^\infty} \leq 2\| A_1 - A_0\|_{L^\infty}
\end{equation}
To handle the second term,
we again use Lemma \ref{JACprop.lem}(b) to obtain
\begin{equation}\label{A1estdagpp.eq}
\begin{split}
\| \widetilde{A}_0 - A_0'\|_{L^\infty} &\le \|\widetilde{A}_0 - A_0\|_{L^\infty} + \| A_0 - A'_0\|_{L^\infty}
\\ &\le 2(\sqrt{n}\gamma + n\gamma^2)\|A_0\|_{L^\infty} \le 4\sqrt{n}\gamma\Lambda.
\end{split}
\end{equation}
Combining \eqref{A1estdag.eq}, \eqref{A1estdagp.eq} and \eqref{A1estdagpp.eq} yields the desirable estimate 
\[ \| \widetilde{A}_1 - A'_0 \|_{L^\infty} \le 2 \| A_1 - A_0\|_{L^\infty} + 4\sqrt{n}\gamma\Lambda.\]
Since Lemma \ref{carlpull1.lem} gives 
\[\|\widetilde{B}\|_{\C} \le (1 + \sqrt{n}\gamma + n\gamma^2)\|B\|_{\C_\varphi}  \le 2 \|B\|_{\C_\varphi} \]
we obtain 
\[\| \widetilde{A}_1 - A'_0 \|_{L^\infty(\rn)} + \|\widetilde{B}\|_{\C} < 2 \kappa + 4\sqrt{n}\gamma\Lambda,\]
as desired.
\end{proof}

\bibliographystyle{alpha}
\bibdata{references}
\bibliography{references}

\newcommand{\etalchar}[1]{$^{#1}$}
\begin{thebibliography}{HMM15b}

\bibitem[AA11]{AA}
Pascal Auscher and Andreas Axelsson.
\newblock Weighted maximal regularity estimates and solvability of non-smooth
  elliptic systems {I}.
\newblock {\em Invent. Math.}, 184(1):47--115, 2011.

\bibitem[AAA{\etalchar{+}}11]{AAAHK}
M.~Angeles Alfonseca, Pascal Auscher, Andreas Axelsson, Steve Hofmann, and
  Seick Kim.
\newblock Analyticity of layer potentials and {$L^2$} solvability of boundary
  value problems for divergence form elliptic equations with complex
  {$L^\infty$} coefficients.
\newblock {\em Adv. Math.}, 226(5):4533--4606, 2011.

\bibitem[AAH08]{AAH}
Pascal Auscher, Andreas Axelsson, and Steve Hofmann.
\newblock Functional calculus of {D}irac operators and complex perturbations of
  {N}eumann and {D}irichlet problems.
\newblock {\em J. Funct. Anal.}, 255(2):374--448, 2008.

\bibitem[AAM10]{AAMc}
Pascal Auscher, Andreas Axelsson, and Alan McIntosh.
\newblock Solvability of elliptic systems with square integrable boundary data.
\newblock {\em Ark. Mat.}, 48(2):253--287, 2010.

\bibitem[AHL{\etalchar{+}}02]{AHLMcT}
Pascal Auscher, Steve Hofmann, Michael Lacey, Alan McIntosh, and Ph.
  Tchamitchian.
\newblock The solution of the {K}ato square root problem for second order
  elliptic operators on {${\Bbb R}^n$}.
\newblock {\em Ann. of Math. (2)}, 156(2):633--654, 2002.

\bibitem[AMT17]{AMT-onep}
Jonas Azzam, Mihalis Mourgoglou, and Xavier Tolsa.
\newblock The one-phase problem for harmonic measure in two-sided {NTA}
  domains.
\newblock {\em Anal. PDE}, 10(3):559--588, 2017.

\bibitem[Axe10]{Axel10}
Andreas Axelsson.
\newblock Non-unique solutions to boundary value problems for non-symmetric
  divergence form equations.
\newblock {\em Trans. Amer. Math. Soc.}, 362(2):661--672, 2010.

\bibitem[Bou87]{Bourgain}
J.~Bourgain.
\newblock On the {H}ausdorff dimension of harmonic measure in higher dimension.
\newblock {\em Invent. Math.}, 87(3):477--483, 1987.

\bibitem[CF74]{CF-weights}
R.~R. Coifman and C.~Fefferman.
\newblock Weighted norm inequalities for maximal functions and singular
  integrals.
\newblock {\em Studia Math.}, 51:241--250, 1974.

\bibitem[CFMS81]{CFMS}
L.~Caffarelli, E.~Fabes, S.~Mortola, and S.~Salsa.
\newblock Boundary behavior of nonnegative solutions of elliptic operators in
  divergence form.
\newblock {\em Indiana Univ. Math. J.}, 30(4):621--640, 1981.

\bibitem[CJ87]{ChristJourne}
Michael Christ and Jean-Lin Journ\'{e}.
\newblock Polynomial growth estimates for multilinear singular integral
  operators.
\newblock {\em Acta Math.}, 159(1-2):51--80, 1987.

\bibitem[CM86]{CoifmanMeyer}
R.~R. Coifman and Yves Meyer.
\newblock Nonlinear harmonic analysis, operator theory and {P}.{D}.{E}.
\newblock In {\em Beijing lectures in harmonic analysis ({B}eijing, 1984)},
  volume 112 of {\em Ann. of Math. Stud.}, pages 3--45. Princeton Univ. Press,
  Princeton, NJ, 1986.

\bibitem[Dah77]{DahlbergRH2}
Bj\"{o}rn E.~J. Dahlberg.
\newblock Estimates of harmonic measure.
\newblock {\em Arch. Rational Mech. Anal.}, 65(3):275--288, 1977.

\bibitem[Dah86]{Dahlbergperturb}
Bj\"{o}rn E.~J. Dahlberg.
\newblock On the absolute continuity of elliptic measures.
\newblock {\em Amer. J. Math.}, 108(5):1119--1138, 1986.

\bibitem[DG57]{DeG}
Ennio De~Giorgi.
\newblock Sulla differenziabilit\`a e l'analiticit\`a delle estremali degli
  integrali multipli regolari.
\newblock {\em Mem. Accad. Sci. Torino. Cl. Sci. Fis. Mat. Nat. (3)}, 3:25--43,
  1957.

\bibitem[DJ84]{DavidJourne}
Guy David and Jean-Lin Journ\'{e}.
\newblock A boundedness criterion for generalized {C}alder\'{o}n-{Z}ygmund
  operators.
\newblock {\em Ann. of Math. (2)}, 120(2):371--397, 1984.

\bibitem[DJ90]{DJ}
G.~David and D.~Jerison.
\newblock Lipschitz approximation to hypersurfaces, harmonic measure, and
  singular integrals.
\newblock {\em Indiana Univ. Math. J.}, 39(3):831--845, 1990.

\bibitem[DJS85]{DavidJourneSemmes}
G.~David, J.-L. Journ\'{e}, and S.~Semmes.
\newblock Op\'{e}rateurs de {C}alder\'{o}n-{Z}ygmund, fonctions
  para-accr\'{e}tives et interpolation.
\newblock {\em Rev. Mat. Iberoamericana}, 1(4):1--56, 1985.

\bibitem[Esc96]{Esc}
Luis Escauriaza.
\newblock The {$L^p$} {D}irichlet problem for small perturbations of the
  {L}aplacian.
\newblock {\em Israel J. Math.}, 94:353--366, 1996.

\bibitem[FJK84]{FJK}
Eugene~B. Fabes, David~S. Jerison, and Carlos~E. Kenig.
\newblock Necessary and sufficient conditions for absolute continuity of
  elliptic-harmonic measure.
\newblock {\em Ann. of Math. (2)}, 119(1):121--141, 1984.

\bibitem[FKP91]{FKP}
R.~A. Fefferman, C.~E. Kenig, and J.~Pipher.
\newblock The theory of weights and the {D}irichlet problem for elliptic
  equations.
\newblock {\em Ann. of Math. (2)}, 134(1):65--124, 1991.

\bibitem[Geh73]{Gehring}
F.~W. Gehring.
\newblock The {$L^{p}$}-integrability of the partial derivatives of a
  quasiconformal mapping.
\newblock {\em Acta Math.}, 130:265--277, 1973.

\bibitem[HKM93]{HKM-Book}
Juha Heinonen, Tero Kilpel\"{a}inen, and Olli Martio.
\newblock {\em Nonlinear potential theory of degenerate elliptic equations}.
\newblock Oxford Mathematical Monographs. The Clarendon Press, Oxford
  University Press, New York, 1993.
\newblock Oxford Science Publications.

\bibitem[HMM15a]{HMayMour}
Steve Hofmann, Svitlana Mayboroda, and Mihalis Mourgoglou.
\newblock Layer potentials and boundary value problems for elliptic equations
  with complex {$L^\infty$} coefficients satisfying the small {C}arleson
  measure norm condition.
\newblock {\em Adv. Math.}, 270:480--564, 2015.

\bibitem[HMM15b]{HMiMo}
Steve Hofmann, Marius Mitrea, and Andrew~J. Morris.
\newblock The method of layer potentials in {$L^p$} and endpoint spaces for
  elliptic operators with {$L^\infty$} coefficients.
\newblock {\em Proc. Lond. Math. Soc. (3)}, 111(3):681--716, 2015.

\bibitem[HMT17]{HMT}
Steve Hofmann, Jos\'{e}~Mar\'{i}a Martell, and Tatiana Toro.
\newblock $a_\infty$ implies nta for a class of variable coefficient elliptic
  operators.
\newblock {\em J. Differ. Equations}, 263(10):6147--6188, 2017.

\bibitem[Iwa98]{Iwan-Geh}
Tadeusz Iwaniec.
\newblock The {G}ehring lemma.
\newblock In {\em Quasiconformal mappings and analysis ({A}nn {A}rbor, {MI},
  1995)}, pages 181--204. Springer, New York, 1998.

\bibitem[JK82a]{JK-NTA}
David~S. Jerison and Carlos~E. Kenig.
\newblock Boundary behavior of harmonic functions in nontangentially accessible
  domains.
\newblock {\em Adv. in Math.}, 46(1):80--147, 1982.

\bibitem[JK82b]{JK-VMO}
David~S. Jerison and Carlos~E. Kenig.
\newblock The logarithm of the {P}oisson kernel of a {$C^{1}$} domain has
  vanishing mean oscillation.
\newblock {\em Trans. Amer. Math. Soc.}, 273(2):781--794, 1982.

\bibitem[KKPT00]{KKoPT}
C.~Kenig, H.~Koch, J.~Pipher, and T.~Toro.
\newblock A new approach to absolute continuity of elliptic measure, with
  applications to non-symmetric equations.
\newblock {\em Adv. Math.}, 153(2):231--298, 2000.

\bibitem[Kor98]{Korey}
Michael~Brian Korey.
\newblock Ideal weights: asymptotically optimal versions of doubling, absolute
  continuity, and bounded mean oscillation.
\newblock {\em J. Fourier Anal. Appl.}, 4(4-5):491--519, 1998.

\bibitem[KT97]{KT-Duke}
Carlos~E. Kenig and Tatiana Toro.
\newblock Harmonic measure on locally flat domains.
\newblock {\em Duke Math. J.}, 87(3):509--551, 1997.

\bibitem[KT99]{KT-Annals}
Carlos~E. Kenig and Tatiana Toro.
\newblock Free boundary regularity for harmonic measures and {P}oisson kernels.
\newblock {\em Ann. of Math. (2)}, 150(2):369--454, 1999.

\bibitem[KT03]{KT-Lens}
Carlos~E. Kenig and Tatiana Toro.
\newblock Poisson kernel characterization of {R}eifenberg flat chord arc
  domains.
\newblock {\em Ann. Sci. \'{E}cole Norm. Sup. (4)}, 36(3):323--401, 2003.

\bibitem[MM85]{McIntoshMeyer}
Alan McIntosh and Yves Meyer.
\newblock Alg\`ebres d'op\'{e}rateurs d\'{e}finis par des int\'{e}grales
  singuli\`eres.
\newblock {\em C. R. Acad. Sci. Paris S\'{e}r. I Math.}, 301(8):395--397, 1985.

\bibitem[MMM19]{MarinMartellMitrea}
Juan~Jos{\'e} Mar{\'\i}n, Jos{\'e}Mar{\'\i}a Martell, and Marius Mitrea.
\newblock The generalized h{\"o}lder and morrey-campanato dirichlet problems
  for elliptic systems in the upper half-space.
\newblock {\em Potential Analysis}, 2019.

\bibitem[Mos61]{Moser}
J\"{u}rgen Moser.
\newblock On {H}arnack's theorem for elliptic differential equations.
\newblock {\em Comm. Pure Appl. Math.}, 14:577--591, 1961.

\bibitem[MPT14]{MPT}
Emmanouil Milakis, Jill Pipher, and Tatiana Toro.
\newblock Perturbations of elliptic operators in chord arc domains.
\newblock In {\em Harmonic analysis and partial differential equations}, volume
  612 of {\em Contemp. Math.}, pages 143--161. Amer. Math. Soc., Providence,
  RI, 2014.

\bibitem[MT10]{MT}
Emmanouil Milakis and Tatiana Toro.
\newblock Divergence form operators in {R}eifenberg flat domains.
\newblock {\em Math. Z.}, 264(1):15--41, 2010.

\bibitem[MZ97]{MalyZiemer}
Jan Mal\'{y} and William~P. Ziemer.
\newblock {\em Fine regularity of solutions of elliptic partial differential
  equations}, volume~51 of {\em Mathematical Surveys and Monographs}.
\newblock American Mathematical Society, Providence, RI, 1997.

\bibitem[Nas58]{Nash}
J.~Nash.
\newblock Continuity of solutions of parabolic and elliptic equations.
\newblock {\em Amer. J. Math.}, 80:931--954, 1958.

\bibitem[Rei60]{Rei}
E.~R. Reifenberg.
\newblock Solution of the {P}lateau {P}roblem for {$m$}-dimensional surfaces of
  varying topological type.
\newblock {\em Acta Math.}, 104:1--92, 1960.

\bibitem[Sem90]{Semmes-DJwNTA}
Stephen Semmes.
\newblock Analysis vs. geometry on a class of rectifiable hypersurfaces in
  {${\bf R}^n$}.
\newblock {\em Indiana Univ. Math. J.}, 39(4):1005--1035, 1990.

\bibitem[Sem91]{Semmes-SmallCAD1}
Stephen Semmes.
\newblock Chord-arc surfaces with small constant. {I}.
\newblock {\em Adv. Math.}, 85(2):198--223, 1991.

\bibitem[Ste70]{Stein-SIOs}
Elias~M. Stein.
\newblock {\em Singular integrals and differentiability properties of
  functions}.
\newblock Princeton Mathematical Series, No. 30. Princeton University Press,
  Princeton, N.J., 1970.

\bibitem[Ste93]{Stein-HA}
Elias~M. Stein.
\newblock {\em Harmonic analysis: real-variable methods, orthogonality, and
  oscillatory integrals}, volume~43 of {\em Princeton Mathematical Series}.
\newblock Princeton University Press, Princeton, NJ, 1993.
\newblock With the assistance of Timothy S. Murphy, Monographs in Harmonic
  Analysis, III.

\end{thebibliography}

\end{document}